\documentclass[11pt]{article}
\usepackage{latexsym}
\usepackage{amsmath, amsthm, amssymb}
\usepackage{amscd}
\usepackage[english]{babel}
\usepackage{graphicx}
\usepackage[arc,all]{xy}
\usepackage{tikz-cd}
\usepackage{braids}
\usepackage[colorlinks]{hyperref}
\usepackage{pict2e}
\usepackage{ifpdf}
\usepackage{xcolor}
\usepackage{float}
\usepackage{wrapfig}
\hypersetup{
    colorlinks=true,
    linkcolor=blue,
    filecolor=magenta,      
    urlcolor=cyan,
    pdfpagemode=FullScreen,
    }

\newtheorem{thm}{Theorem}[section]
\newtheorem{prop}{Proposition}[section]

\newtheorem{lem}[thm]{Lemma}

\theoremstyle{definition} 
\newtheorem{defn}{Definition}[section]

\theoremstyle{remark} 
\newtheorem*{rem}{Remark}

\theoremstyle{remark} 
\newtheorem*{exmp}{Example}

\begin{document}

\title{\textbf{\large{On the geometry of two state models for the colored Jones polynomial}}}

\maketitle

\makeatletter 
\makeatother

\centerline{\author{Uwe Kaiser, \textit{Boise State University}, ukaiser@boisestate.edu}}

\centerline{\author{Rama Mishra, \textit{IISER Pune},r.mishra@iiserpune.ac}}

\begin{abstract}
Using the flow property of the R-matrix defining the colored Jones polynomial, we establish a natural bijection between the set of states on the \textit{part arc-graph} of a link projection and the set of states on a corresponding bichromatic digraph, called \textit{arc-graph}, as defined by Garoufalidis and Loebl \cite{GL}. We use this to give a new and essentially elementary proof for the knot state-sum formula in \cite{GL}.  We will show that the state-sum contributions of states on the part arc-graph defined by the universal $R$-matrix of $U_q(\textrm{sl}(2,\mathbb{C}))$ correspond, under our bijection of sets of states, to the contributions in \cite{GL}. This will show that the two state models are in fact not essentially distinct. Our approach will also extend the formula of Garoufalidis and Loebl to links. This requires some additional non-trivial observations concerning the geometry of states on the part arc-graphs. We will discuss in detail the computation of the arc-graph state-sum, in particular for $3$-braid closures.
\end{abstract}

\vskip 0.2in

\section{Introduction} 

This article is the result of a long-going effort to understand a state-sum formula of Garoufalidis and Loebl for the colored Jones polynomial of a knot. This formula was proven in \cite{GL} from the state-sum formula for the classical Jones polynomial and a cabling formula, which had been observed in \cite{Li}. From the proof given in \cite{GL} it is not obvious to us that the resulting state model is the \textit{same} as the one defined directly from the $R$-matrix of $U_q(\textrm{sl}(2,\mathbb{C}))$-representations, i.\ e.\ from the corresponding ribbon functor \cite{RT},\cite{KM}, \cite{G} (Definition 2.2 will explain what is meant by \textit{same} here).

Let $K$ be a framed oriented link in $\mathbb{R}^3$ with projection $\mathcal{K}$ (blackboard framing). 
Let $P\mathcal{K}$ denote the corresponding \textit{part arc-graph}. This is a $4$-valent planar digraph with vertices labeled by the crossing signs of the link projection.

Quantum link invariants, defined for a quantum group of a simple Lie-algebra and a finite dimensional representation, can be calculated as state-sums over $P\mathcal{K}$, \cite{J}, \cite{RT} and \cite{KiM}. In the case of the quantum group corresponding to $\textrm{sl}(2,\mathbb{C})$ and the irreducible $(n+1)$-dimensional representation, $n\geq 1$, Garoufalidis and Loebl \cite{GL} define a bichromatic digraph $G\mathcal{K}$ and calculate a state-sum for the colored Jones polynomial as a state-sum on this graph. In \cite{GL} the state-sum formula is derived from the formula for the classical Jones-polynomial ($n=1$) and cabling. This is based on the representation theoretic properties. We will show that the two state models are isomorphic to each other in a well-defined sense, even though the state models are defined on different graphs. 

Garoufalidis and Loebl also use their formula to define a non-commutative (this means that the sum is calculated using non-commutative algebra, not that the actual sum is non-commutative in some way) state-sum for the colored Jones polynomial \cite{GL}. Independently Huynh and Le \cite{HL}, see also \cite{A}, defined a non-commutative formula directly from the $R$-matrix defined by Rozansky \cite{R}. There is recent work on the noncommutative approach, see \cite{B} and \cite{H}. 
The relation between those two formulas is briefly mentioned in \cite{HL}, Remark 0.1, but a detailed discussion of the relation of this is still missing in the literature. The results of our article indicate that the two noncommutative state models are also equivalent if there is a natural state-sum defined for the noncommutative formula of \cite{GL}.
The states of \cite{HL} are described in \cite{A} in terms of walks on braids, with state-sum contributions defined by the actions of sequences of operators on a certain ring of polynomials defined from the walks. Algebraically the walks result from computing the expansion of a quantum determinant of a matrix over a ring of operators defined through an analogy to the Burau representation, see \cite{A} for more details. Here walks are defined by sequences of overcrossings, undercrossings and jump-down transitions. 

In the literature there are often restrictions of formulas for the colored Jones polynomial to the case of \textit{knots}. This often seems in order to avoid a discussion of which states are \textit{contributing} (see Definition 2.1). We think that this is in fact an interesting question, in particular when it comes to a discussion of computational complexity. 
We will show that the formula \cite{GL} based on the arc graph of knots extends to links. 

In sections 2 and 3 we define state models and how to compare them.  We define the sets of decorated graphs we will be using. In section 4 we define the bijections of states, which is used in the following sections to show that the usual state model defined from the ribbon functor ($R$-matrix) is equivalent to the state model in \cite{GL}. In section 5 we review the ribbon function definition, and in section 6 we give a direct proof of the state-sum formula \cite{GL}, which also shows the equivalence of the models and extends the formula to link projections. In section 7 we show how the states in both cases are determined by potentials, a notion introduced in section 4, and we show the contributing states are parametrized. Our approach shows that the state-sums are determined without using any order of the vertex set of the part arc-graph or arc-graph. In section 8 we we discuss the computation when we choose the cyclic order of vertices, defined by running along components from chosen base points. In section 9 we explicitly give the evaluation of the colored Jones polynomial for all weaving links on three strings from the \cite{GL}. In section 10 we show how the state-sum of 
is computed algorithmically from the braid word for all $-$-strand braid closures. We give explicit examples of state-sum computations throughout the article. 

\vskip 0.1in

\section{State models}

We will consider \textit{vertex models} as defined in \cite{J}, see also \cite{T}.

\begin{defn} Let $\mathfrak{G}$ be a set of (possibly in some way decorated) finite graphs and let $\mathcal{R}$ be an integral domain (usually $\mathcal{R}=\mathbb{Z}[t^{\pm \frac{1}{D}}]$ for some natural number $D$).
A  \textit{state model} on $\mathfrak{G}$ assigns to each graph $G=(V,E)\in \mathfrak{G}$ ($V$ is the vertex set, $E$ is the edge set) a \textit{state-sum}
\begin{equation*}
Z(G)=\sum_{f\in \textrm{St}(G)}Z(f)=\sum_{f\in \textrm{st}(G)}Z(f)\in \mathcal{R}
\end{equation*}
with $Z(f)=\Delta \cdot \beta (f)\in \mathcal{R}$. The factor $\Delta \neq 0$ is a \textit{global} contribution defined from the \textit{geometry} and possible decorations of the graph but does not depend on the coloring. 
The set $\textrm{St}(G)\subset \{f: E\rightarrow \mathbb{Z}\}$ ($\mathbb{Z}$ could be replaced by any countable set) is a set of \textit{edge-colorings}, the states of the state model. The factors $\beta (f):=\prod_{v\in V}\beta_v(f)\in \mathcal{R}$, are defined \textit{locally}. This means that $\beta_v(f)$ depends only on edge-colorings of edges incident with $v$ or with adjacent vertices. The set of \textit{contributing states} or \textit{admissible states} (see \cite{T})
$\textrm{st}(G):=\{f\in \textrm{St}|\beta (f)\neq 0\}\subset \textrm{St}(G)$ is always finite so that $Z(G)$ is defined. The condition $f\in \textrm{st}(G)$ is determined by local conditions on $f$, conditions on the colorings at vertices. Note that $\beta (f)\neq 0 \Longleftrightarrow \beta_v(f)\neq 0 \ \textrm{for all}\  v$. 
\end{defn}

\begin{rem} (a) Decoration on a graph includes direction, signs or numbers attached to vertices and/or edges. Those decorations are not to be confused with the edge-colorings.

\noindent (b) The above definition focuses on vertex contributions and is adapted to state sums on braid projections. In the link projection case, additionally colorings at extrema are included in the definition of $\beta (f)$. This is necessary so that the state sums define isotopy invariants. \end{rem}

\begin{defn} \label{defn: equivalency}
Let $\mathfrak{G}_1,\mathfrak{G}_2$ be two sets of finite decorated graphs and corresponding state models 
$Z_1,Z_2$. Let
$\Phi: \mathfrak{G}_1\rightarrow \mathfrak{G}_2$. Suppose that for each $G\in \mathfrak{G}_1$ there  
is a surjective map of corresponding state sets $\phi_G: \textrm{St}_1(G)\rightarrow \textrm{St}_2(\Phi (G))$ such that for each $f\in \textrm{St}_2(\Phi(G))$ we have
$$Z_2(f)=\sum_{g\in \textrm{st}_1(G): \phi _G(g)=f}Z_1(g)$$
Then we will say that the state model $Z_1$ \textit{dominates} the state model $Z_2$ \textit{over} $(\Phi ,\phi =(\phi_G)) $. If all maps $\phi_G$ are all bijections (so in particular restrict to bijections $\textrm{st}_1(G)\rightarrow \textrm{st}_2(\Phi (G))$
then we call the two state models \textit{equivalent over} $\Phi $.
 \end{defn}

\begin{rem} Equivalence of the state-sum models can be stronger than equality of the state-sums in $\mathcal{R}$. Note that it follows that the sets of admissible states for a graph $G$ and its image $\Phi (G)$ correspond to each other under the bijection $\phi _G$. \end{rem}

\begin{defn} \label{defn:edge coloring} An edge-coloring on a digraph $G$ is a $\mathbb{Z}$-\textit{flow} if 
$$\sum_{e\ \textrm{ingoing}\ v}f(e)=\sum_{e\ \textrm{outgoing}\ v}f(e)=:f(v).$$
Let $\mathcal{F}(G)$ denote the set of all $\mathbb{Z}$-flows on $G$.
It is called a \textit{flow} if the image of $f$ is in $\mathbb{N}_0:=\{0,1,2,,\ldots \}$.
An $n$-\textit{flow} is a flow with $f(v)\leq n$ for all vertices $v$, i.\ e.\ $f(v)\in \{0,1,\ldots ,n\}=:\underline{n}$.
The set of flows respectively $n$-flows on a digraph $G$ is denoted by $\mathcal{F}(G)$ respectively $\mathcal{F}_n(G)$.
Let $\mathcal{F}^n(G)$ denote the set of flows $f$ on $G$ with $f(e)\in \underline{n}$ for all $n$, the $n$-\textit{bounded} flows. Note: $\mathcal{F}(G)\supset \mathcal{F}^n(G)\supset \mathcal{F}_n(G)$. \end{defn}

\begin{exmp} The flow set $\mathcal{F}_1(G)$ is in bijective correspondence with the set of unions of vertex-disjoint directed cycles of $G$. \end{exmp}

\vskip 0.2in

\section{Graphs and colorings for state-sums of the colored Jones polynomial}

\vskip 0.1in

We will first consider \textit{part arc-graphs} and \textit{arc-graphs} following \cite{GL}, which are defined from Morse projections $\mathcal{K}$. We will also introduce a third type of graph representation (and the corresponding colorings), which we call chord graphs, closely related to the chord diagrams in Vassiliev theory. 

In the following sections we will restrict to closed braid diagrams.

A \textit{Morse projection} $\mathcal{K}$ is a smooth regular embedding of an oriented link in $\mathbb{R}^2\times [0,1]$, which projects to a smooth immersion in the plane $\mathbb{R}^2$, and such that the composition with the projection onto the $y$-axis defines a Morse function on the components of $\mathcal{K}$. Note that each regular projection can be approximated with a Morse projection by planar isotopy. Sometimes we will identify $\mathcal{K}$ just with the embedded circles. 

The image of the immersion defined by $\mathcal{K}$ in $\mathbb{R}^2$ is the part arc-graph $P\mathcal{K}$ and is a four-valent planar directed graph, with signs $\pm $ assigned to each vertex. 
We will color part arc-graphs by $n$-bounded flows according to Definition \ref{defn:edge coloring} such that at all crossings the colors are increasing (respectively decreasing) along the overcrossing arc in the direction of the link orientation. Because of the flow condition, correspondingly the colors are decreasing (respectively increasing) along the undercrossing arc. This defines subsets $\mathcal{F}^n_{\pm }(P\mathcal{K})$ 
of $\mathcal{F}^n(\mathcal{P}\mathcal{K})$. The corresponding states are called $(+)$-colorings (respectively $(-)$-colorings).

Next we enhance the notion of arc graph as defined in \cite{GL} by adding additional decoration. Again we assume that the Morse projection $\mathcal{K}$ of a link is given as above. 

\vskip 0.1in

\begin{center}
\includegraphics[width=0.8\textwidth]{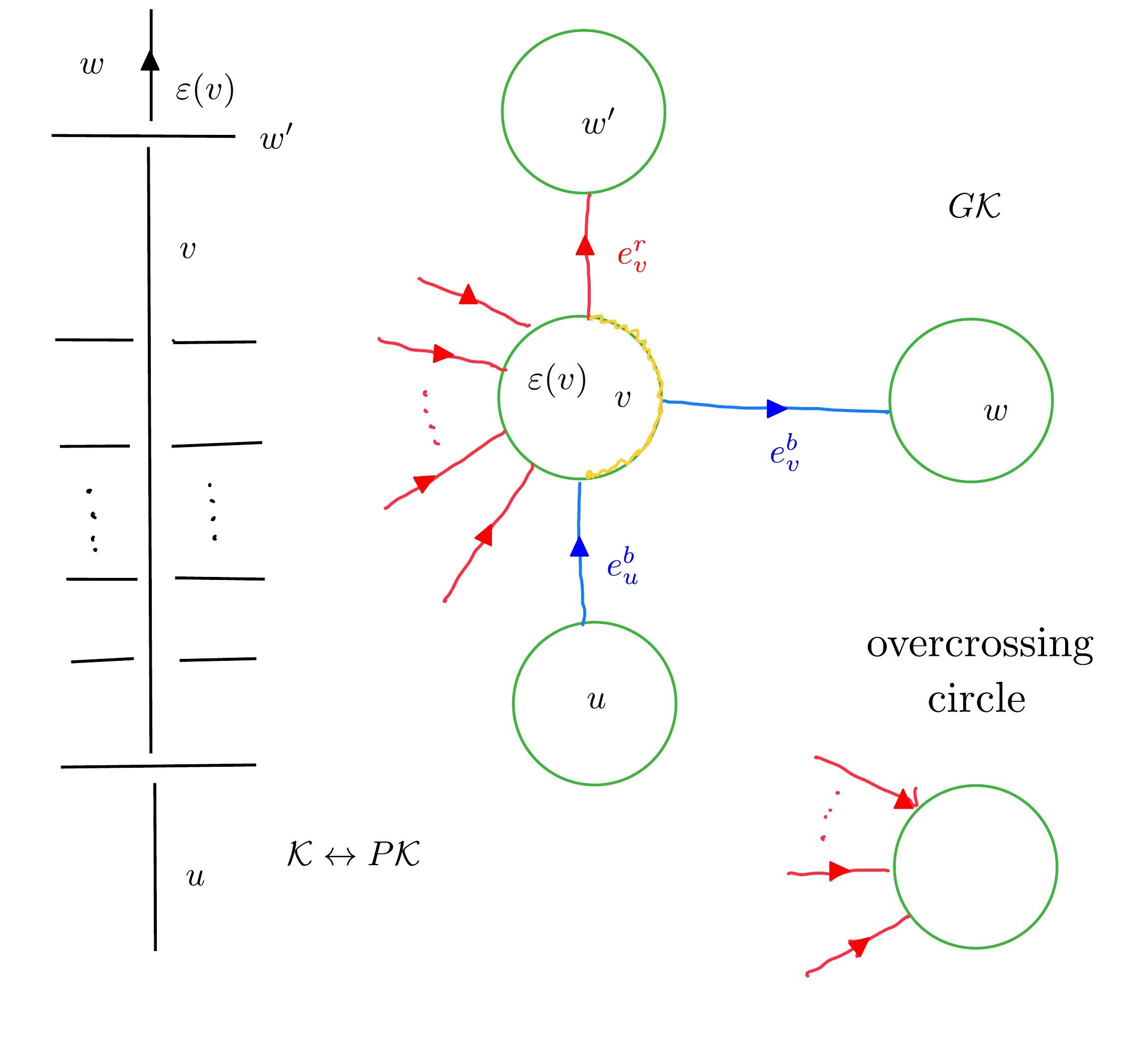}
\end{center}

\vskip -0.5in

\centerline{\textit{Figure 1}}

\vskip 0.1in

\begin{defn} \label{defn:arc-graph} Let $G\mathcal{K}=G_+\mathcal{K}$ be the bichromatic \textit{labeled} fat digraph with (fat) vertices corresponding to overcrossing-arcs of $\mathcal{K}$. The edges of the graph are \textit{decorated} by colors blue and red. 
The directed blue edge $e_v^b$ from a vertex $v$ to the vertex $w$ corresponds to the transition at the undercrossing at which we switch from the overcrossing arc ending at vertex $v$ to the next overcrossing arc ending at $w$. The red arc $e_v^r$ corresponds to the \textit{jump up} transition at the end of arc $v$, as  in the picture to the overcrossing arc (or circle) $w'$.  The vertices corresponding to overcrossing arcs of $G\mathcal{K}$ are decorated by the signs of the vertices of $\mathcal{K}$ at which the corresponding overcrossing arc ends. 
Note that the vertices corresponding to overcrossing circles are not decorated by a sign. 
The boundary of each fat vertex is also decorated by an orientation. For a vertex corresponding to an arc this is from 
the incoming blue edge in the direction of the outgoing red edge. 

There will be additional decorations $\pm $, called \textit{rotation numbers} along the boundaries of fat vertices, which indicate when part arcs of $P\mathcal{K}$ run through a right oriented maximum respectively minimum. 
(We explain below in more detail how the part arcs on the boundaries of fat vertices are identified with the edges of $P\mathcal{K}$.) 
If a part arc runs through a sequence of right oriented extrema the contributions from maxima and minima will cancel, so that each part arc will carry at most one $\pm $. 

The order of incoming red and blue arcs at a fat vertex is according to the corresponding order of overcrossings in $\mathcal{K}$. We call this enhanced graph the \textit{arc graph} for the given projection $\mathcal{K}$. In the setting of \cite{GL} no preferred ordering between the outgoing edges is given. But, our correspondence between states will suggest that the outgoing red edge precedes the outgoing blue edge. 
There is a graph defined from undercrossing arcs in the same way, denoted $G_-\mathcal{K}$.
\end{defn}

\begin{rem} (a) There is exactly one red edge $e_v^r$ and one blue edge $e_v^b$ out of each vertex $v$ of $G\mathcal{K}$ corresponding to an overcrossing arc. For an overcrossing \textit{circle} there are no 
outgoing edges and there is no incoming blue edge. Depending on possible overcrossings, there will be an even number of incoming red edges. A special case of an overcrossing arc is the projection of a figure eight. In this case there are one blue edge and one red edge incident with this vertex, which are both outgoing and ingoing. There can be additional incoming red edges corresponding to the figure eight component overcrossing other components of $\mathcal{K}$. We will usually not consider projections of this form because we could isotope into a \textit{split} projection. 

\vskip 0.1in

\noindent (b) We divide the boundary of each fat vertex corresponding to an overcrossing arc into an \textit{active half-circle}, starting at the incoming blue edge and ending at the outgoing red edge. The complementary half-circle 
is the \textit{passive} half-circle, and is colored yellow in Figure 1. The boundary of the fat vertex corresponding to an overcrossing circle will be considered active. Note that the part-arcs of the active part of the boundaries of fat vertices of $G\mathcal{K}$ exactly correspond to the part-arcs of $P\mathcal{K}$. 

\vskip 0.1in

\noindent (c) Suppose that a Morse projection $\mathcal{K}$ of a link is equipped with
a preferred part arc of the corresponding part arc graph $P\mathcal{K}$ as in Figure 2, i.\ e.\ the only part of the projection outside of the box is the right maxima arc.

\vskip 0.1in

\begin{center}
\includegraphics[width=0.6\textwidth]{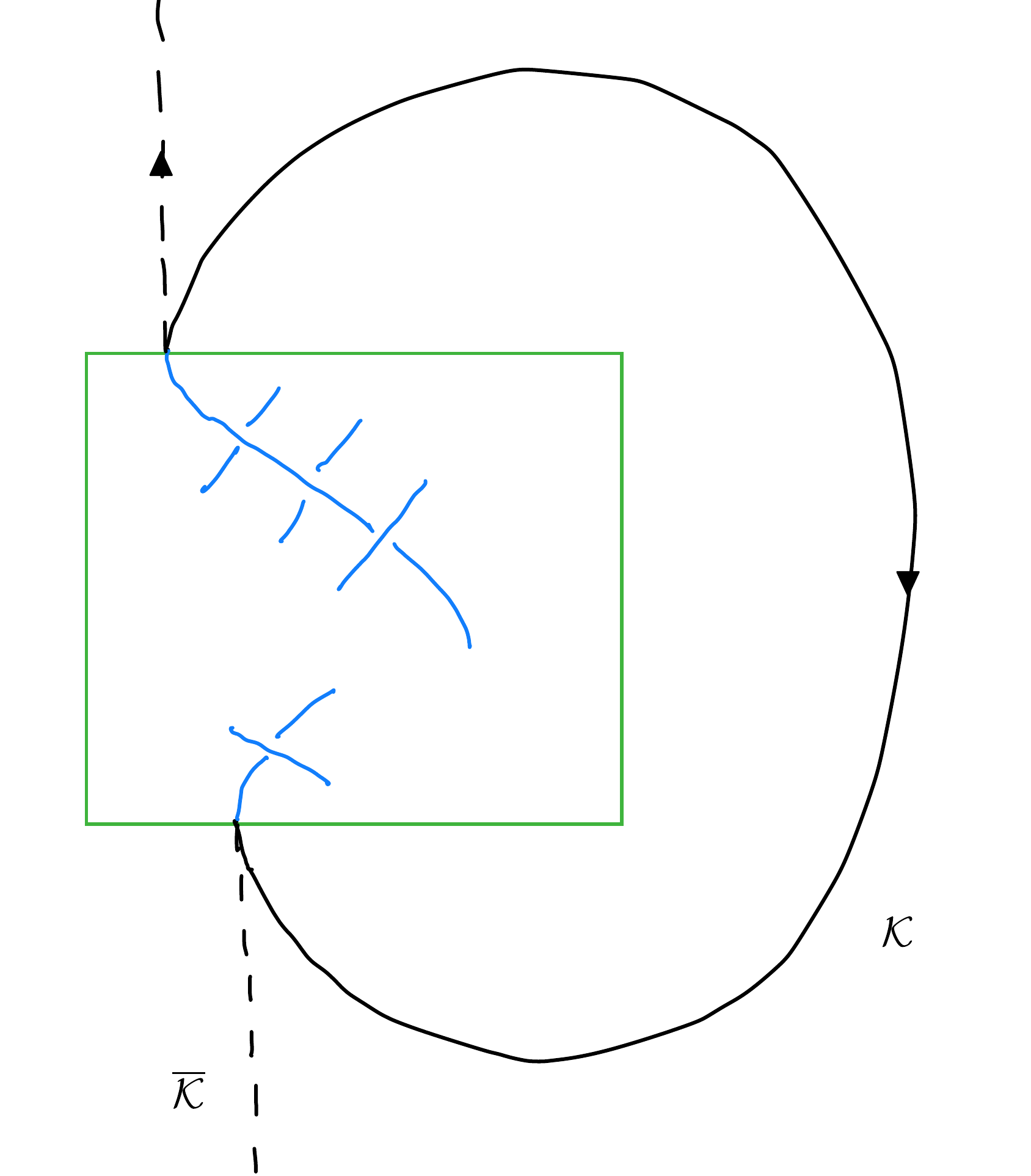}
\end{center}

\centerline{\textit{Figure 2}}

\vskip 0.1in

In this case we can open the projection naturally to include a long component, and consider this \textit{long component} projection $\overline{\mathcal{K}}$ (the dashed half arcs replace the semi-circle). Usually we will omit the dash in notation and denote by $\mathcal{K}$ also the Morse projection of a long component projection. The corresponding part arc-graph is defined from the new projection just as in the closed case but the special arc is now cut into the two oriented half arcs. Note that a long component Morse projection closes 
uniquely to a projection by closing long component with a right oriented maximum. The two processes are inverse to each other.
Suppose that the long component is not an overcrossing arc and the crossing adjacent to the bottom half arc is an undercrossing at the end of the half arc as in Figure 2. We will usually not introduce a different notation for long component projections.
In this case we associate to the long component projection the \textit{reduced arc-graph}, see the discussion in section 7 below.

\vskip 0.1in

\noindent (d) A Morse projection $\mathcal{K}$ is alternating if and only if each fat vertex of $G\mathcal{K}$ has exactly one incoming red edge, or no incoming red edge for a circle arc without crossings. 

\vskip 0.1in

\noindent (e) Note that $P\mathcal{K}$ is a planar graph while $G\mathcal{K}$ is an abstract graph. Not every bichromatic enhanced graph as defined above is the arc graph of a Morse projection. This is due to the missing horizontal compatibility of red edges in the graph. Below we define chord graphs, which will include this information. It seems interesting to extend arc graphs to the setting of virtual projections by adding undirected red edges corresponding to virtual crossings. 

\end{rem}

\begin{defn} A coloring of $G\mathcal{K}$ is a coloring of all blue and red edges. Moreover, for each fat vertex corresponding to an overcrossing circle there will be a color assigned to a part-arc of the boundary, if there are incoming red edges, or to the whole boundary circle otherwise. The coloring will be $n$-bounded at each vertex corresponding to an overcrossing arc, and it will satisfy that the sum of all incoming red edge colors is $0$. Note that if we run through a component of $\mathcal{K}$ we run through the active edges 
of $G\mathcal{K}$ (we will jump across blue edges whenever we switch from one overcrossing arc to the next one), thereby traversing a connected component of $G\mathcal{K}$. While traversing the active boundaries of the corresponding fat vertices following the orientations we cross red incoming edges. We require that the sum of the colors of all incoming red edges minus the sum of the colors of all outgoing red edges adds to zero for each component. We require additionally that the $n$-flow condition is met at each vertex corresponding to an arc. 
See the definition of potential below to understand these conditions. The set of all those colorings of $G\mathcal{K}$ is denoted $\mathcal{F}_n(G\mathcal{K})$. This is a subset of the set of functions $f: E\rightarrow \underline{n}$, except if $G$ contains 
vertices corresponding to overcrossing circles. In this case $f$ is defined on the union of the set of all red and blue edges and a part arc in the boundary of the fat vertex or the whole boundary of the fat vertex as discussed above. 
\end{defn}

\vskip 0.1in

Let $\mathcal{P}\mathcal{K}$ the part arc-graph of a Morse projection $\mathcal{K}$ of a link $K$ with $\mu $ components.  Let $V$ be the set of vertices corresponding to crossings of distinct components of $P\mathcal{K}$, and $V=\sqcup _{1\leq j\leq k\leq \ell} V_{jk}$ with $V_{jk}$ corresponding to the set of crossings of the $j$-th and $k$-th component, $V_{jk}=V_{kj}$.
In the following we will assume that the components of $\mathcal{K}$ are ordered and a part arc is chosen on each component, which is not an overcrossing circle. We call this a \textit{basing}. The chosen part arc will usually be the first part arc on an overcrossing arc but exceptions are possible. Note that when considering braid projections there is a natural way to choose the basing. For long component projections the long component will always be chosen as the first component and the part arc 
can be chosen by one of the two long half-arcs. We call a projection with ordering and basing of components \textit{based}. 

\begin{defn} \label{defn=potential}  A \textit{potential} on a based projection $\mathcal{K}$ is a pair of functions $(r,\beta )$ with $r: V\rightarrow \mathbb{N}_0$ and $\beta :\{2,\ldots ,\mu \}\rightarrow \mathbb{N}_0$. 
Here $V$ is the vertex set of the part arc-graph. We call $\beta_i$ the \textit{base values} of the potential, and the function $r$ defines the 
\textit{jump values}.
For $\mu >1$ the following relations are required for $1\leq i\leq \mu$:
$$\sum_{ i\neq j\\ ,c\in V_{i,j}}\lambda_i(c)r(c)=0$$
We call these relations \textit{cycle relations}. Here $\lambda_i(c)=\pm 1$ depending on whether while walking along the $i$-th component the crossing $c$ is an overcrossing respectively an undercrossing. 
We will sometimes work with $\mathbb{Z}$-potentials defined as above but with $\beta ,r$ possibly taking values in $\mathbb{Z}$.  
\end{defn}

\begin{rem}
\noindent (a) It is important to note that the cycle relations are not independent, e.\ g.\ the relation for $i=1$ is the negative of the sum of the relations for $2\leq j\leq \mu -1$. This follows because each $r(c)$ appears in exactly two relations but with different signs. So if we sum relations for $1\leq i\leq \mu$ then the result is $0$. 

\noindent (b) For a long component projection we will assume that the first component is long component. Moreover, the part arc on the first component chosen 

\noindent (c) A braid induces framing, orientation, ordering  and basing (choice of base-point) of the components of the link. The base-point on a component is determined by the part-arc at the bottom of the link on the first arc on the component from the left.
\end{rem}

Each potential determines an edge-coloring of $P\mathcal{K}$ in $\mathbb{Z}$ as follows: Use the function $\beta$ to define the coloring of part-arcs at basings for all components except the first one. For the left-bottom part-arc on the first component define the color to be $0$. Then, along each component, at crossing $c$ change the edge-coloring according to the diagram below. This means adding the value $r(c)$ along an overcrossing and subtracting along an undercrossing (flow condition). We think of the potential as giving the derivative and base values for the coloring of $P\mathcal{K}$. 
A \textit{contributing potential} is a potential for which all resulting edge-colorings take values in $\mathbb{N}_0$. 
We use this language to indicate that this means that the state will be contributing to the state-sum of the $n$-colored Jones polynomial for a large value of $n$.
The potential is $n$-contributing (usually we just say only contributing if $n$ is clear from the context) if 
all resulting edge colors are in $\underline{n}$. This implies that $r,\beta $ have to take values in $\underline{n}$.
It follows from the formulas for state-sum contributions discussed in sections 5 and 6
that $n$-contributing states are exactly those which will be contributing to the state-sum of the $n$-th colored Jones polynomial.  
Of course, restricting the values of $r,\beta $ to $\underline{n}$ does not necessarily define $n$-bounded flows. But any negative edge coloring will result in non-contributing states. This will follow from the definitions of quantum binomials and Pochhammer symbols in section 5. See Definition \ref{defn:edge coloring} above. We will discuss the inequalities defining contributing states in section 7.
 
The cycle conditions guarantee that returning to the basepoints along each component means returning to the starting color $\beta (j)=:\beta_j, \beta_0:=0$ along each component. Note that starting from the base-points each self-crossing is traversed twice and the color changes cancel. 

\begin{defn} Each Morse projection $\mathcal{K}$, which is equipped with an ordering of the components and a choice of base point on each component, defines a \textit{chord graph} as follows: The chord graph $C\mathcal{K}$ of $\mathcal{K}$ is given by $\mu $ ordered, based and oriented circles. Each crossing is represented by an oriented chord. If the orientation of the circle followed by the orientation of an incoming chord at a point on a circle is the positive orientation of the plane then proceeding along the circle in the direction of its orientation at this point is an overcrossing of the correspond arc on the circle, correspondingly at outgoing chords the arcs along circles correspond to undercrossing arcs of $\mathcal{K}$. 
The sign of the crossing is added as $\pm $ decorating the chords by pairs $(r,\varepsilon )$. Sometimes we do not need this additional decoration. If we want to emphasize the crossing signs we also talk about \textit{signed} chord graphs. 

A coloring of a chord graph is given by adding the data from a potential to the chords, and basepoints on the circles ordered $2,\ldots ,\mu$.  
\end{defn}

The chords subdivide the circles into arcs corresponding to the edges of $\mathcal{K}$. A potential on $\mathcal{K}$ induces a coloring of the chord graph. We will usually work with $\mathbb{Z}$-chord graph colorings, so allow functions $r,\beta $ to take values in $\mathbb{Z}$. Note that changing the orientation of a chord and the sign of the color of the chord will not change the corresponding part-arc coloring on $\mathcal{K}$. By definition, a chord graph is the graph of an abstract collection of $\mu$ circles and chords, which is not necessarily induced from the chord graph on a part arc-graph $\mathcal{K}$. But we assume that all chord graph colorings satisfy the cycle conditions. Note that for each chord graph induced from a projection $\mathcal{K}$ the number of chords between different circles is even, while we do not require this for abstract chord graphs. 

If $\mathcal{K}$ is defined from a closed braid projection then we will furthermore assume that chords \textit{respect} the \textit{horizontal} positions induced from a braid in the following way: For this imagine that we label the strings of the braid $1,2,\ldots , \mathfrak{s}$. Then we can assign numbers between $1$ and $\mathfrak{s}$ to each part-arc represented by a segment between chords on a circle of the chord graph. Here we assume that the part-arcs run vertically upwards along the braid except at a crossing represented by a braid generator $\sigma_i$ corresponding to the permutation $(i,i+1)$ and the incoming left bottom part-arc will get an additional horizontal position label $i$. Note that crossing an \textit{incoming} chord (means that the orientation of the chord is positive) with positive sign means that the horizontal position increases. The same will be true for an outgoing chord carrying a negative sign.  
The point here is that in order to be induced from a braid the horizontal positions have to match as in the following picture. Note that the two vertical arcs represent arcs on component circles but can also be arcs on the same circle. This is just the representation of the braid generator $\sigma_k^{\varepsilon }$. Compare with the discussion of \textit{special} Morse projections as discussed following Theorem \ref{thm:main}.

\begin{center}
\includegraphics[width=0.45\textwidth]{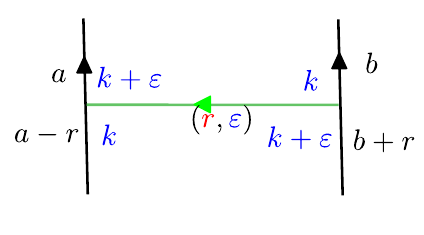}
\end{center}

\vskip -0.2in

\centerline{\textit{Figure 3}}

\vskip 0.1in
  
\noindent Note that switching sign and orientation of the chord and replacing $r$ by $-r$ simultaneously, we get both the same jumps in colors and horizontal positions on the two circle segments. Using negative chord colors $r$ will be essential for some of our arguments. 
  
\begin{exmp} Below is the braid (long) projection of $\sigma_1\sigma_2^{-1}\sigma_1$, which is a $2$-component link. The cycle relation is $r_2=r_3$. For contributing states we further have to restrict to $r_1\leq r_3\leq \beta_2$. The chord graph gives a very transparent way of recording the coloring of all part-arcs. Note that we use the overcrossing arc ordering of the vertices and corresponding values for $r(i)=r_i$. In this case we placed the base color on the part arc following the first undercrossing from the bottom on the second component. Where the base values are placed is not important for part arc graph or chord graph. But since we color blue edges of the arc graph we usually place base values on the specific part arcs leaving the crossing on the undercrossing.

\begin{center}
\includegraphics[width=0.95\textwidth]{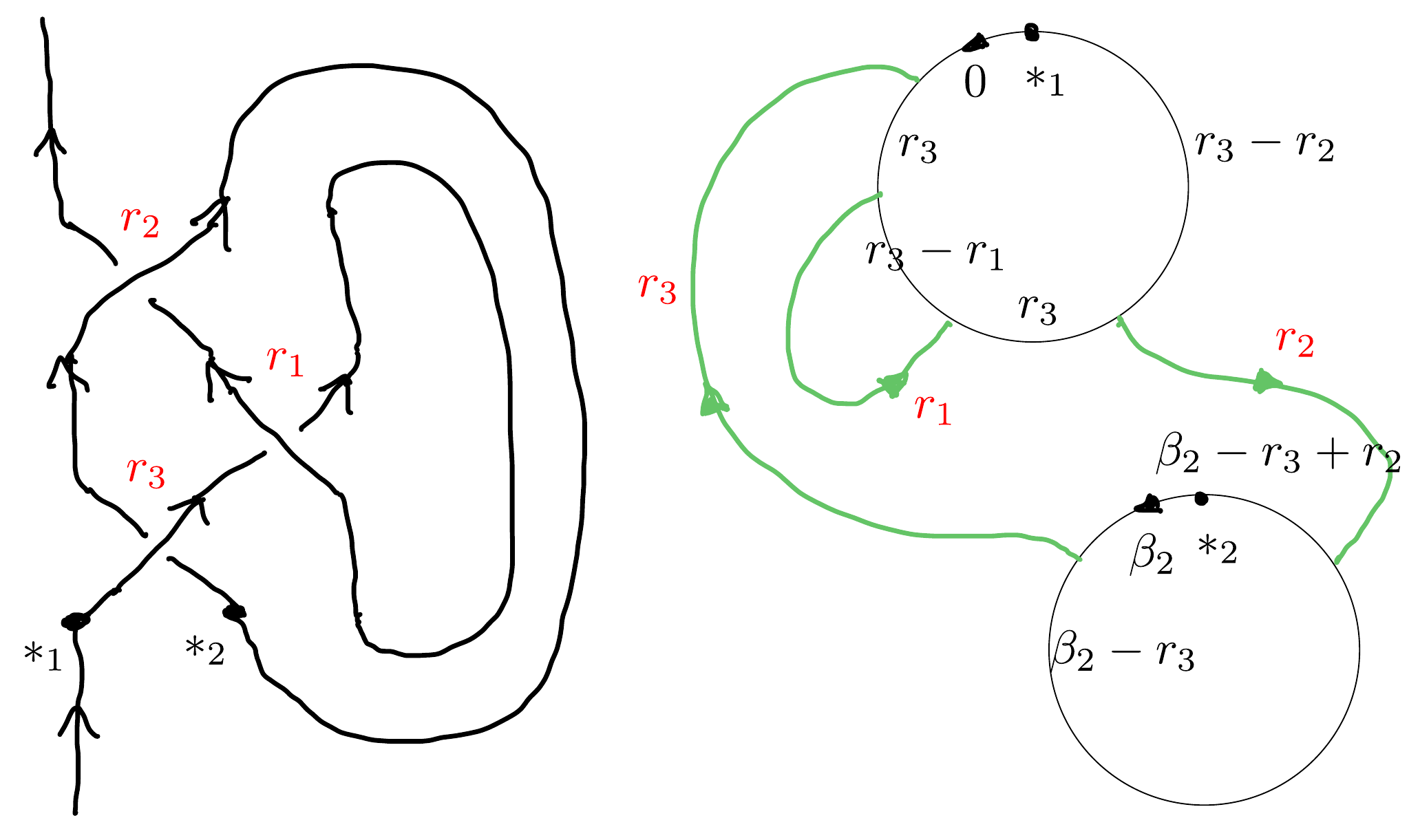}
\end{center}

\vskip -0.2in

\centerline{\textit{Figure 4}}

\vskip 0.1in

\noindent Figure 5 below shows the corresponding arc graph as defined in section 2. Note that ${\color{green}*}$ are labeling the part-arcs, which run downwards in the braid closure. Note that the the $f(e_3^b)=\beta_2-r_3$ is not a base value. 

\vskip 0.1in

\begin{center}
\includegraphics[width=0.75\textwidth]{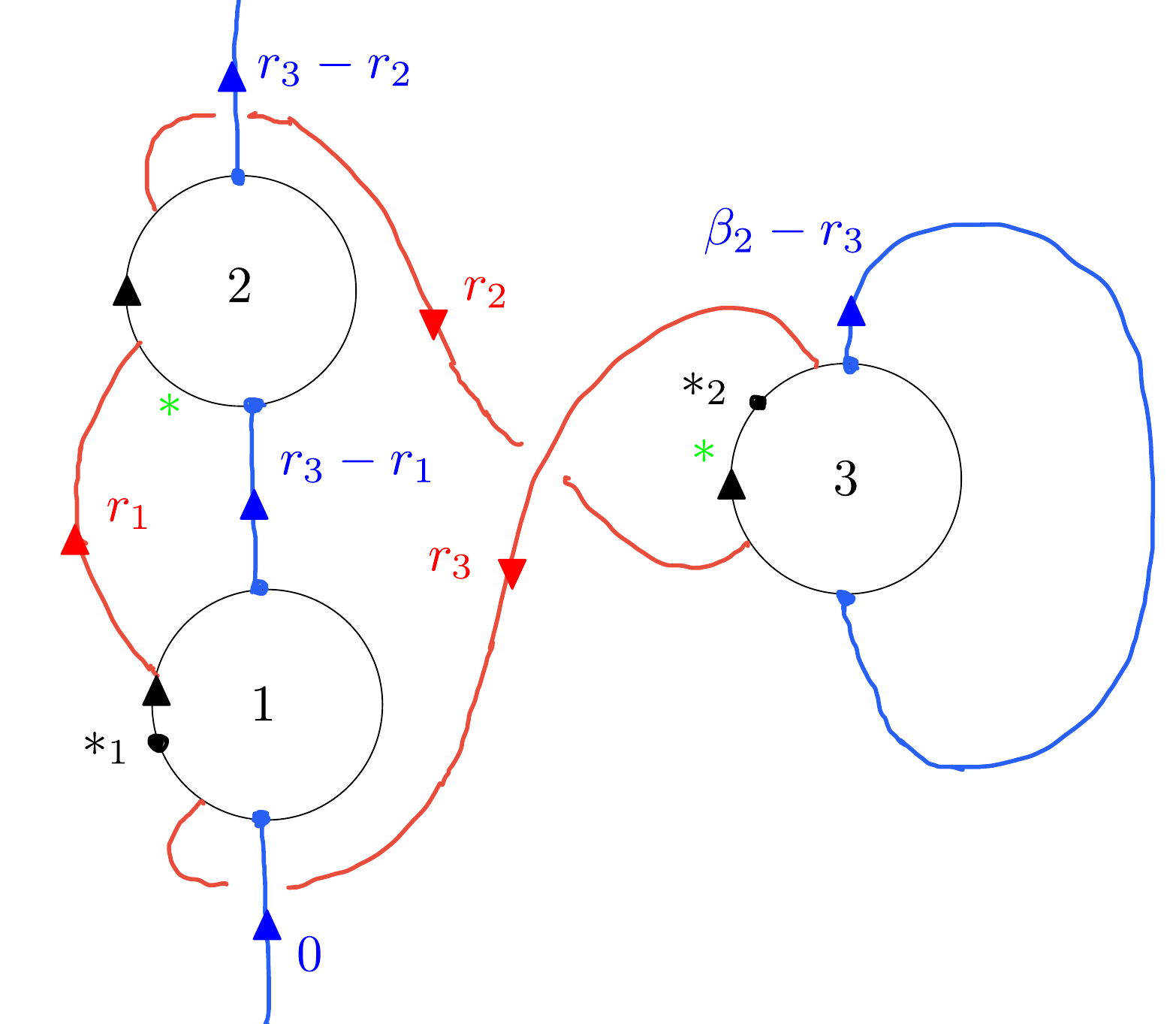}
\end{center}

\vskip -0.2in

\centerline{\textit{Figure 5}}

\vskip 0.1in

\end{exmp}

\noindent From the colored arc-graph the corresponding Garoufalidis-Loebl state-sum contribution will be determined without reference to the part arc-graph, see section 6. 

\vskip 0.1in

\begin{rem} For each Morse projection $\mathcal{K}$, $P\mathcal{K}$ determines $C\mathcal{K}$, from which $G\mathcal{K}$ can be deduced as follows. Just cut the fat vertices from disks bounding the circles of the chord graph by introducing cut arcs (see Figure 6) following each outgoing chord. The blue edges then are easily added like in the picture. Note that the cut arcs added correspond to the passive half arcs in the boundaries of fat vertices. Also colorings of $C\mathcal{K}$ induce the colorings on $G\mathcal{K}$.  

\vskip 0.1in

\begin{center}
\includegraphics[width=0.7\textwidth]{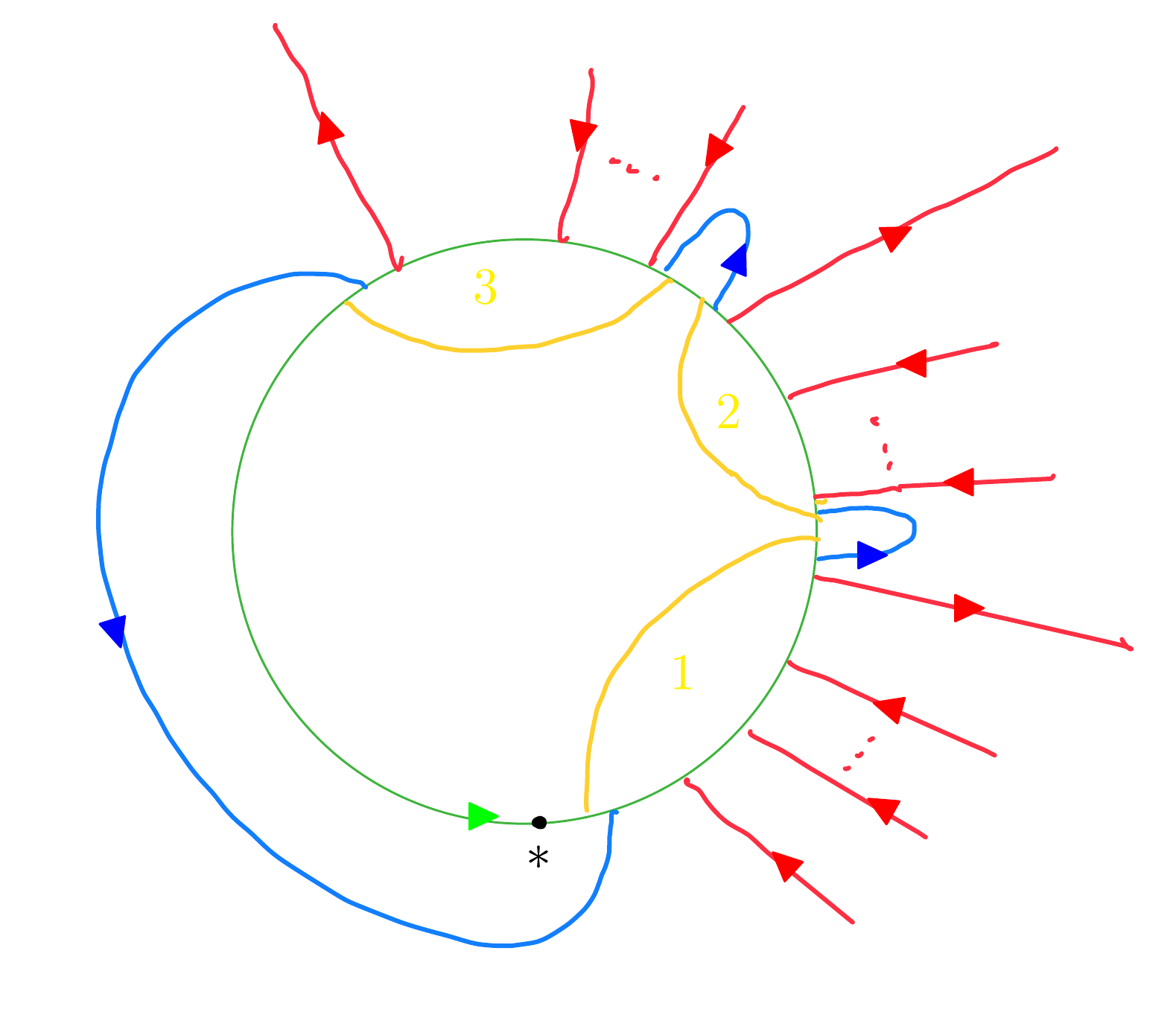}
\end{center}

\vskip -0.2in

\centerline{\textit{Figure 6}}

\end{rem}

\vskip 0.2in

\section{The flow lemma}

\begin{lem}[Flow lemma, this is in \cite{GL} for $n=1$ and knots] \label{lem: flow lemma} Let $\mathcal{K}$ be the Morse projection of a link. Then there is a natural bijection $\phi_{\pm }$
from the set of $\mathbb{Z}$-flows on $G_{\pm}\mathcal{K}$ to the set of $\mathbb{Z}$-flows of $P\mathcal{K}$. For $(+)$ it assigns the color of a blue edge of $G\mathcal{K}$ to the color of the part arc of $P\mathcal{K}$ following the transition, and assigns the red edge colors to the changes in the colors of the two incoming part arcs at the crossing corresponding to the transition. This map restricts for each $n\geq 1$ and $\pm $ to bijections of flows:
$$\phi_{\pm}: \mathcal{F}_n(G_{\pm}\mathcal{K})\rightarrow \mathcal{F}^n_{\pm}(P\mathcal{K})$$     
\end{lem}

\begin{proof} We will only consider the $(+)$-case and let $\mu $ be the number of components of the link in the morse projection. A given general potential as above defines a $(+)$-coloring on $P\mathcal{K}$ 
in $\mathbb{Z}$ as explained following Definition 3.3. But it also defines a corresponding coloring $f$ of the arc graph $G\mathcal{K}$. In fact, define $f(e_v^r)=r(v)$ for all vertices $v$, with vertices identified with overcrossing arcs. Define $f(e_v^b)=\beta_j$ for the vertex $v$ corresponding to the overcrossing arc containing the specific part arc chosen for the basing of the projection in Definition 3.3 for all $v$ of $G\mathcal{K}$ corresponding to overcrossing arcs. There can be vertices of $G\mathcal{K}$ corresponding to overcrossing circles. Assign the corresponding base value $\beta_i$ if this vertex represents the $i$-th component. Note that the part arcs of $P\mathcal{K}$ uniquely correspond to part arcs in the boundaries of the fat vertices. So $(+)$-coloring of $G\mathcal{K}$ defined following Definition 3.3 also defines a coloring of all part-arcs, which are following a vertex $v$ at which an overcrossing arc $v$ ends. So this defines $f(e_v^b)$ for all vertices $v$ of $G\mathcal{K}$. Note that conversely the coloring of $G\mathcal{K}$ extends to all part arcs in the boundaries of fat vertices by using red arcs adjacent to the vertex. In fact, using the orientation of the boundary of each fat vertex, add the value $f(e_w^r)$ for each edge $e_w^r$ ending at vertex $v$.
Figure 7 below shows how the colors are mapped \textit{locally}. The left hand side shows the overcrossing arc 
$u$ and the colorings of $P\mathcal{K}$ defining the corresponding colorings of the edges at the vertex $u$ of $G\mathcal{K}$. If we are given $i,j,k,\ell $ then the colorings of $G\mathcal{K}$ are as given in the Figure. 
Conversely if we have given the edge colorings of the blue and red edges at the vertex $u$ of $G\mathcal{K}$ with the $f(e_u^r)=r$ then the color $i$ is determined by considering the segment of the overcrossing arc $w$ from its beginning until the crossing $w$. The coloring at the beginning is given by $f(e_{w'}^b)$ where $w'$ is the overcrossing arc preceding $w$. Then $i$ is defined by adding the jumps at all overcrossings before reaching the vertex $w$. Then $\ell =i+r$ and $j,k$ are defined as $j_0+r_1+\ldots +r_{\rho }$ and $k=j-r$. Since $r\leq j$ it follows $k\geq 0$. It also follows immediately that  
the flow conditions correspond to each other, and moreover a flow in $\mathcal{F}_+^n(P\mathcal{K})$ defined in this way, corresponds to a flow in $\mathcal{F}_n(G\mathcal{K})$.
For contributing states we assume that all colorings are nonnegative. We have $i,j,k,\ell \leq n$ for the coloring of 
the edges at vertex $u$ of $P\mathcal{K}$ (corresponding to the overcrossing arc vertex of $G\mathcal{K})$). 
Because we have a $(+)$-coloring of $P\mathcal{K}$ it follows $i+j=k+\ell $ and $\ell \geq i$. It follows $\ell -i\geq 0$. The numbers $r_1,\ldots ,r_{\rho }\geq 0$ are the jumps along the overcrossing arc $u$. It follows that 
$j=j_0+r_1+\ldots +r_{\rho }=\ell -i+k$. This is the flow condition also at the vertex $u$ of $G\mathcal{K}$.
Moreover $j\leq n \Longleftrightarrow (\ell -i)+k\leq n$. If we begin with an $n$-flow on $G\mathcal{K}$ we have $j\leq n$, which implies $k=j-r\leq n$. Also $\ell =i_0+r_1'+\ldots +r_{\rho '}'+r\leq n$ because this is a sum of colorings of incoming edges at the vertex $w$, and $i=\ell -r\leq n$. 
So the $n$-bounded $(+)$-flows on $P\mathcal{K}$ exactly correspond to the $n$-flows on $G\mathcal{K}$. 

\begin{center}
\includegraphics[width=0.99\textwidth]{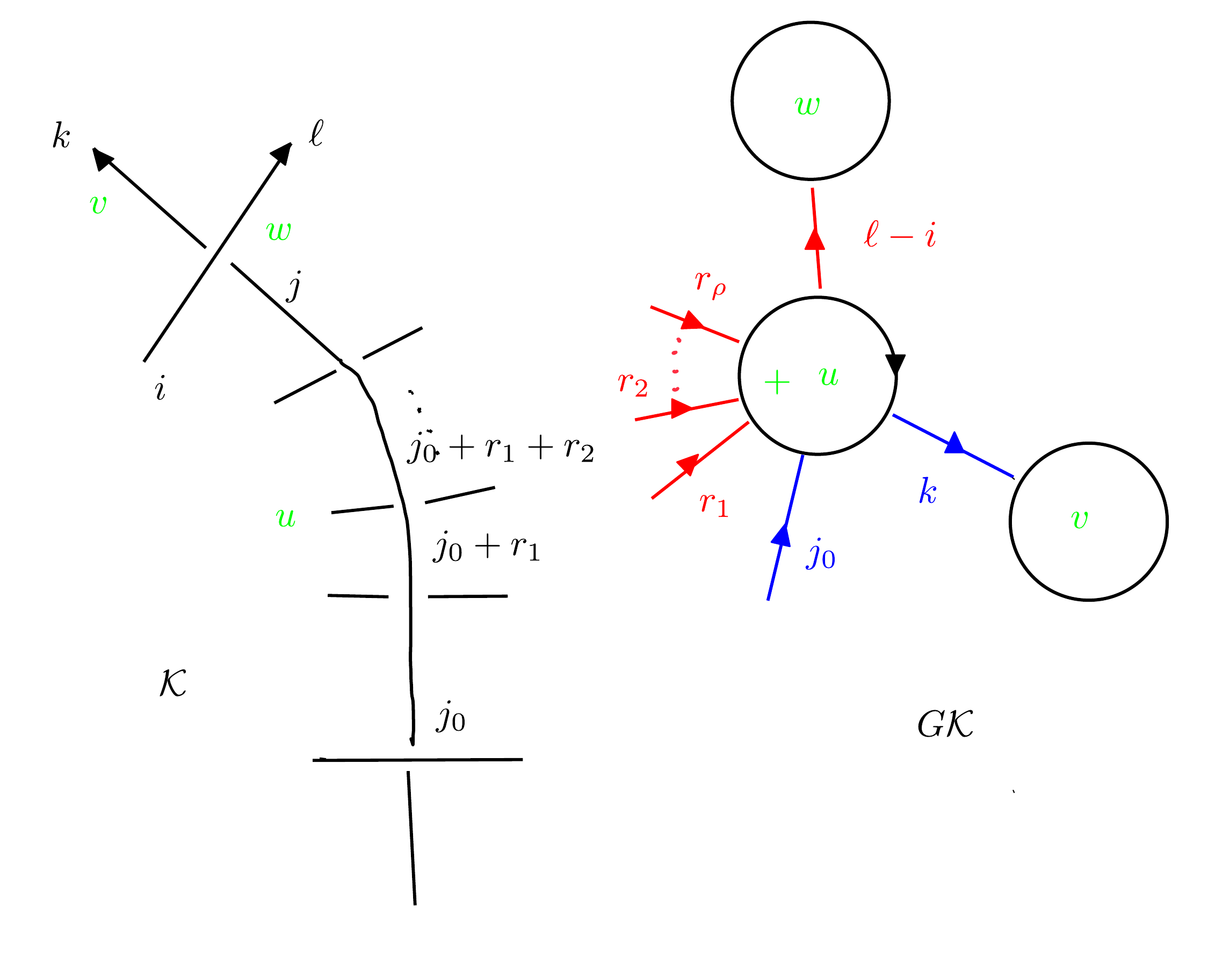}
\end{center}

\vskip -0.2in

\centerline{\textit{Figure 7}}

\end{proof}

\begin{rem} (a) The boundedness restrictions determine the contributing states in each model. 

\vskip 0.1in

\noindent (b) It was pointed out in the proof of  \cite{GL} in the knot case that colorings of the arc graph are determined by the colors of the red edges through the flow condition. Similarly Garoufalidis and Le point out that in the knot case the jumps at the crossings determine the part arc colorings of a long knot. 
It is in fact only this basic observation, which we generalize in the above lemma. 

\vskip 0.1in

\noindent (e) Our formula generalization for \cite{GL} differs from the formula by the replacement of $t$ with $t^{-1}$, corresponding to the various conventions for the skein relation of the Jones polynomial.

\vskip 0.1in

\noindent (f) The flow condition is not special to the colored Jones polynomial.  The $R$-matrix for the standard representation of $\mathfrak{sl}(N)$ given in \cite{T}, 4.2.1 gives non-trivial contributions only $i=j\leq k=\ell $, with $i,j$ denoting the incoming colors and $k,\ell $ the outgoing colors at a crossing. Note that this condition transfers to 
$$f(e_v^b)=\sum_{e<e_v^r}f(e)$$
for the vertex colorings of $G\mathcal{K}$, additionally to the flow condition.
\end{rem}

\section{The $R$-matrix state model}

We review and reformulate the state model for the framed colored Jones polynomial $J_n$ as defined in \cite{G}. Recall that $n$ labels the $(n+1)$-dimensional standard irreducible representation of $\textrm{sl}(2,\mathbb{C})$, or the corresponding type \textbf{1} $U_q(\textrm{sl}(2,\mathbb{C})$-module, see \cite{J}, 5.2.
 Our normalization will coincide with \cite{GL} so that $n=1$ corresponds to the trivial $2$-dimensional representation and the classical Jones polynomial. Link components will all be colored by the same representation. We will use the word \textit{coloring} for the edge-colorings of the projection graphs only.
Throughout $\varepsilon (v)$ will denote the sign of the crossing $v$ (vertex of the part arc-graph $P\mathcal{K})$. Moreover, in \cite{G} a variable $v$ is used for $J_n$ while we will use the variable $v^2=t$ in this paper, which corresponds up to a change from $t$ to $t^{-1}$ to the variable $t$ in \cite{GL}. So our variable $t$ corresponds to $q$ in \cite{G} and $t^{-1}$ in \cite{GL}.
While describing the state-sum formula of \cite{G} we will already rewrite all formulas to variable $t$. This includes rewriting quantum symbols in corresponding crossing-sign dependent functions, which are used in \cite{GL}.

Let $a,b\in \mathbb{Z}$. Then
$$\{a\}:=v^a-v^{-a}=v^{-a}(v^{2a}-1)=( t^{a/2}-t^{-a/2}=t^{-a/2}(t^a-1).$$
and 
$$[a]=\frac{\{a\}}{\{1\}}=\frac{v^a-v^{-a}}{v-v^{-1}}=v^{a-1}(1+v^{-2}+\cdots +v^{-2(a-1)})$$
This corresponds to
$$[a]=t^{\frac{a-1}{2}}(1+t^{-1}+\cdots +t^{-(a-1}).$$
This is mostly important for normalization. The $n$-th colored Jones polynomial of the unknot will be $[n+1]$.
In fact if a framed link $L$ contains $k$ trivially framed unknotted and unlinked components then $J_L=[n+1]^kJ_{L'}$, where $L'$ is the link defined from $L$ by omitting the unknotted and unlinked components, see \cite{G}. 

For $b\geq 0$ we define the \textit{Pochhammer symbol}
\begin{align*}
\begin{split}
\{a\}_b&=\{a\}\{a-1\}\cdots \{a-b+1\}\\
&=t^{-ab/2+b(b-1)/4}(t^a-1)(t^{a-1}-1)\cdots (t^{a-b+1}-1) \\
&=(-1)^bt^{-ab/2+b(b-1)/4}(1-t^a)(1-t^{a-1})\cdots (1-t^{a-b+1}) \\
&=t^{ab/2-b(b-1)/4}(1-t^{-a})(1-t^{-(a-1)})\cdots (1-t^{-(a-b+1)})
\end{split}
\end{align*}
and for $b<0$, $\{a\}_b:=0$. 
We will relate the above definition to the \textit{sign-dependent} Pochhammer symbol (not used by \cite{GL} but we will use in rewriting their formula) defined by
$$\{a\}_{b,t}:=\prod_{s=0}^{b-1}(1-t^{a-s})$$
with $\varepsilon \in \{\pm 1\}$.
\begin{equation}\label{eq=Pochhammer}
\{a\}_b=(-1)^{(1+\varepsilon )\frac{b}{2}}t^{-\varepsilon (\frac{ab}{2}-\frac{b(b-1}{4})}\{a\}_{b,t^{\varepsilon }}
\end{equation}

It follows for $b\geq 0$ from the definitions in \cite{G}:
\begin{equation*}
{a \choose b}:=\frac{\{a\}_b}{\{b\}_b}=t^{b(b-a)/2}\frac{(t^a-1)\cdots (t^{a-b+1}-1)}{(t^b-1)\cdots (t-1)}
\end{equation*}
Note that the expression on the right hand side is not defined for $b=0$, and is defined to be $1$ in this case. 
As above these quantum binomials are related to the sign-dependent versions in \cite{GL} as follows:
$(a)_t:=\frac{t^a-1}{t-1}$ and 
$(a)_t!:=(1)_t(2)_t\cdots (a)_t$
and
$${{a}\choose {b}}_t:=\frac{(a)_t!}{(b)_t!(a-b )_t!}=\frac{(t^a-1)\cdots (t^{a-b+1}-1)}{(t^b-1)\cdots (t-1)}$$
and thus 
\begin{equation}\label{eq=binomial}
{a \choose b}=t^{\varepsilon b(b-a)/2}{{a}\choose {b}}_{t^{\varepsilon }}
\end{equation}
This follows because
\begin{align*}
{{a}\choose {b}}_{t^{-1}}&=\frac{(t^{-a}-1)\cdots (t^{-(a-b+1)}-1)}{(t^{-b}-1)\cdots (t^{-1}-1)} \\
&=\frac{t^{-(a+(a-1)+\cdots (a-b+1))}}{t^{-(b+(b-1)+\cdots +1)}}{a\choose b}_t\\
&=t^{b(b-a)}{a\choose b}_t
\end{align*}

We will use the following well-known lemma for the \textit{symmetry} of quantum binomials:

\begin{lem} \label{lem=binomial symmetry} For $c,d$ nonnegative integers we have 
$${c+d \choose c}={c+d \choose d}$$ 
\end{lem}
\begin{proof}[Proof of lemma] We have the formula:
$${c+d\choose c}=t^{-cd/2}\frac{(t^{c+d}-1)\cdots (t^{d+1}-1)}{(t^c-1)\cdots (t-1)}$$
If $c\leq d$, multiplying both numerator and denominator on the right hand side by
$$(t^d-1)\cdots (t^{c+1}-1)$$
and noting that $d(d-(c+d))=c(c-(c+d))$ the result follows. 
If $d<c$ both numerator and denominator contain the product
$$(t^c-1)\cdots (t^{d+1}-1),$$
so after cancellation the result follows in this case too.
\end{proof}

Note that by formula \eqref{eq=binomial},
$${c+d\choose c}=t^{-\varepsilon cd/2}{c+d\choose c}_{t^{\varepsilon }}$$
and 
$${c+d\choose d}=t^{-\varepsilon dc/2}{c+d\choose d}_{t^{\varepsilon }}$$
so that symmetry of quantum binomials also holds for the sign-dependent versions. 

The formula in Proposition 3.7 of \cite{G} generalizes to link projections, as described in \cite{G}, Remark 3.5, in the following way:
Let $\mathcal{F}_{-,*}^n$ be the set of all colorings $(-)$ of the part arc-graph $P\mathcal{K}$ of a link projection $\mathcal{K}$ as defined in section 2. We we assume that the first strand of the braid projection is colored $0$ at the bottom, which is indicated by $*$. 
Let $J'_n(K)$ be the framed colored Jones-polynomial \cite{G}, we will discuss more details in section 6. Let $\mathfrak{s}$ be the number of braid strings and let $\{c_1,\ldots c_{\mathfrak{c}}\}$ be the set of crossings in the braid projection. Then 
$$J'_n(K)=\sum_{s\in \mathcal{F}_{-,*}^n}F(n,s)$$
where $F(n,s)$ is the product
\begin{equation}\label{eq=Fterm}
F(n,s):=\prod_{\ell=2}^{\mathfrak{s}}v^{-n+2b_{\ell }(s)}\prod_{k=1}^{\mathfrak{c}}f_{\varepsilon (c_k)}(n,i(k,s),j(k,s),r).
\end{equation}
Here $i(s),j(s)$ are the colors at the crossing $c_k$ determined by the state $s$ and the following convention
(colors are decreasing on overcrossings):

\begin{center}
\includegraphics[width=0.8\textwidth]{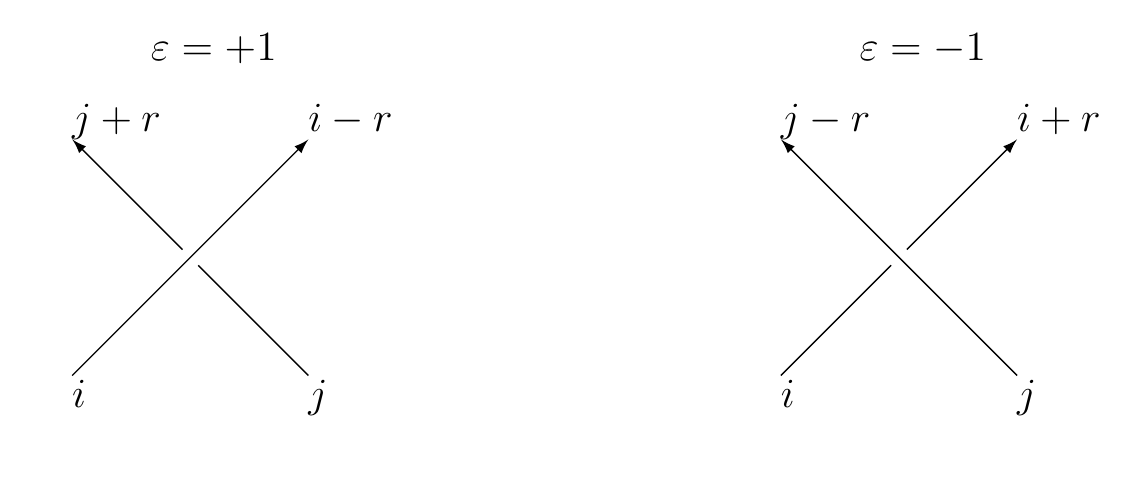}
\end{center}

\vskip -0.2in

\centerline{\textit{Figure 8}}

\vskip 0.1in

with the functions $f_{\pm }$ are defined by (see \cite{GL}, 3.1, \cite{Ja}, Chapter 3, and \cite{KiM}, 2.32 and 2.34)
\begin{equation}\label{eq=positive crossing}
f_+(n,i,j,r):=(-1)^rv^{-((n-2i)(n-2j)+r(r-1))/2}{j+r \choose r}\{n+r-i\}_r
\end{equation}
while the contribution at a negative crossing will be
\begin{equation}\label{eq=negative crossing}
f_-(n,i,j,r):=v^{((n-2i-2r)(n-2j+2r)+r(r-1))/2}{i+r \choose r}\{n+r-j\}_r.
\end{equation}

\vskip 0.1in

\begin{rem} (a) What is called state $s$ in \cite{G} in the long knot case we call a potential inducing the state on the part arc-graph with part arc colorings $(i(s),j(s))$ on part arcs labeled by the crossings. In order to be able to include overcrossing circles without crossings in the link case restrict the values of $r,\beta_j$ to $\underline{n}$ so that the colorings at extrema are in $\underline{n}$.

\noindent (b) Some care has to be taken when considering the above formula for the state-sum with Proposition 3.7 in \cite{GL} for the long knot case. Garoufalidis and Loebl sum over all possible nonnegative jumps at vertices, noticing that the quantum algebra definitions will set state-sum contributions to $0$ if colorings of part arcs are $>n$. This will not be true in the link case because of the $b_i(s)$ in the formula because there can be overcrossing circles such that the maximim is not connected to a crossing at which the label $b_i(s)$ appears. We have to restrict the state space so that all part arc colors are $\leq n$, or only consider projections without overcrossing circles. Note that the set of contributing states is usually much more restricted by the requirement that all induced part arc colors $c$ (defined from a given potential) satisfy $c\in \underline{n}$. 
\end{rem}

\noindent The following examples will show the dependence of state-sums on the framing.
We will also be able to state the skein relation of the classical Jones polynomial in the framed case, which is of course related to a modification of the Kauffman bracket. 

\begin{exmp} If we represent the unknot by the empty braid then there is a unique state with color $0$ on the single strand, and the state-sum gives $J'_n(U)=1$ for all $n$. The corresponding vertex product is empty because there is no crossing. It is known that the colored Jones polynomial is changed by a power of $v$ if we change the framing in the first component of a link (see \cite{L}, (1.7)). If we use the closure of the braid $\sigma_1$ to represent the unknot $U(1)$ with framing $+1$ the decreasing condition at the single vertex implies that the only state is the all-$0$ state and $J'_n(U(1))=v^{-n^2/2-n}$. On the other hand if we use $\sigma_1^{-1}$ to represent $U(-1)$, the $-1$-framed unknot, then we get states labeled by $0\leq r\leq n$ for a single crossing. 

\begin{center}
\includegraphics[width=0.3\textwidth]{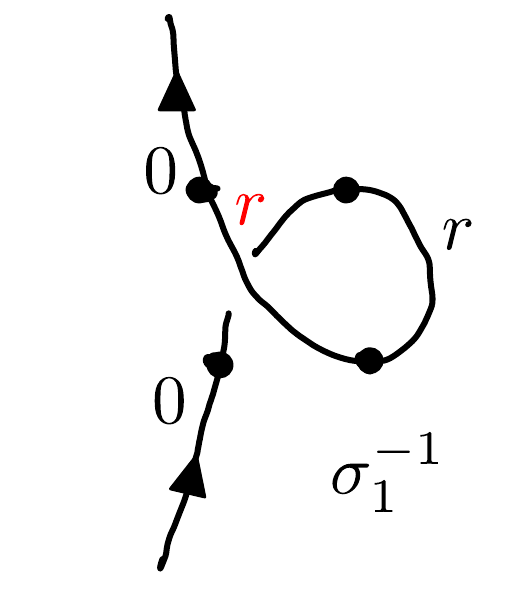}
\end{center}  

\vskip -0.2in

\centerline{\textit{Figure 9}}

\vskip 0.1in

Using ${r \choose r}=1$ the following state-sum results for $U(-1)$:

\begin{align*}
J_n'(U(-1))&=\sum_{0\leq r\leq n}v^{-n+2r}v^{\left(n(n-2r)+r(r-1)\right)/2}\{n\}_r\\
&=v^{n^2/2-n}\sum_{0\leq r\leq n}v^{r(2-n)+\frac{r(r-1)}{2}}\{n\}_r,
\end{align*}

\noindent and thus the following identity for the well-known $q$-Pochhammer symbol:

$$\sum_{0\leq r\leq n}v^{r(2-n)+\frac{r(r-1)}{2}}\{n\}_r=v^{2n},$$

\noindent which is easily proved by induction using $\{n+1\}_{r+1}=(v^{n+1}-v^{-(n+1)})\{n\}_r$.
For $n=1$ and $\{1\}_0=1$ (empty product) and $\{1\}_1=v-v^{-1}$ the right hand side is
$$1+v(v-v^{-1})=v^2$$
The induction step gives:
\begin{align*}
\ &1+\sum_{1\leq r\leq n+1}v^{r(2-(n+1))+\frac{r(r-1)}{2}}\{n+1\}_{r}\\
=&1+\sum_{0\leq r\leq n}v^{(r+1)[2-(n+1)]+\frac{r(r+1)}{2}}(v^{n+1}-v^{-(n+1)})\{n\}_r\\
=&1+v^2\cdot v^{2n}-v^{-2n}\cdot v^{2n}\\
=&v^{2(n+1)}
\end{align*}
\end{exmp}

\begin{exmp} We consider the case $n=1$. The framed Jones polynomial $J'=J_1'$ defined by the state-sum \cite{G} above satisfies the skein relation
$$v^{-1/2}J'(K_+)-v^{1/2}J'(K_-)=(v^{-1}-v)J'(K_0)$$
and $J'(U)=1$ for $U$ the trivially framed unknot (represented by the $1$-braid closure of a single strand), $J'(K\sqcup U)=(v^{-1}+v)J'(K)$ for each non-empty framed link $K$, and 
$$J'(K(1))=v^{-3/2}J'(K).$$
It is a good exercise to check the skein relation follows directly from the formula above: For $n=1$ only $r=0$ and $r=1$ are possible values at a crossing. For a given $\mathcal{K}$ and crossing we can divide the set of states at a crossing into $\mathcal{F}_{(\pm ,r)}$ with $r\in \{0,1\}$. For $r=0$ there are all colors $0$ or $1$ states possible for both the positive and negative crossing. For $r=1$ the sign of the crossing determines the possible part arc colorings at the crossing. For the Conway smoothing $K_0$ at the crossing the states are naturally divided into subsets by the two part arcs colored as $(\textrm{left},\textrm{right})$: $(0,0)$, $(0,1)$, $(1,0)$ and $(1,1)$. Now the all $0$, respectively all $1$-states, in $\mathcal{F}_{(\pm ,0)}$ correspond to $(0,0)$ and $(1,1)$-states of the smoothing. The vertex contribution of the crossing at the positive crossing is $v^{-1}$ while the contribution at the negative crossing is $v$, corresponding to the states on $K_0$ 
multiplied by $v^{-1}-v$. Recall that $\{1\}_1=v-v^{-1}$. For $r=1$ admissibility reduces the colorings at the positive crossing to $i=1$ and $j=0$ and the formula for $f_+$ shows vertex contribution $(-1)^1v^{-(1-2)/2}\{1\}_1$, which multiplied by $v^{-1/2}$ contributes the factor $-(v-v^{-1}=v^{-1}-v$ for each corresponding $(1,0)$-state of the smoothing, while the formula for $f_-$ shows vertex contribution $v^{(1-2)/2}\{1\}_1$, which multiplied by $-v^{1/2}$ 
also contributes $v^{-1}-v$ as a factor but this time to the $(0,1)$ states on $K_0$.
If we use the usual writhe normalization we have the definition of the unframed Jones polynomial
$$J(K)=v^{\frac{3}{2}\omega (\mathcal{K})}J'(\mathcal{K})$$
Substituting in the skein relation for framed links above and using $]\omega (K_{\pm})=\omega (K_0)\pm 1$ where we use the notation for the link instead of the projection, we get
$$v^{-2}J(K_+)-v^2J(K_-)=(v^{-1}-v)J(K_0)$$ 
\end{exmp}

\begin{rem} We note that in the formula for $F$ in \cite{G}, Proposition 3.7 the variable $v$ has to be replaced by $q$ and thus we have a factor $2$ additionally. This for example can easily be seen from the calculations of   $J_n' $ for the examples above. \end{rem}

\section{Proving the Garoufalidis-Loebl state-sum}

We will now rewrite the usual state-sum contributions as described in \cite{G} into the Garoufalidis Loebl state-sum contributions using the above identities and the flow Lemma. There two important steps: The first will be to transform the formulas in the last section into formulas for the $(+)$-model (increasing colors on overcrossing arcs at overcrossings). The second is to rewrite the two functions $f_{\pm }$ into a single vertex contribution formula, depending on the crossing sign. 

Our convention at crossings will be: 
\begin{center}
\includegraphics[width=0.8\textwidth]{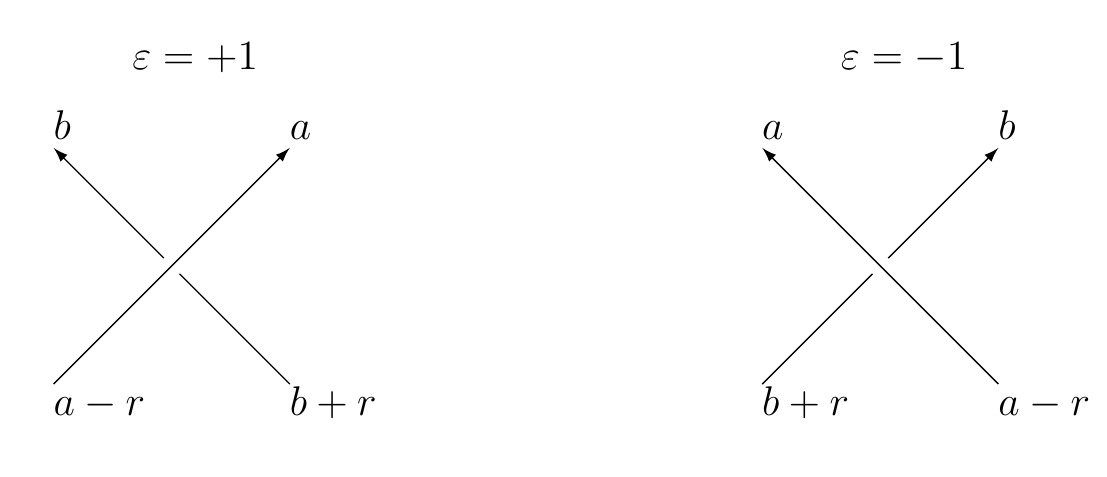}
\end{center}
\vskip -0.2in

\centerline{\textit{Figure 10}}

\vskip 0.1in

Note that in order to apply the formulas from section 5 we change the orientation of the projection. It is known that a change of orientation does not change the colored Jones polynomial. There will be though a change in rotation numbers, which will show up in defining $\textrm{rot}(f)$ for a state on $G\mathcal{K}$ in this section, and it will appear in the definition of $\delta '(\mathcal{K},n)$, which is a global power of $t$ for the state-sum contributions. We also point out that there are different conventions with respect whether right or left oriented extrema introduce the terms $b_{\ell }(s)$ in equation \eqref{eq=Fterm}.

\vskip 0.1in

ln the following we consider now $a,b,r,\varepsilon $ as functions on the set of crossings and often omit the argument in formulas. Then the vertex state-sum contribution at $c$ is given by \eqref{eq=positive crossing} and \eqref{eq=negative crossing}, with the variable $v$ replaced by the variable $t=v^2$. We also use here the symmetry of the quantum binomials. The following expression is immediate by substituting formulas \eqref{eq=binomial} and \eqref{eq=Pochhammer} and using Lemma \ref{lem=binomial symmetry}. It gives a single expression to be evaluated at each crossing with the sign of the crossing contained in the definitions of the $q$-algebra symbols. 

\begin{align*}
&(-1)^{(1+\varepsilon)\frac{r}{2}}t^{-\varepsilon /2\left((n-2b-(1-\varepsilon )r)(\frac{n}{2}-a+\frac{1}{2}(1-\varepsilon )r \right)+\frac{r(r-1)}{2}})\\
&\cdot  t^{\varepsilon /2 \cdot br}{{b+r}\choose {r}}_{t^{-\varepsilon}}\cdot (-1)^{(1+\varepsilon )\frac{r}{2}}t^{-\varepsilon (\frac{(n+r-a)r}{2}-\frac{r(r-1)}{4})}  \{n+r-a\}_{r,t^{\varepsilon }}
\end{align*}

\noindent Using that $1+\varepsilon \in \{0,2\}$ to eliminate the sign dependence, and cancelling the $r(r-1)/4$-terms, we see that the vertex contribution function is given by
\begin{equation}\label{eq=vertex contribution}
t^{\varepsilon /2\cdot (-n^2/2+ n\rho +\sigma)-\frac{1}{2}\tau}\cdot {{b+r}\choose {r}}_{t^{-\varepsilon}}\cdot \{n+r-a\}_{r,t^{\varepsilon }}
\end{equation}
with the two functions, defined for each crossing:
\begin{align*}
\rho &=b+a-r \\
\sigma& =-2ab+(b-a)r+br-(r-a)r+r^2 \\
\tau&=r(b-a+r)
\end{align*}
Here we also used $\frac{1}{2}(1-\varepsilon )^2=(1-\varepsilon )$.
If we let $a=r+R$ then 
$$\sigma =-2ar-2aR+(b-a)r+br-r^2+ar+r^2=2b(r-a)=-2bR$$

\vskip 0.1in

\noindent The crucial geometric statements leading to the Garoufalidis-Loebl  formula are contained in the following
two propositions. As noted above we usually omit the arguments of the functions defined on the set of crossings.

\vskip 0.1in

\begin{prop} For each $\mathbb{Z}$-coloring of a chord graph, respectively the corresponding coloring of a part arc-graph inducing this coloring,
$$\sum_cr(a-b-r)=0$$
\end{prop}

\vskip 0.1in

\begin{proof}
The formula is equivalent to:

\begin{equation}
\sum_cr(a-b)=\sum r^2
\end{equation}

We will first show the equation for chords only between distinct components. As a first step we consider the abstract chord graph with exactly one chord for each $1\leq j<k\leq \ell $ directed from the $k$-th to the $j-th $ component with chord color $r_{j,k}$. In this case the sum $\sum_cr(a-b)$ is easily computed as follows. The contribution of $r_{jk}$ in the sum is 
\begin{equation}
r_{jk}(-\sum_{1\leq i\leq j-1}r_{ij}+\sum_{j<i\leq k}r_{ji}+\sum_{1\leq i\leq j}r_{ik})
\end{equation}
\vskip 0.1in

\begin{center}
\includegraphics[width=0.8\textwidth]{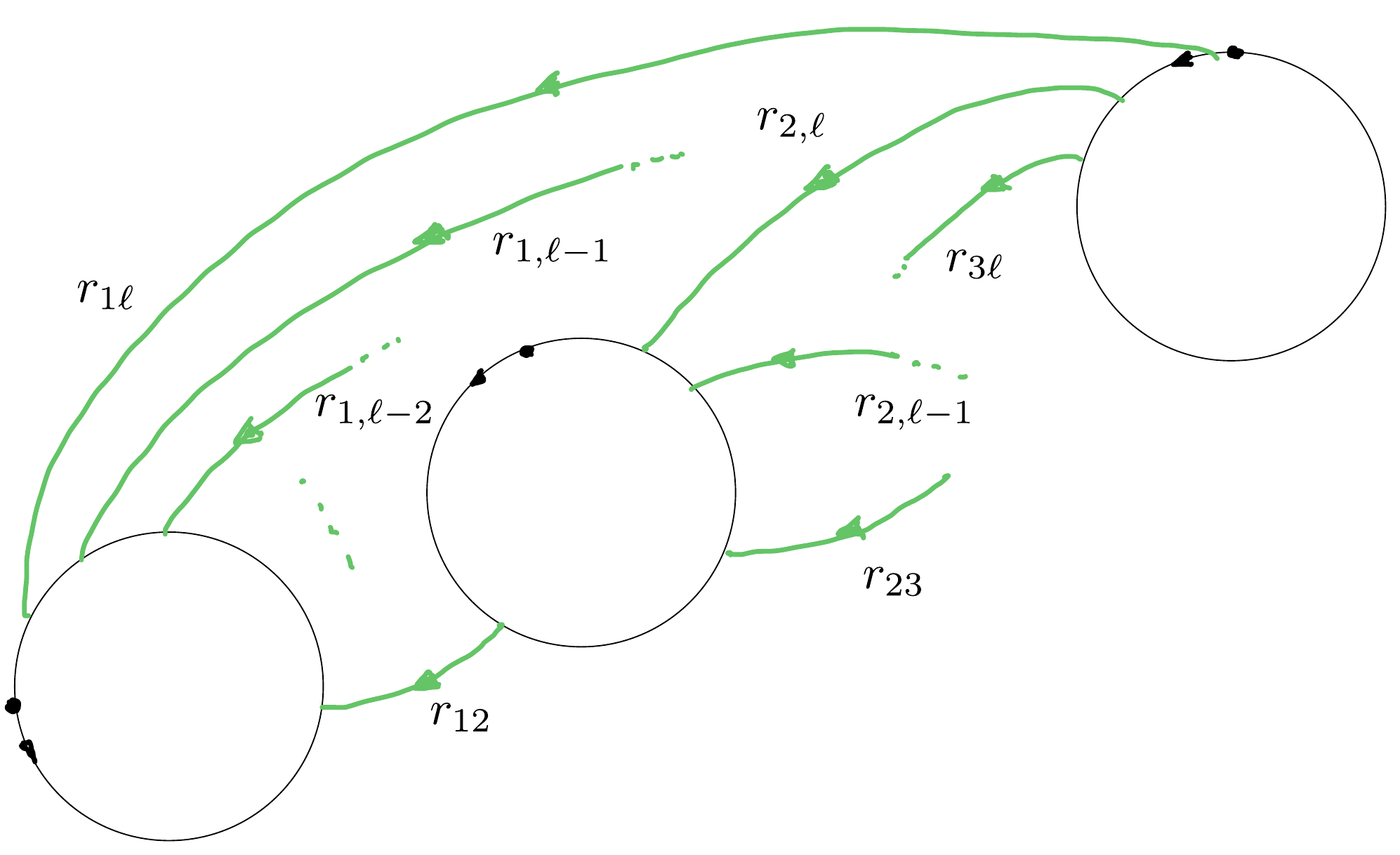}
\end{center}

\vskip -0.2in

\centerline{\textit{Figure 11}}

\vskip 0.1in

By using the cycle relation for the $j$-th component, the first two sums can be replaced by $-\sum_{k<i\leq \mu }r_{ji}$ and so the contribution from $r_{jk}$ is (those be read from the incoming chords of the $j$-th component and the corresponding values on the $k$-th component):
\begin{equation}
r_{jk}(r_{jk}+\sum_{1\leq i<j}r_{ik}-\sum_{k<i\leq \mu }r_{ji})=r_{jk}^2+r_{jk}(\sum_{1\leq i<j}r_{ik}-\sum_{k<i\leq \mu }r_{ji})
\end{equation}

\vskip 0.1in

\begin{center}
\includegraphics[width=0.8\textwidth]{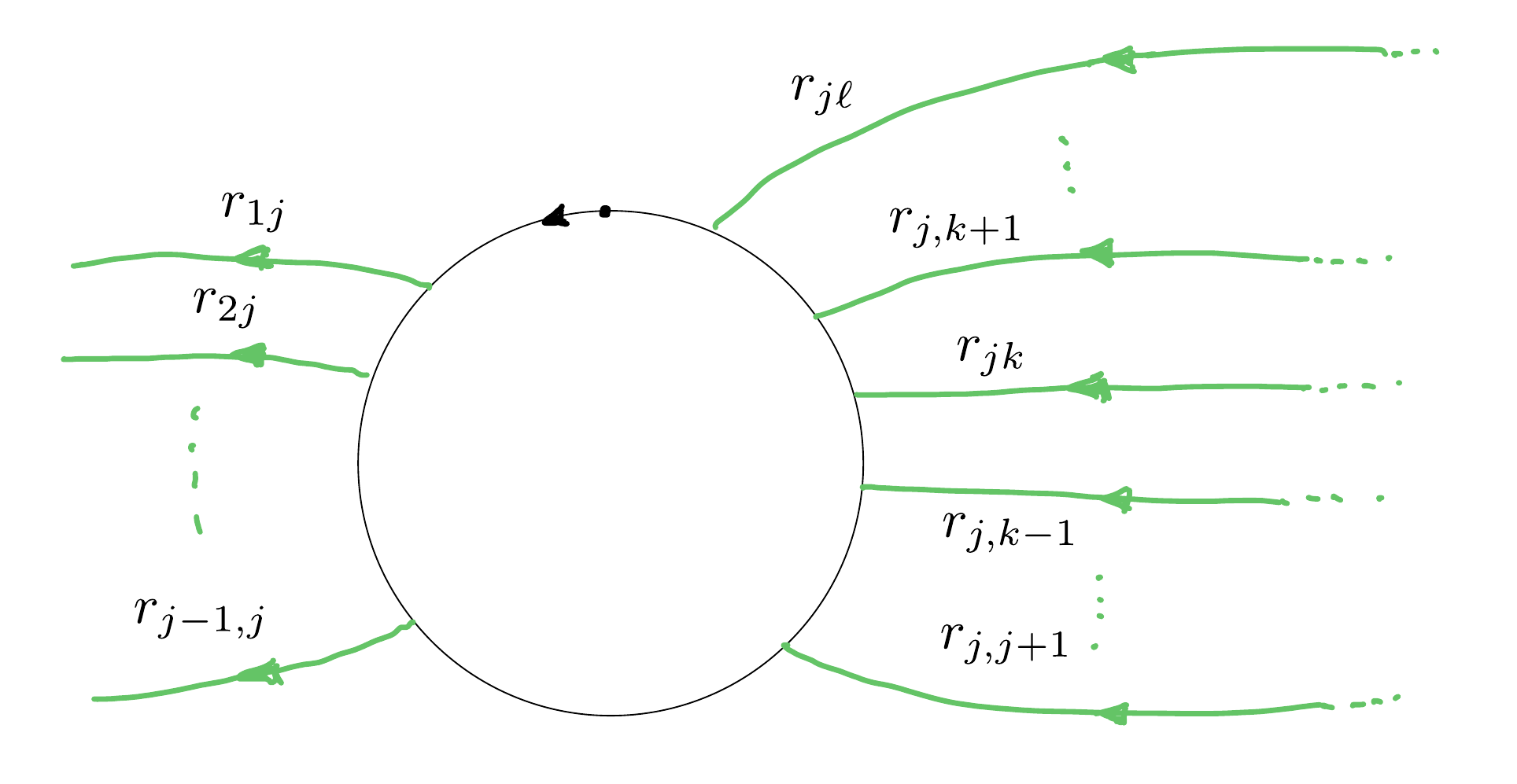}
\end{center}

\vskip -0.2in

\centerline{\textit{Figure 12}}

\vskip 0.1in

Thus it suffices to prove the claim:
\begin{equation}
\sum_{1\leq j<k\leq \mu}r_{jk}(\sum_{1\leq i<j}r_{ik}-\sum_{k<i\leq \mu }r_{ji})=0
\end{equation}
Now we can define a skew-symmetric matrix $\mu \times \mu $-matrix $(r_{jk})$ with the coefficients $r_{jk}$ for $j<k$ given by chord colors of the chord graph. The $\mu $ cycle relations then exactly say that the sum of columns of this matrix is the zero-vector. In this case, the claim follows from the lemma below. 
The chord graphs resulting from actual link projections differ from the standard graph considered above by doubling of chords and switching the order of attachment to the circles. We will show that the relation of the proposition remains true in this case. Also note since the relation holds for chord colorings by arbitrary integers and 
changing chord orientation and color by multiplication by $-1$ we finally prove the result for arbitrary chord graphs and thus arbitrary part arc-graphs. 

We calculate the sums $\sum r(a-b)$ for chord graphs differing locally only by the following chord respectively two chords as in the picture 

\vskip 0.2in

\begin{center}
\includegraphics[width=0.95\textwidth]{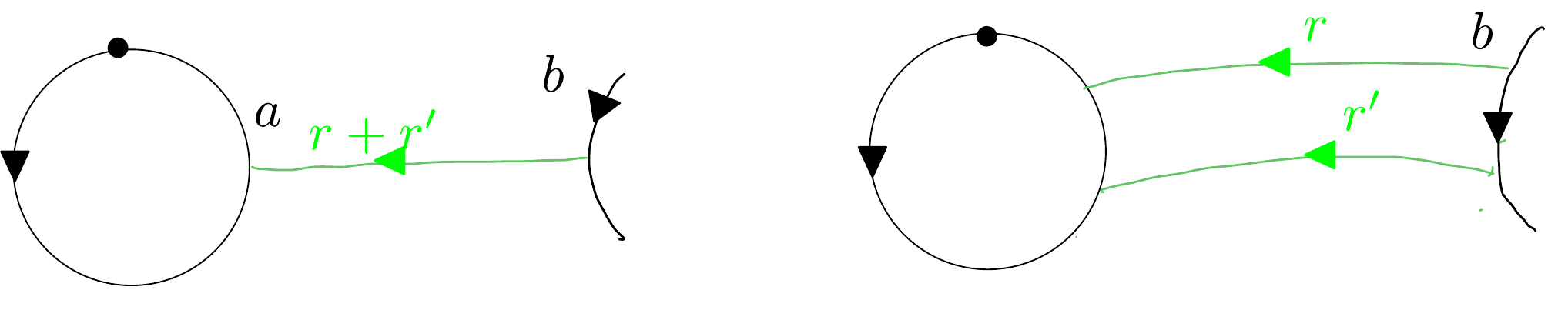}
\end{center}

\vskip -0.2in

\centerline{\textit{Figure 13}}

\vskip 0.1in

Then the contribution to the sum from the chord in the left picture is 
$$(r+r')(a-(b-(r+r'))=(r+r')^2+(r+r')(a-b)$$
while the contribution from the two chords on the right is
$$r(a-(b-r))+r'(a-r-(b-r-r'))=r^2+r'^2+(r+r')(a-b)$$
Note that the contributions from all other chords are unchanged.
In going from formula (1) to formula (3) above we eliminate copies of $r^2$ and both sides and then justify that the remaining sum on the left of the equation is $0$. But this means that in the specific situation above we subtract $(r+r')^2$ for the above sum and specific chord and $r^2+r'^2$ for the doubled chord. The result is the same in each case as our computation shows. Because for the chord colored $r+r'$ we know from before that the sum is $0$ it will also follow for the doubled chord.   

\vskip 0.1in

Next, adding chords attached to single circles, so corresponding to self-crossings in the part arc-graph, we first 
attach a chord just in the special way:

\begin{center}
\includegraphics[width=0.35\textwidth]{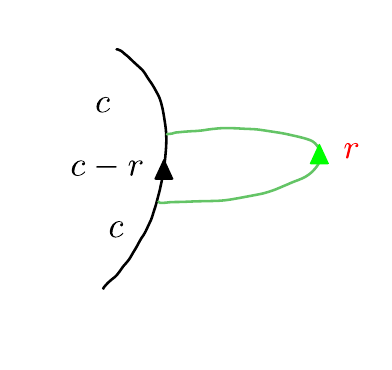}
\end{center}

\vskip -0.3in

\centerline{\textit{Figure 14}}

\vskip 0.1in

Note that this will change $\sum r(a-b)$ by a contribution $r(c-(c-r))=r^2$ so the result is true in this case. 
It suffices at this point to show that the relation remains true under arbitrary switches in the way the chords are attached to the circles. It will suffice to consider the following case of two consecutive chords. There will other chords attached to the circles not shown.

\vskip 0.2in

\begin{center}
\includegraphics[width=0.45\textwidth]{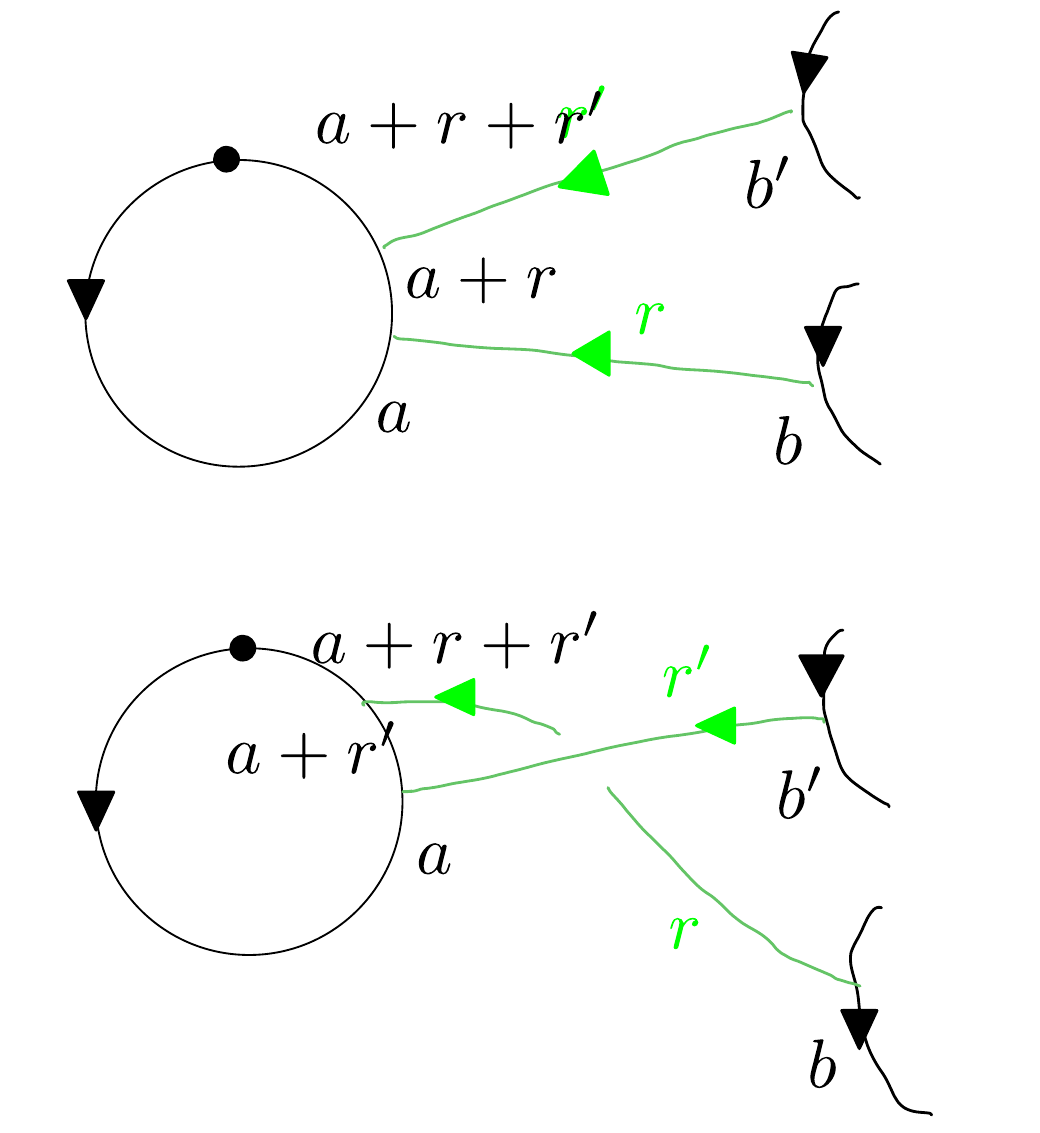}
\end{center}

\vskip -0.2in

\centerline{\textit{Figure 15}}

\vskip 0.1in

The contribution to $\sum r(a-b)$ for the above picture is
$$r(a+b-r)+r'(a+r+r'-b')$$
and after switching the endpoints of chords:
$$r'(a+r'-b')+r(a+r+r'-b),$$
which in each case is $r^2+r'^2+rr'+r(a-b)+r'(a-b')$, so the sum is not changing. But any possible constellation of chords joining distinct components of the chord graph can be reached by sequences of those switches.
\end{proof}

\vskip 0.1in

\begin{lem} Let $A=(a_{jk})$ be an $\mu \times \mu $ skew-symmetric matrix (over $\mathbb{Z}$ (the argument works in fact over any commutative ring with $1$) such that for $A=(a^1,\ldots ,a^{\mu})$ we have the sum of column vectors $\sum_{j=1}^{\mu }a^j$ is the zero-vector.
Define for each $j<k$,
$$s_{jk}=\sum_{1\leq i<j}a_{ik}-\sum_{k<i\leq \mu}a_{ji}.$$
Then 
$$(*) \qquad \qquad \sum_{1\leq j<k\leq \mu}a_{jk}s_{jk}=0$$
\end{lem}

\begin{proof} The set $\mathcal{A}_{\mu }$ of skew symmetric $\mu \times \mu $ matrices  with zero column sum is a submodule of the module of all skew symmetric matrices. Note that the skew symmetry implies the zero row sum condition. It is easily proved by induction that the following set of elementary matrices $E_{jk}$ defined for $1\leq j<\mu$ and $j<k\leq \mu $ is a generating set of $\mathcal{A}_{\mu}$. Let $(j,k)$ be the position of the upper left entry of the $2\times 2$-matrix 
$$\begin{pmatrix} 1 &  -1 \\ -1 & 1\end{pmatrix}\ \textrm{for} \ k>j+1$$
in a skew symmetric matrix with only the entries for $(j,k), (j,k+1), (j+1,k), (j+1,k+1)$ above the diagonal are non-zero and given by the $2\times 2$ matrix, and the entries below the diagonal are such that the matrix is skew-symmetric. Similarly,  we position the $2\times 2$-matrix
$$\begin{pmatrix} 1 & -1 \\ 0 & 1\end{pmatrix} \ \textrm{for}\ k=j+1$$
above the diagonal with the $0$ on the diagonal, and define the entries below the diagonal to get a skew-symmetric matrix. For example, when $\ell =3$ there is only the matrix
$$E_{23}=\begin{pmatrix} 0 & 1 & -1 \\ -1 & 0 & 1 \\ 1 & -1 & 0 \end{pmatrix},$$
corresponding to $\dim \mathcal{A}_{\mu }=1$. Note that for $\mu =0,1$, $\mathcal{A}_{\mu }=\{0\}$. For $\mu >3$ we can proceed inductively to prove that the matrices $E_{jk}$ generate:
$$A=a_{12}E_{12}+(a_{12}+a_{13})E_{13}+\cdots +(a_{12}+\cdots +a_{1,\mu -1})E_{1,\mu -1}+A',$$
with $A'\in \mathcal{A}_{\mu }$ a matrix with first row and column $0$ because $\sum_{i=2}^{\mu } a_{1i}=0$. To prove the lemma it thus suffices to show that for each $A\in \mathcal{A}_{\mu }$ satisfying $(*)$, als $B=A+E_{jk}$ satisfies $(*)$. It suffices to consider $A+E_{jk}$, the result for $A-E_{jk}$ is obvious because if $A$ satisfies $(*)$ also $-A$ satisfies $(*)$. For $k>j+1$ this is easily checked. Note that the coefficients of $B$ are only changed in the $2\times 2$-submatrix with upper left position at $(j,k)$, and there it will be
$$\begin{pmatrix} a_{jk}+1 & a_{j,k+1}-1 \\ a_{j+1,k}-1 & a_{j+1,k+1}+1\end{pmatrix}$$
Contributions in $(*)$ for $B$ from $a_{i\ell }$ with $\ell <k$ and $i=j,j+1$, or $i>j+1$ and $\ell =k,k+1$ will not change because $\pm 1$ cancel out in $s_{i \ell}$. So it suffices to consider how $(*)$ is changing considering the products of 
$b_{i\ell }$ with the corresponding $s_{i\ell }$ and $i\in \{j,j+1\}, \ell \in \{k,k+1\}$. 
Finally, new contributions in the products $(a_{j,k}+1)(-(a_{j,k+1}-1))$, $(a_{j+1,k}-1)(-(a_{j+1,k+1}+1))$, $(a_{j+1,k}-1)(a_{jk}+1)$ and $(a_{j+1,k+1}+1)(a_{j,k+1}-1)$ will cancel as is easily checked. The case $k=j+1$ is similar but involves both the zero column sum and skew symmetry. In order to avoid too many index consideration we will discuss the case $\mu =5$ and $E_{23}$ because it shows all necessary considerations. In this case the upper diagonal part of the matrix $B$ is given by
$$\begin{pmatrix} 0 & a_{12} & a_{13} & a_{14} & a_{15} \\  \ & 0 & a_{23}+1 & a_{24}-1 & a_{25} \\ \ & \ & 0 & a_{34}+1 & a_{35} \\ \ & \ & \ & 0 & a_{45} \\ \ & \ & \ & \ & \ 0 \end{pmatrix}$$
In this case multiplying out just terms in $a_{jk}$ and $s_{jk}$ within the changed $2\times 2$-submatrix gives 
$$(a_{23}+1)(-(a_{24}-1))=-a_{23}a_{24}+a_{23}-a_{24}+1$$
$$(a_{34}+1)(a_{24}-1)=a_{34}a_{24}-a_{34}+a_{24}-1$$
Note that the $\pm a_{24}$ and $\pm 1$ cancel out, the $-a_{23}a_{24}$ and $a_{34}a_{24}$ are in the $(*)$-summation for $A$, so there is a contribution of $a_{23}-a_{34}$. But the $+1$ in the upper left contributes an additional $a_{13}$, the $+1$ in the lower right an additional $-a_{35}$. So the change in $(*)$ is given by
$$a_{13}+a_{23}-(a_{34}+a_{35})=0,$$
which follows by considering the zero row sum relation in column $3$. The general argument is completely analogous with $a_{13}$ and $a_{35}$ contributions replaced by corresponding sums. 
\end{proof}

\begin{prop} \label{prop=2} For each $\mathbb{Z}$-coloring of a signed chord graph, respectively the corresponding coloring of a part arc-graph inducing this chord graph coloring,
$$\frac{1}{2}\sum_c\varepsilon (b+a-r)=\sum_c\varepsilon b,$$
or equivalently
$$\sum_c\varepsilon (a-b)=\sum_c\varepsilon r.$$
\end{prop}

\begin{proof} We will consider the equivalent statement for states on signed chord graphs. The argument will proceed as follows: First the claim is proved for part arc-graphs without chords joining distinct circles, but allowing distinct components. This corresponds to considering the horizontal multiplication of braids. Then we will argue that for general chord graphs we can assume for components $j>i$  that at all crossings of $j$ and $i$ the component $j$ will overcross component $i$. This corresponds to a change of the sign of the crossing. Finally we will argue that the equation to be proved remains valid under braid moves corresponding to specific Reidemeister II and III moves. This will extend the claim to the general link case. We will work in the free abelian group generated by variables $r$ representing the chord colors. Then we can define an element in the free abelian group for each circle segment representing a part-arc. 
The following figure shows a component with a self-chord colored $r$ and sign $\varepsilon $. We chose a specific orientation of the chord but the computation is similar if the orientation is reversed. 
\vskip 0.1in

\begin{center}
\includegraphics[width=0.7\textwidth]{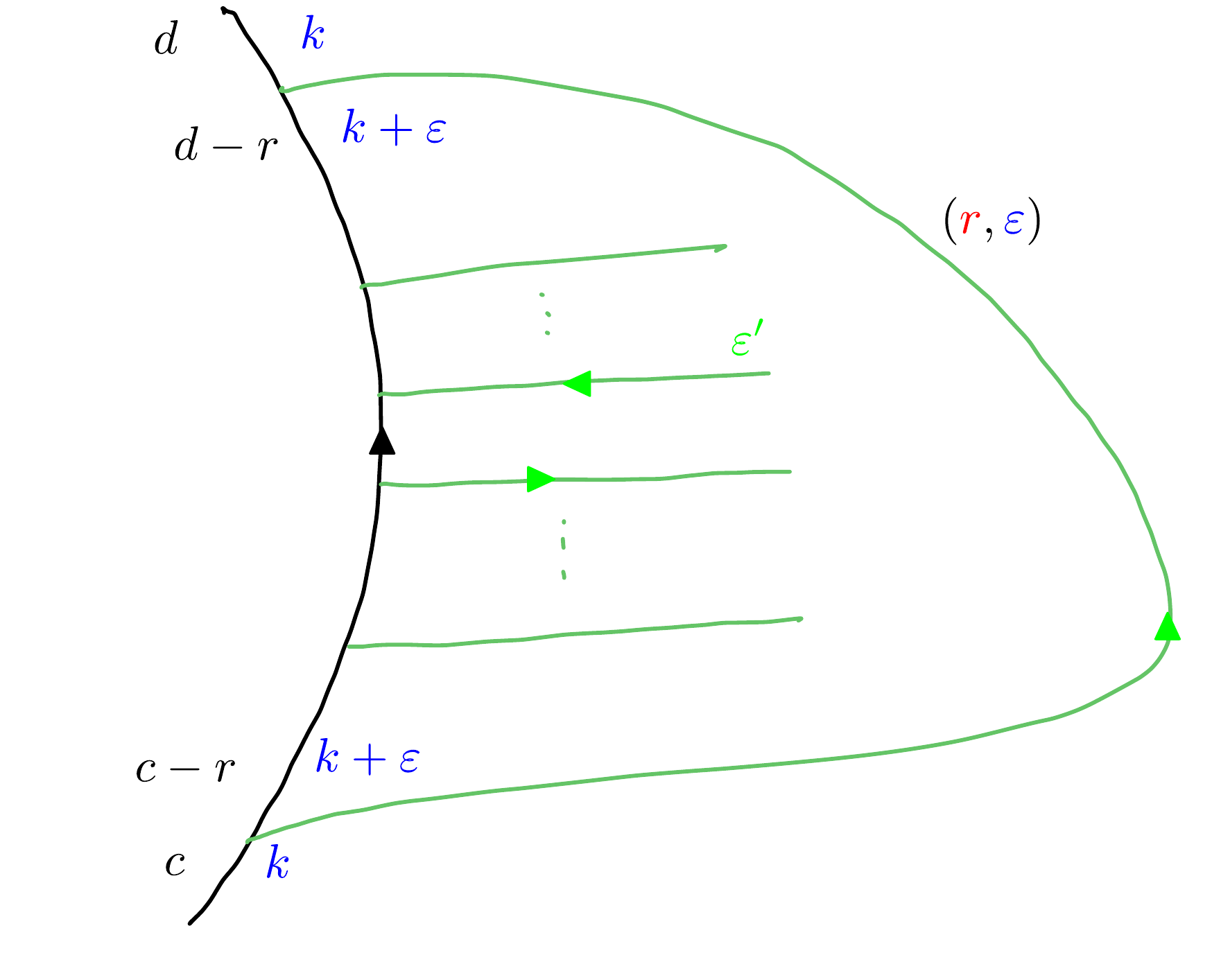}
\end{center}

\vskip -0.2in

\centerline{\textit{Figure 16}}

\vskip 0.1in

\noindent The contribution in the left sum for the chord labeled $r$ is $d-(c-r)=r+(d-c)$, with $c,d$ not containing any $r$, so the contribution of $r$ in $\sum \varepsilon (a-b)$ is $r$. Now the variable $r$ is carried along the segments following the orientation of the circle in the chord graph. But note that each incoming positive/outgoing negative chord will both contribute $+r$ and move the horizontal position to the right. In the same way each outgoing positive/incoming negative chord will contribute $-r$ while moving to the left. Because we start and end at $k+\varepsilon $ these contributions of $r$ will cancel. This proves the result in the knot case or more generally in the case of arbitrary braid closures where different components have no crossings. Please check some of the arguments given in the example tbelow to see e.\ g.\ that the choice of $\beta $ in the potential does not change the equation. 
In order to discuss the non-trivial link case note that we can always change crossings and assume that component $j$ always overcrosses component $i$ for $j>i$ (see the argument in section 4). It will suffice to show that the equation is preserved under special Reidemeister II braid moves involving distinct components, and arbitrary Reidemeister III braid moves. Consider the following Reidemeister II move of distinct components. 

\begin{center}
\includegraphics[width=0.9\textwidth]{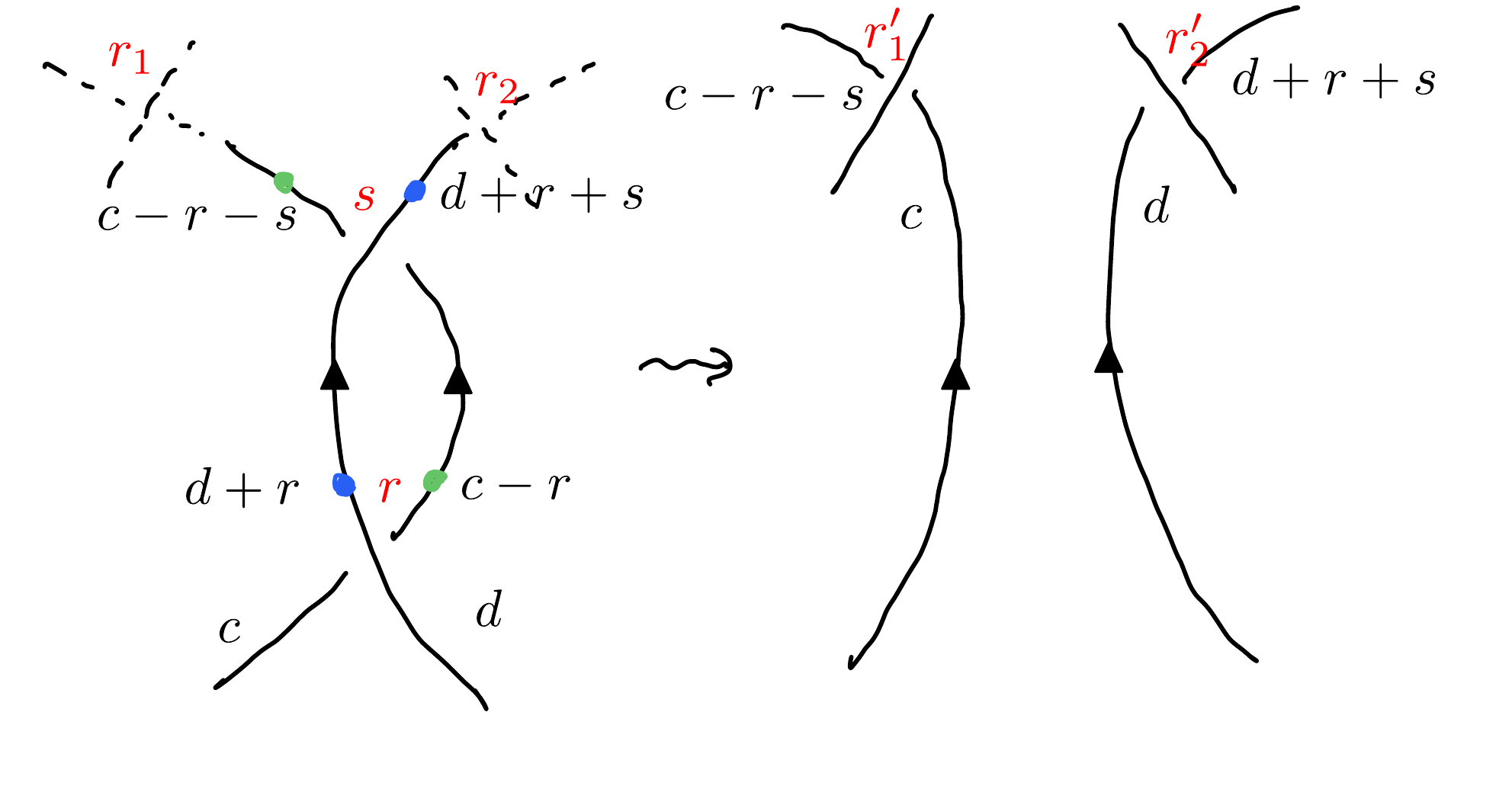}
\end{center}

\vskip -0.2in

\centerline{\textit{Figure 17}}

\vskip 0.1in

\noindent If one of the strands is not included at any further mixed crossings then the corresponding cycle relation is $r+s=0$. In the figure we have a bottom negative crossing and the top crossing is positive. If we calculate $\sum \varepsilon r$ the contribution will be $-r+s=2s=-2r$. The contribution in $\varepsilon (a-b)$ is 
$$-[(d+r)-(c-r)]+(d-c)=-2r.$$
If both components of the Reidemeister II move are included in crossings with distinct components then we can change $r_1,r_2$ to $r_1',r_2'$ such that the equation will hold for the new diagram. We ignore the consideration of any self-crossings here. Note that the signs at the crossing slabeled $r_1, r_2$ does not matter. Again it can be checked that the equation will hold for the right diagram if and only if it holds for the left diagram. The discussion of Reidemeister III moves is in fact easier because all what has to be made sure is that the positions of $r_1,r_2,r_3$ in a diagram before a move and after the move are chosen correctly so that the contributions coincide. The remaining details are left to the reader. 
\end{proof}

\begin{exmp} 
Figure 18 shows the projection of the long component $2$-component link closure for $\beta =\sigma_1\sigma_2^{-2}\sigma_1^{-1}\sigma_2$.
\begin{center}
\includegraphics[width=0.6\textwidth]{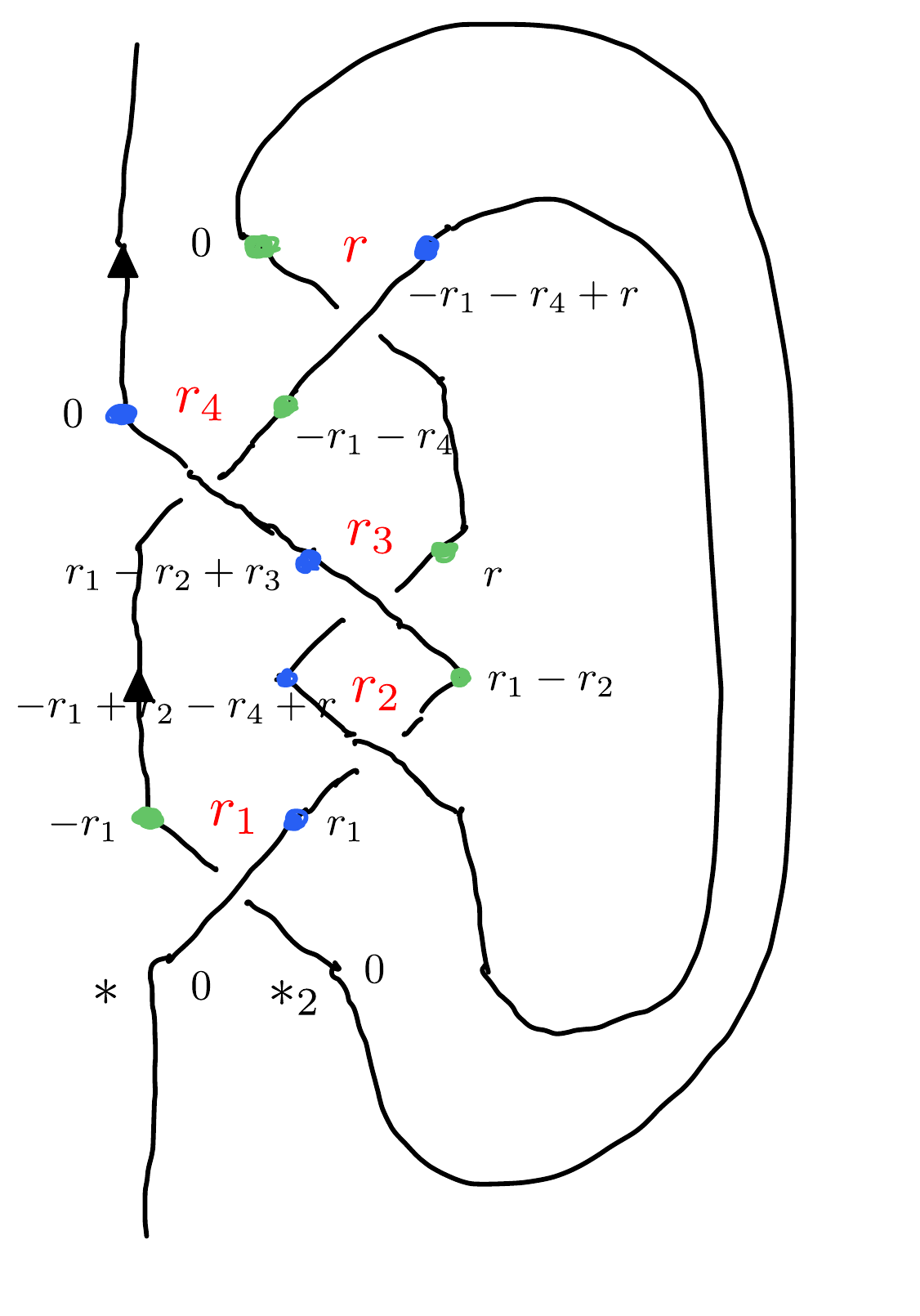}
\end{center} 

\vskip -0.2in

\centerline{\textit{Figure 18}}

\vskip 0.1in

The cycle relation is $r_1-r_2+r_3+r_4=0$. There are four crossings of the two components and one self-crossing of the second component. The base value of the part-arc at the base $*_2$ of the second component is set to $0$. It is easy to see that changing this value to $c\neq 0$ will not change the computation in the two sums considered in Proposition 6.2. Note that the signs at the five crossings from bottom to top are $(+,-,-,-,+)$. There will be a contribution of $c$ in $\sum \varepsilon (a-b)$ at each mixed crossing at which the second component crosses from the right and a $-c$ if it crosses from the left. In fact, the sign of $\pm c$ does not depend on the sign of the crossing. If we change the sign of a crossing then the roles of $a$ and $b$ in the computation change. At the self-crossing $c$ appears in both $a$ and $b$. So the reason is just that we \textit{always} can set $c=0$ is that the crossing number of the second component with the first component is even. 
We calculate 
\begin{align*}
\sum_c\varepsilon (a-b)=&+[r_1-(-r_1)] \\
&-[-r_1+r_2-r_4+r)-(r_1-r_2)]\\
&-[(r_1-r_2+r_3)-r]\\
&-[0-(-r_1-r_4)]\\
&+[(-r_1-r_4+r-0]\\
=&\ r_1-r_2-r_3+r_4+r\\
=&\ \sum_c \varepsilon r
\end{align*}
We were lucky in that we did not have to use the cycle relation again to get the correct result. Of course we already used it at the third crossing to simplify. Note that $r$ appears three times but the contributions from $a-b$ calculated at crossings different from the one colored $r$ cancel out.  
\end{exmp}
 
\vskip 0.1in

The statements in the proposition above are non-trivial. They apply to each state-sum contribution individually. Thus we get the following strengthened result of Garoufalidis and Loebl including the statement that the understanding that the two state-sums are equivalent in the sense of Definition \ref{defn: equivalency}. Now consider the arc-graph corresponding to $\mathcal{K}$ and the coloring $f$ from the set of edges into $\mathbb{N}_0$ corresponding to \ref{lem: flow lemma}.
The contributing states are given by restricting values to $f(v)\in \underline{n}$ for each vertex $v$ of $G\mathcal{K}$ (see \ref{defn:edge coloring}).

The quantity $-\frac{\varepsilon }{2}\sigma =bR$ is the \textit{local excess}, and the sum over all vertices is the \textit{excess} of the state.

The flow lemma identifies the coloring $b$ in Figure 10 with $f(e_v^b)$, and $f(e_v^r)$ is the jump $r$ in colors at a crossing $c$. Recall that overcrossing arcs are identified with vertices at their endpoints, so
$$\sum_c \varepsilon b=\sum_v \varepsilon (v) f(e_v^b).$$
Identifying $R=\sum_{e<e_v^r}f(e)$ we define
$$\textrm{exc}(f):=\sum_v\varepsilon (v)f(e_v^b)\sum_{e<e_v^r}f(e),$$
where we identity $R=\sum_{e<e_v^r}f(e)$ at each vertex. We use here that $R=a-r$.
Finally following the definition of rotation of a state $f$ in \cite{GL}, Definition B.6, we define
$$\textrm{rot}(f):=\sum_{i=2}^{\mathfrak{s}} \beta_i(s)=\sum_{e'\in T(e)}f(e').$$
Here $T(e)\neq \emptyset $ only for $e_v^b,e_v^r$ for vertices corresponding to braid generators 
at the top of the braid. For $e_v^b$ we have $T(e)=\{e_v^b\}$ while for $e_v^r$ we have 
$\{e: e<e_v^r\}\cup \{e_v^r\}$. In each case this determines the color of the part arc, which is a closure arc of the braid diagram. 
Recall that the $\beta_i$ are the colors on the part-arcs at the top (or bottom) of the braid-closure diagram.
See also the comments about rotation following the statement of Theorem 6.2.

Finally note that $\sum_v\varepsilon (v)=\omega (\mathcal{K})$ is the writhe of the projection. 
Let $\delta (f):=-(\textrm{exc}(f)+\textrm{rot}(f))$.
It follows now from equation \eqref{eq=vertex contribution} and \eqref{eq=Fterm} that each state has an addition \textit{total} multiplicative contribution of
\begin{equation} 
\delta '(\mathcal{K},n):=-\frac{n^2}{4}\omega (\mathcal{K})+\frac{n}{2}(\mathfrak{s}-1)
\end{equation}
Next recall that $f(v)=f(e_v^b)+f(e_v^r)$

For each braid projection $\mathcal{K}$ of a framed link $K$ and $\delta (\mathcal{K}), \delta (f)$ as above we let
$\mathcal{F}_{n}^*(G\mathcal{K})\subset \mathcal{F}_{n}(G\mathcal{K})$ be the set of $n$-flows, which will induce $0$ for the part arc framing of that part arc, which is the closure arc of the first braid string. This implies $f(e_w^b)=f(e)=0$ for $w$ the overcrossing arc preceding the overcrossing arc $v$ containing the bottom part arc of the first string. Here $e$ runs through all red edges preceding the $+$-marker on the fat vertex $v$ corresponding to $v$. If the closure part arc is on an overcrossing circle then the base value of this circle and all incoming red edge colorings are $0$. Now the following theorem has been proved:

\begin{thm} (see \rm{\cite{GL}} for $0$-framed long knot projections) \label{thm:main}
The framed colored Jones polynomial can be calculated from a braid projection $\mathcal{K}$ by the following state-sum:
\begin{align}\label{eq=main formula}
\begin{split}
J_n'(K)&=t^{\delta '(\mathcal{K},n)}\sum_{f\in \mathcal{F}_{n}^*(G\mathcal{K})}t^{\delta (f)}
\prod_{v}t^{n\varepsilon (v)f(e_v^b)}{f(v) \choose f(e_v^b)}_{t^{-\varepsilon (v)}}\cdot \\ &
\cdot \{n-\sum_{e<e_v^r}f(e)\}_{f(e_v^r),t^{\varepsilon (v)}}
\in \mathbb{Z}[t^{\pm 1/4}]
\end{split}
\end{align}

\noindent Moreover, each state-sum contributions for this state-sum corresponds to the state-sum contribution for the state-sum defined for $\mathcal{P}\mathcal{K}$ defined in \cite{G} under the bijection of states defined in the flow lemma \ref{lem: flow lemma}.  Thus the  $R$-matrix state model \rm{\cite{GL}} is equivalent to the state models on arc graphs \rm{\cite{GL}} as defined in 
\ref{defn: equivalency}.  \hfill{$\square$}
\end{thm}

For simplicity we have given the proof of the Garoufalidis-Loebl formula \cite{GL} for braid projections. The proof immediately generalizes to all Morse projections of links. The main difficulty is the extension of rotation numbers to Morse projections following \cite{T}, see also \cite{GL}, Definition B.6. Each Morse projection defines a \textit{Gauss map} on $\mathcal{K}$ by the unit tangent vectors of the immersions along the components of $\mathcal{K}$. We know that the value $1\in S^1$ is a regular value with preimages the local maxima and minima. The Gauss map restricts to maps on part arcs and defines the \textit{rotation numbers} in $\{0,\pm 1\}$ for each part arc with $\pm 1$ according to right oriented maxima respectively minima \cite{W}. If the rotation number of a part arc is $\neq 0$ then we will indicate this by $\pm $ in the boundary of the fat vertex corresponding to the overcrossing arc containing that part arc. Note that there can be several part arcs with nontrivial rotation number on a given overcrossing arc. The inclusion of only right oriented extrema comes from 
the more general ribbon functor definition of $\textrm{sl}(2,\mathbb{C})$ because those correspond to the \textit{insertion} of right oriented evaluation and coevaluation maps. For each state on $G\mathcal{K}$, or $P\mathcal{K}$, the part arcs carrying nonzero rotation numbers will contribute factors of $t$ exponentiated by the part arc color for this state. In order to prove the formula \eqref{eq=main formula} we first consider special Morse projections for which at crossings all strings are directed upwards.
Special Morse projections have additional properties with respect to extrema, e.\ g.\ a right going maximum is immediately followed by a right or left directed minimum on the same part arc. Which situation occurs is given by the rotation number of that part arc. 
So the rotation data controls left-right position on a grid, which is necessary to generalize the proof of Proposition \ref{prop=2} beyond the case of braid projections. Finally we use that every Morse projection easily can be isotoped by planar isotopy into a special Morse projection using local rotations at crossings. It is easy to see that the state-sum contributions will not change under those rotations. 

\begin{rem}
\noindent (a) For $0$-framed links $K$, $J_n'(K)\in \mathbb{Z}[t^{\pm 1/2}]$. This follows easily by the change of variables in Examples 1 and 2 above. Recall that for each framed link $K$ and $r\in \mathbb{Z}$ we let $K(r)$ denote
the link with framing changed by $r$ twists. Then we have
$$J_n'(K(r))=t^{-r(\frac{n^2}{4}+\frac{n}{2})}J_n'(K)$$
For a framed link $K$ with projection and corresponding writhe $\omega (\mathcal{K})$ the writhe is the total framing, and 
$$J_n'(K(-\omega (\mathcal{K}))=t^{\omega (\mathcal{K})(\frac{n^2}{4}+\frac{n}{2})}J_n'(K).$$
The unframed (or better framing corrected) version of the colored Jones polynomial thus is computed from an arbitrary projection by the relation
$$J_n(K):=t^{\omega (\mathcal{K})(\frac{n^2}{4}+\frac{n}{2})}J_n'(K).$$
If we define
\begin{align*}
\delta (\mathcal{K},n)&:=\delta '(\mathcal{K},n)+\omega (\mathcal{K} )(\frac{n^2}{4}+\frac{n}{2})\\
&=\frac{n}{2}(\omega (\mathcal{K})+\mathfrak{s}-1))
\end{align*}
and replace $\delta '(\mathcal{K},n)$ by $\delta (\mathcal{K},n)$ we get the state-sum formula for the unframed colored Jones polynomial from the theorem. It follows that $J_n(K)\in \mathbb{Z}[t^{\pm 1/2}]$. 

For $n=1$ we have the skein relations for framed links:
$$t^{-1/4}J'(K_+)-t^{1/4}J(K_-)=(t^{-1/2}-t^{1/2})J(K_0)$$
and the corresponding writhe normalized for the unframed version:
$$t^{-1}J(K_+)-tJ(K_-)=(t^{-1/2}-t^{1/2})J(K_0).$$

This is not the standard relation for the Jones polynomial derived from the Kauffman bracket but differs by a relative sign of the two sides. It corresponds to the relation in \cite{GL}, Appendix B, and \cite{RT}, Theorem 4.2.1. Because the number of components of $K_{\pm}$ and $K_0$ differs by v$1$ modulo $2$ for each skein triple it follows that the corresponding invariants only differ by sign. 

In fact, for the standard skein relation the Jones polynomials of even number component links carry an additional $-$ sign if the normalization on the unknot is kept. Since the colored Jones polynomial can be calculated from $n$-parallels the sign correction will only occur for $n$ odd (like in the classical case).

In fact, the mod $2$ number of components is nicely related to $\delta (\mathcal{K},n)$ as follows: Let $K=\hat{\beta }$ and $|K|$ the number of components of $K$. We assume that $\beta $ is a $\mathfrak{s}$-braid. Note that the mod $2$ number of transpositions in the permutation carried by $\beta $ is equal to $\omega (\mathcal{K})$, while the number of cycles in the disjoint cycle decomposition is $K$. Let $\ell_i$ be the length of the $i$-th cycle and recall that each cycle of length $\ell_i$ can be written with $\ell-i-1$ transpositions. Thus we have the equations:
$$\sum_{i=1}^{|K|}(\ell_i-1)\equiv \omega (K), \quad \sum_{i=1}^{|K|}\ell_i=\mathfrak{s}$$
It follows that if $n$ is even, or $|K|+1\equiv \omega (\mathcal{K})+\mathfrak{s}-1\equiv 0\ \textrm{mod}\ 2$, or the number of components of $K$ is odd, then $\delta (\mathcal{K},n)\in \mathbb{Z}$ and it follows that 
$J_n(K)\in \mathbb{Z}[t^{\pm 1}]$. On the other hand, if $n$ and $|K|$ are odd then $J_n(K)\in t^{1/2}\mathbb{Z}[t^{\pm 1}]$.
This coincides with the skein relation deduced of $J=J_1$ above. 

\noindent b) For $sl(N+1)$ in the fundamental representations the quantum invariants determine the HOMFLY polynomial. In this case the \textrm{st}-condition translates to flows on $G\mathcal{K}$ satisfying the condition: If $f(e_v^r)\neq 0$ then $f(e_v^b)=\sum_{e<e_v^r}f(e)$ for $N$-flows.

\noindent (c) Let $\overline{K}$ be the reflected link and $\overline{\mathcal{K}}$ be correspondingly the reflected projection. 
Then $G_+\overline{\mathcal{K}}=\overline{(G_-\mathcal{K})}$, where the bar over the graph denotes switching the signs of the vertices.

Using this and $J_n(\overline{K})(t)=J_n(K)(t^{-1})$, it is easy to derive a formula calculating $J_n(K)$ from $G_-\mathcal{K}$.
It should be noted that the graphs $G\mathcal{K}$ and $G(\overline{\mathcal{K}})$ are usually very different as oriented graphs 
and thus give rise to very different state-sums calculating the same link polynomial, up to $t\mapsto t^{-1}$.
A trivial component contributes $t^{1/2}+t^{-1/2}$ for our skein relation but $-(t^{1/2}+t^{-1/2})$ in the standard case. 

\noindent (d) It should have become clear by now that the arc graph is just an equivalent representation of exactly the data necessary to be taken from the part arc-graph necessary to define either of the two state-sums. In fact, after labeling the overcrossing arcs in cyclic way all what is needed are the jumps at the crossings for $f(e_v^r)$ and the part-arc colors of outgoing undercrossing part arcs for $f(e_v^b)$. The real assertion of formula 6.2, or its generalization to Morse projection, is contained in Propositions 6.1 and 6.2, which are generalizations of Lemma B.5 in \cite{GL} or Lemma 2.1-2.4 in \cite{Li}. The main arguments for $n=1$ here are already contained in \cite{T}, sections 5 and 6. It seems that the \textit{integration} of certain state-sum contributions simplifies the computation using formula 6.2 over the usual $R$-matrix formula. The remaining steps are just rewriting of quantum algebra expressions.
\end{rem}

\section{Evaluation of the Garoufalidis-Loebl state-sum}

Note that the arc-graphs are \textit{abstract} enhanced graphs representing the data necessary to calculate the state-sum \eqref{eq=main formula}. For explicit computations it will be helpful to present the data in non-geometritc way so that it can easily be implemented as an algorithm.

Suppose that $G\mathcal{K}$ is defined from a Morse projection $\mathcal{K}$ without overcrossing circles. The case including overcrossing circles can easily be described but requires to consider vertices corresponding to the circles separately. For the moment we will work with non-reduced graphs.

The enhanced arc-graph $G\mathcal{K}$ can be encoded by a \textit{graph sextuple} $G'\mathcal{K}:=(V,\varepsilon ,\sigma ,\tau ,\mathfrak{o}, \mathfrak{r})$. Here $V$ is the set of vertices of $G\mathcal{K}$ and $\varepsilon $ is the function assigning the crossing sign to each vertex (i.\ e.\ overcrossing arc) of $G\mathcal{K}$.  At each vertex of $G\mathcal{K}$ we have a unique outgoing blue edge $e_v^b$ and red edge $e_v^r$. Let $\sigma (v)$ respectively 
$\tau (v)$ be the endpoints of those directed edges in $G\mathcal{K}$. This defines maps $\sigma ,\tau : V\rightarrow V$, with $\sigma $ a bijection. 
We call $\sigma $ the \textit{undercrossing map} and $\tau $ the \textit{jump map}. 

\begin{rem} $\tau $ is a bijection if and only of $\mathcal{K}$ is an alternating projection, in particular $\tau ^{-1}(v)\neq \emptyset $ for each vertex $v$ in this case, and the inverse map is defined.  
\end{rem}

Next, $\mathfrak{o}=(\mathfrak{o}_v)_{v\in V}$ is a total ordering $\mathfrak{o}_v$ on each of the sets $\tau^{-1}(v)$ for each $v\in V$. Note that the order $\mathfrak{o}_v$ also defines a total ordering on the set of all incoming edges of the fat vertex $v$ of $G\mathcal{K}$
with the convention that $e_{\sigma ^{-1}(v)}^b$ precedes $e_w^r$ for each $w\in \tau^{-1}(v)$.
 Also note that $V=\sqcup_v \tau^{-1}(v)$ is a disjoint union.
This is the ordering, which is used in the Pochhammer symbol in \eqref{eq=main formula}.
 If $\tau^{-1}(v)=\emptyset $ then 
the overcrossing arc corresponding to $v$ is a part arc. But if $\tau^{-1}(v)\neq \emptyset $ then a choice of element $w\in \tau^{-1}(v)$ determines the unique part arc on the boundary of the \textit{fat} vertex $v$ following the edge $e_w^r$, and thus a unique part arc of $P\mathcal{K}$. 

Finally, $\mathfrak{r}$ is given by two functions $\mathfrak{r}_{\sigma },\mathfrak{r}_{\tau }: V\rightarrow \{-1,0,1\}$. These are the sums of the rotation numbers of all part-arcs following the transitions $\sigma ,\tau $ from the vertex $v$. So for $\sigma $ this is the rotation number of the overcrossing arc $\sigma (v)$, for $\tau $ this is the sum of the rotation numbers of all part arcs following $v\in \tau^{-1}(\tau (v))$ in the order of $\tau^{-1}(\tau (v))$. Consider the left part in Figure 19 below, with $v$ replaced by $\tau (v)$. 
A non-decorated part arc will have rotation number $0$. See \ref{defn:arc-graph}.

\begin{rem} (a) Each part arc-graph or arc-graph determines a quintuple as above. But not each such quintuple will determine an arc-graph induced from a part arc-graph. 

\noindent (b) The unique cycle decomposition of $\sigma $ will determine the number of components of $\mathcal{K}$. Let $\mu $ be the number of cycles in this decomposition. 

\noindent (c) We let $e_{\tau^{-1}(v)}:=\{e_w^r:\tau (w)=v\}$, and this will be $\emptyset $ if $\tau^{-1}(v)=\emptyset $.

\noindent (d)  Note that at this point we will \textbf{not} work with any particular order on the vertex set $V$. 

\end{rem}

Suppose that the cycle length of the $i$-th cycle is $\mathfrak{c}(i)$ for $i=1,\ldots ,\mu $. Note that 
$\mathfrak{c}=\sum_{i=1}^{\mu }\mathfrak{c}(i)$ is the number of elements in $V$, and we let $V=\sqcup_{\ell =1}^{\mu }V_{\ell }$ be with $V_{\ell }$ be the set of vertices corresponding to the overcrossing arcs of the $\ell $-th component. 
In order to state the cycle equations from $G'\mathcal{K}$ we need to identify the overcrossings and undercrossings of the $\ell $-th component with all other components. In fact
$$\bigcup_{k\neq \ell }\bigcup_{v\in V_{\ell }} \tau^{-1}(v)\cap V_k$$
will be the set of vertices corresponding to overcrossings of some  overcrossing arc in $V_{\ell }$ with distinct components. 
Note that these are not vertices in $V_{\ell }$. 
Similarly
$$\{v\in V_{\ell }: \tau (v)\notin V_{\ell }\}$$
corresponds to the set of all crossings at which the component $\ell $ undercrosses a distinct component.
 
We will next describe how to use the data $G'\mathcal{K}$ to construct the set of contributing states on $G\mathcal{K}$, which we know in natural one-to-one correspondence with the contributing states on $P\mathcal{K}$ by the flow lemma \ref{lem: flow lemma}.
Next recall that a potential will be an ordering of the components and a choice of vertex $v^i$ for $i=1,\ldots \mu$ on each cycle of the cycle decomposition of $\sigma $, a function $\mathbf{j}: V\rightarrow \underline{n+1}$, and values $\beta_i\in \underline{n+1}$ for $i=1,2,\ldots ,\mu$. We assume that the function $\mathbf{j}$ satisfies the cycle equations. For a state $f$ of $G\mathcal{K}$ we use the notation $j_v:=f(e_v^r)$, $i_v:=f(e_v^b)$, $j_{\tau^{-1}(v)}:=\sum_{w,\tau (w)=v}f(e_w^r)$.
Then Figure 19 \label{fig=19} shows the local picture at each vertex: 

\begin{center}
\includegraphics[width=1.0\textwidth]{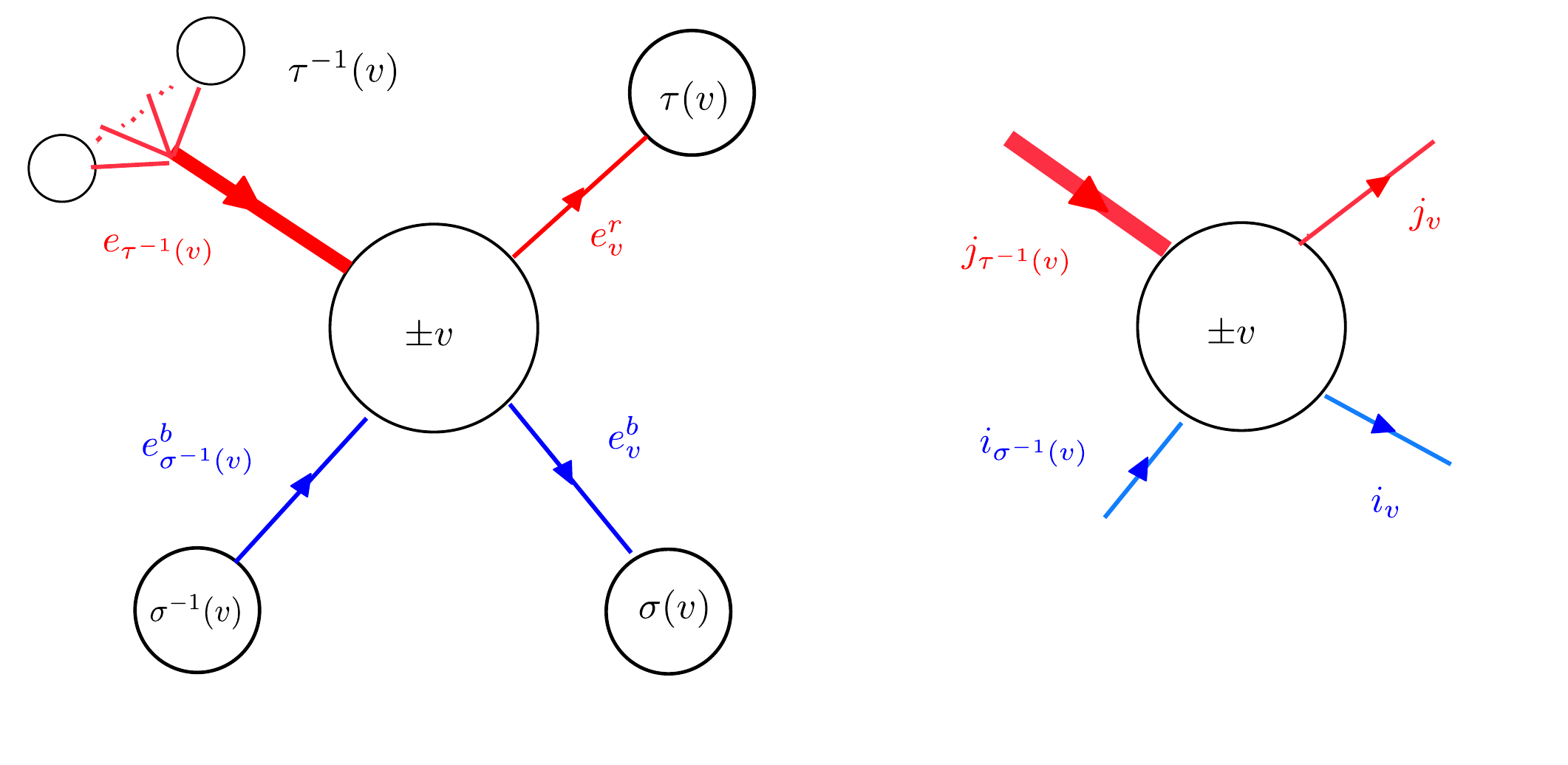}
\end{center} 

\vskip -0.2in

\centerline{\textit{Figure 19}}

\vskip 0.1in

The flow equation using the above notation now is
\begin{equation}\label{eq=local 1}
i_v+j_v=i_{\sigma ^{-1}(v)}+j_{\tau ^{-1}(v)}
\end{equation}

This is a recursion to calculate the state colorings from the potential, for each cycle of $\sigma $ at a time. 
We show how to use the recursion to find the function $\mathbf{i}: V\rightarrow \mathbb{N}_0$ in the cycle case.
Below we will show for a $2$-component example how to find the function $\mathbf{i}$ from the function $\mathbf{j}$ with base values $(0,\beta )$ for the two components. 

Now suppose $\mu =1$ and let $\mathfrak{c}=|V|$ be the number of crossings of $\mathcal{K}$. Then for each $v\in V$ let $k_v\in \underline{\mathfrak{c}}$ be the unique number such that $\sigma ^{-k_v}(v)=v^1$. Suppose that $i_{v^1}=\beta_1$ is given.  
Then 
\begin{equation}\label{eq=solved recursion}
i_v=i_{v^1}+\sum_{k=0}^{k_v-1}j_{\tau^{-1}\sigma^{-k}(v)}-j_{\sigma^{-k}(v)}
\end{equation}
The recursion will start with equation \eqref{eq=local 1}:
$$i_v=i_{\sigma^{-1}(v)}+j_{\tau^{-1}(v)}-j_v$$
and then substitute the recursion with $v$ replaced by $\sigma ^{-1}(v)$ to substitute for $i_{\sigma ^{-1}(v)}$:
$$i_v=i_{\sigma^{-2}(v)}+(j_{\tau^{-1}\sigma^{-1}(v)}-j_{\sigma^{-1}(v)})+(j_{\tau^{-1}(v)}-j_v)$$
After $k_v$ steps we have the solution for $v$. This solution is in fact compatible because if we apply to $v^1$ the recursion step 
$\mathfrak{c}$ times, recall that $\sigma^{-\mathfrak{c}}(v)=v$ since $\sigma^{\mathfrak{c}}=\textrm{id}$, then we get
$$i_{\sigma^{-\mathfrak{c}}(v^1)}=i_{v^1}+\sum_{k=0}^{\mathfrak{c}-1}j_{\tau^{-1}\sigma^{-k}(v^1)}-j_{\sigma^{-k}(v^1)}=i_{v^1}$$
This follows from
\begin{equation}\label{eq=cyclic}
\sum_{k=0}^{\mathfrak{c}-1}j_{\tau^{-1}\sigma^{-k}(v)}=\sum_{k=0}^{\mathfrak{c}-1}j_{\sigma^{-k}(v)}
\end{equation}
because $\{\sigma^{-k}(v):k=0,\ldots ,\mathfrak{c}-1\}=V$ and $V=\sqcup \tau^{-1}(v)$ as we observed above. 
Note that formula \eqref{eq=solved recursion} just formally states how to calculate the part arc colors from a potential, just following 
along a component. 
In the link case, the above discussion has to be done for each cycle of $\sigma $ separately, noting that the final argument uses the cycle conditions. In fact, in the process of computing the part arc color by following a component, overcrossings of distinct components are in the left sum while undercrossing a distinct component is in the right hand sum, and those contributions cancel for each component. Thus for each $v\in V$ such that $\sigma ^{-k_v}v=v^{\ell }$ with $\ell \in \{1,\ldots ,\mu\}$ we have 
\begin{equation}\label{eq=solved recursion link}
i_v=\beta_{\ell }+\sum_{k=0}^{k_v-1}j_{\tau^{-1}\sigma^{-k}(v)}-j_{\sigma^{-k}(v)}
\end{equation}
The contributing states are determined by the $n$-flow conditions \ref{defn:edge coloring} for colorings of $G\mathcal{K}$, and are
$$i_v,j_v\geq 0\ \textrm{and} \ i_v+j_v\leq n$$
for all $v\in V$. These conditions are equivalent, see the proof of \ref{lem: flow lemma} to the flows on $P\mathcal{K}$ being $n$-bounded, i.\ e.\ all part arc colors take values in $\{0,\ldots ,n\}$, see \ref{defn:edge coloring}.
Obviously the equations are satisfied if for all $v\in V$:
$$j_v\geq 0, i_{\sigma ^{-1}(v)}\geq 0 \ \textrm{and}\ i_{\sigma^{-1}(v)}+j_{\tau^{-1}(v)}\leq n,$$
because the part arc colors along the overcrossing arc $v$ are increasing from $i_{\sigma^{-1}(v)}$ to $i_{\sigma^{-1}(v)}+j_{\tau^{-1}(v)}$, then jump down to $i_v$. 
The last inequalities are equivalent to 
$$j_v, i_v\geq 0 \ \textrm{and}  \ i_v\leq n-j_v$$
So, using the recursion \eqref{eq=solved recursion link} we get the system of inequalities for each $v\in V$ in terms of the potential, for all $v^{\ell }\neq v\in V$, $\ell \in \{1,\ldots ,\mu\}$ depending on $v$,

\begin{equation}\label{eq=contributing}
\begin{aligned}
0&\leq j_v\leq \beta_{\ell }+j_{\tau^{-1}(v)}+\sum_{k=1}^{k_v-1}j_{\tau^{-1}\sigma^{-k}(v)}-j_{\sigma^{-k}(v)}\leq n \\
0&\leq j_{v^{\ell }},\beta_{\ell } \ \textrm{and}\ \beta_{\ell }+j_{v^{\ell }}\leq n \ \textrm{for the starting vertex}\ v^{\ell }
\end{aligned}
\end{equation}

Note that the part arc coloring preceding the base vertex on the $\ell $-th component is
$i_{\sigma^{-1}(v^{\ell })}+j_{\tau^{-1}(v^{\ell })}\leq n$ because $i_{\sigma^{-1}(v^{\ell })}+j_{\tau^{-1}(v^{\ell })}-j_{v^{\ell }}=\beta _{\ell }$ and the part arc color following the crossing $v^{\ell }$ is just $\beta_{\ell }$.

\begin{defn} A \textit{return loop} at the self-crossing $v\in V_{\ell }$ is a loop in $P\mathcal{K}$ starting at vertex $v$ on the overcrossing arc in the direction of the orientation, following the link until it returns to $v$. The number of overcrossing arcs except the one at which the loop starts is the number $k$ (modulo $\mathfrak{c}(\ell )$) in $\sigma^k\tau (v)=v \Longleftrightarrow v=\tau^{-1}\sigma^{-k}(v)$. For $k=0$ the return loop follows an overcrossing arc until its return to $v$. This is a circle in $P\mathcal{K}$, which could be eliminated by isotopy, 
\end{defn}

Figure 20 shows three return loops. 

\begin{center}
\includegraphics[width=1.0\textwidth]{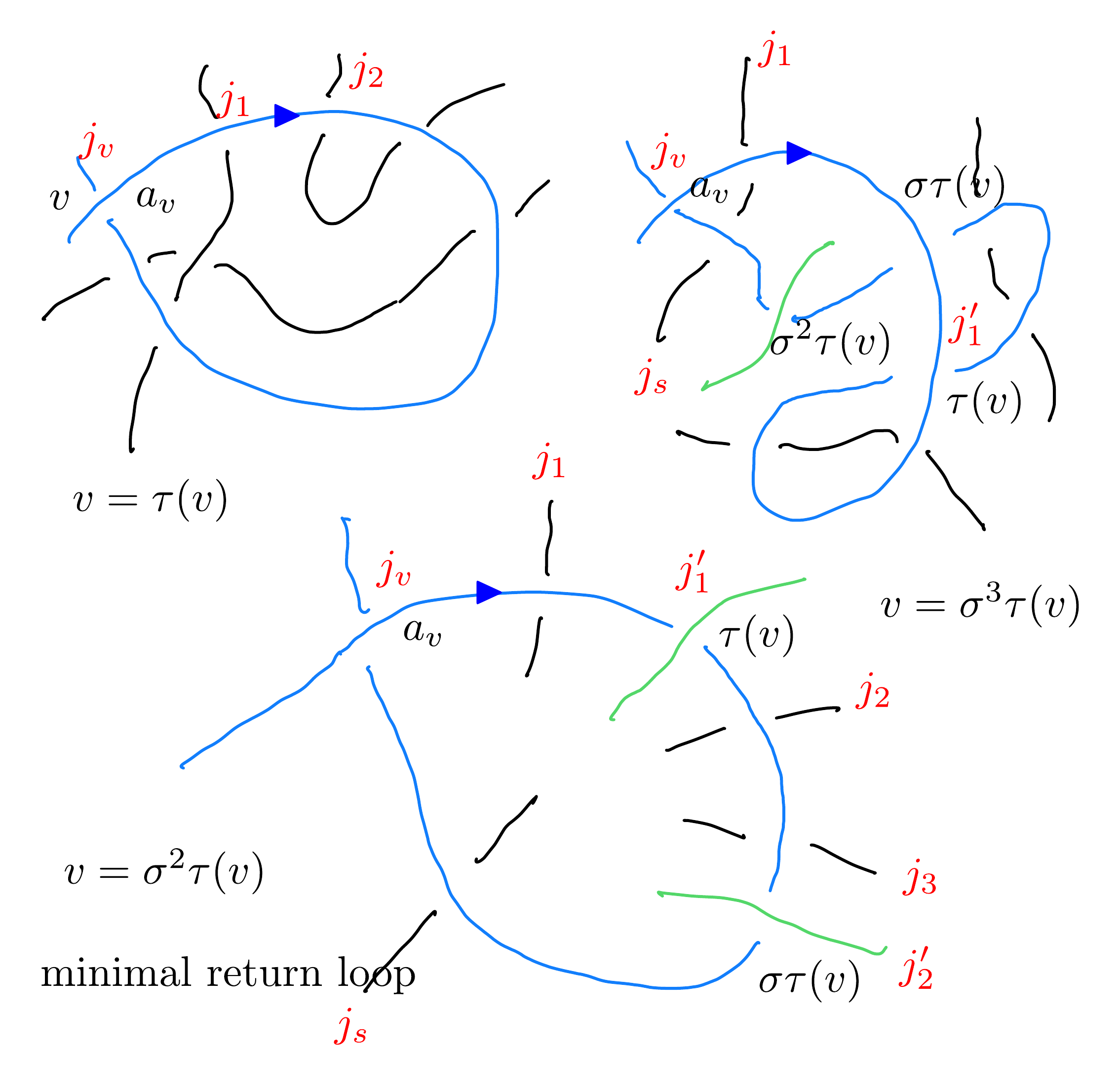}
\end{center} 

\vskip -0.2in

\centerline{\textit{Figure 20}}

\vskip 0.1in

Note that for $k=0$ the inequality
$$j_v\leq a_v+j_1+\ldots +j_s\leq n$$
holds. For $k>0$ such an equation is
$$j_v\leq a_v+j_1+\ldots j_s-(j_1'+\ldots j'_{k}\leq n$$
Of $j_s=j'_t$ can occur (like $j_2=j_1'$ in the right figure of Figure 20). 
Then there will be a shorter return loop contained in the original loop. 
A \textit{minimal return loop} is a loop, which does not contain shorter return loops. 
The system of inequalities for contributing states seems to contain geometric information about return loops
in the part arc graph. The study of the geometry of return loops reminds of the purely inductive arguments in \cite{LiMi} proving the existence of the Homflypt polynomial. 

\vskip 0.1in

In order to determine $\mathcal{F}_n^*G\mathcal{K}$ in \ref{thm:main} for a long component projection, note that the part arc color $0$ on the part arc corresponding to the long half arcs in general does not imply $\beta_1=0$ if we take for the base vertex on the first component, which we choose to be the long one, but in fact it will imply that the vertex preceding this bottom vertex has color $0$. This follows because overcrossing can only increase the colors. So depending on the projection this will imply that the jumps are zero up this part-arc. To simplify, and actual be able to work with the reduced graph we assume that the first crossing of the bottom long arc is an undercrossing and is the base vertex of the first cycle. This will imply $\beta_1=0$ because at the base crossing the value can only decrease. This implies $j_{v^1}=0$ and by the flow condition of course also that $i_{\sigma^{-1}(v^1)}=j_{\tau^{-1}(v^1)}=0$. Note that $j_{\tau^{-1}(v)}=0 \Longleftrightarrow j_w=0$ for all $w\in \tau^{-1}(v)$. But note that 
$i_w\neq 0$ is possible for $w\in \tau^{-1}(v)$ in this case (and thus also $i_{\sigma ^{-1}(w)}$ and $j_{\tau^{-1}(w)}$ can be nonzero for $w\in \tau^{-1}(v)$ in this case), so these vertices cannot be discarded. 
If we delete the base vertex on the long component including all edges of $G\mathcal{K}$ incident to that vertex we will not change state contributions because $0$-colorings factor in as multiplication by $1$. While computationally comfortable, some of the symmetry is lost because $\tau $ is no longer defined 
for the vertices in $\tau^{-1}(v^1)$. Also $\sigma $ will no longer be defined on $\sigma ^{-1}(v^1)$ if we delete the base vertex $v^1$. In all recursively defined formulas we will still include the base vertex $v^1$ so that the maps $\sigma ,\tau $ are always defined. 

Let
$V^*:=V\setminus (\{v^1\}\cup \tau^{-1}(v^1))$ and $V':=V\setminus \{v^1\}$. Then we can parametrize the set of contributing states as follows:
\begin{equation}\label{eq=contributing states}
\mathcal{F}_n^*G\mathcal{K}\leftrightarrow \{\mathbf{j}:=(j_v,\beta_{\ell }): v\in V^*,\ell =2,\ldots \mu, \eqref{eq=contributing}, \textrm{cycle relations} \ \ref{defn=potential} \}
\end{equation}
In order to write formula \ref{thm:main} for the data $G'\mathcal{K}$ from a long component projection we need to determine 
the contributions from right oriented extrema. For this let $v\in V^*$ define
\begin{equation}\label{eq=rot}
\textrm{rot}(v,\mathbf{j})=\mathfrak{r}_{\tau }(v)j_v+\mathfrak{r}_{\sigma }(v)i_v
\end{equation}
and
$$\textrm{rot}(\mathbf{j}):=\sum_{v\in V^*}\textrm{rot}(v,\mathbf{j})$$

\begin{rem}
In order to see that $\textrm{rot}(\mathbf{j})$ for a state determined by $\mathbf{j}$ and Definition \eqref{eq=rot} determines the contributions resulting from right oriented extrema, consider first a right oriented extremum on an overcrossing arc $\sigma (v)$, and assume that this is the only right oriented extremum on the overcrossing arc. The part arc color of the part arc containing the maximum is given by 
$i_v+\sum j_w$, where the sum runs through those $w\in \tau^{-1}(\sigma (v))$ preceding the maximum.  Recall from above that 
a choice of element in $\tau^{-1}(v)$ determines a unique part arc if $\tau^{-1}(v)\neq \emptyset $, and otherwise there is only one part-arc on the given overcrossing arc. The contribution $i_v=\mathfrak{r}_{\sigma }(v)i_v$, while the $j_w$ are contributions $\mathfrak{r}_{\tau }(w)j_w$ for the $w$ along the overcrossing arc $\sigma (v)$ preceding the maximum. It is not hard to see that 
the result for the rotation of a state persists, even though contributions for individual extrema will show up differently in the formula for $\textrm{rot}(\mathbf{j})$ given above. For example, if right oriented maximum is followed by an right oriented minimum, then 
the part arc color before the maximum will count with factor $+1$ while the part arc color before the minimum will contribute with factor $-1$ in the definition by ribbon functor \cite{L}. In the formula for $\textrm{rot}(\mathbf{j})$ given above this will be counted correctly because in this case, if no further right oriented extrema follow on the part arc considered, the rotation numbers of part arcs preceding the maximum will be $0$.
In the braid projection case this cannot occur but it is possible to have overcrossing arcs rotating through several maxima, in which case the above analysis similarly proceeds. 
\end{rem}

Finally the last item to be considered is the excess of a state in terms of the potential. We will use the following definition for $v\in V$:
\begin{equation}\label{eq=tilde}
\tilde{j}_v:=\sum_{\substack{w\in \tau^{-1}(\tau (v))\\ w<v}}j_w+i_{\sigma^{-1}(\tau(v))}
\end{equation}
Then 
\begin{equation}\label{eq=excess}
\textrm{exc}(\mathbf{j}):=\sum_{v\in V^*}\varepsilon (v)i_v\tilde{j}_v
\end{equation}

Let
\begin{equation}
\delta (\mathbf{j}):=-(\textrm{exc}(\mathbf{j})+\textrm{rot}(\mathbf{j}).
\end{equation}

We can now restate Theorem \ref{thm:main} using the graph sextuple $G'\mathcal{K}$. Note that this description does not use any ordering of the vertex set $V$. For a long component projection $\textrm{rot}(\mathcal{K})$ let the \textit{rotation number} of $\mathcal{K}$ be the sum of the rotation numbers of all components. This is a sum over the rotation numbers of all part arcs and thus is also the sum of the rotation numbers of all circle components resulting from smoothing all crossings. Let 
\begin{equation}\label{eq=delta'}
\delta '(\mathcal{K},n):=-\frac{n^2}{4}\omega (\mathcal{K})+\frac{n}{2}\textrm{rot}(\mathcal{K}).
\end{equation}
Note that our rotation numbers are positive in clockwise direction. 

\begin{thm}\label{thm:GL2}
Let $\mathcal{K}$ be a long component projection of an oriented link $K$ with data quintuple $G'\mathcal{K}$ and $\delta '(\mathcal{K},n)$ defined in equation \eqref{eq=delta'}. Let
$$\beta_n(\mathbf{j})=\prod_{v\in V'}\beta_n(\mathbf{j},v)$$
for $\mathbf{j}\in \mathcal{F}_n^*G\mathcal{K}$
as defined in \eqref{eq=contributing states}, and for each $v\in V^*$
\begin{equation}\label{eq=vertex contribution}
\beta_n(\mathbf{j},v)=t^{n\varepsilon (v)i_v}{i_v+j_v \choose i_v}_{t^{-\varepsilon (v)}} \{n-\tilde{j}_v\}_{j_v,t^{\varepsilon (v)}}
\end{equation}
with $\tilde{j}_v$ defined in \eqref{eq=tilde}.
Then 
\begin{equation}\label{eqn=GL2}
J_n'(K)=t^{\delta '(\mathcal{K},n)}\sum_{\mathbf{j}\in \mathcal{F}_n^*G\mathcal{K}}t^{\delta (\mathbf{j})}\beta_n(\mathbf{j})
\end{equation}
\end{thm}

\begin{rem} (a) The appearence of $\frac{n}{2}\textrm{rot}(K)$ is a consequence of evaluation maps at extrema, see e.\ g.\ the definition of the endomorphism $K$ in \cite{G}, p.\ 1264 or Theorem 5.3.2 in \cite{CP}.

\noindent (b) Note that the representation of the colored Jones polynomial in \ref{thm:GL2} references only the vertex set of $G\mathcal{K}$, which is identified with the crossings of $\mathcal{K}$ and defined by overcrossing arcs of $P\mathcal{K}$. 
We did reduce the reference to edge colorings to the computation of blue edge colors from the potential. 

\noindent (c) Note that the Potts model on a graph is defined from colorings of vertices. Even though the 

\noindent (d) We have purposefully not introduced any ordering of vertices in this section. Given an ordering and basing of the components of a projection there is a natural \textit{cyclic ordering} by running along components starting from the basepoint. Also, if the projection is a closed braid projection in the plane starting from a given braid word then the crossings are naturally ordered by each braid generator corresponding to a crossing. If $\mathcal{O}$ is any ordering of $V$ then there is a natural $\mathcal{K}$-\textit{twisting} of this order given by the orderings $\mathfrak{o}_v$ for each $v\in V$. If $\mathcal{O}$ orders the vertices as 
$\{v_1,\ldots ,v_{\mathfrak{c}}\}$ then the ordering is given by the elements in $\tau^{-1}(v_i)$ ordered by $\mathfrak{o}_{v_i}$ as above preceding the elements in $\tau^{-1}(v_{i+1})$ ordered by $\mathfrak{o}_{v_{i+1}}$. This could be related to Definition 5.1 in \cite{GL}.

\noindent (e) If $\mathcal{K}$ is an alternating link projection then, in the non-reduced description,  $\tau^{-1}(v)$ consists of a single element and the orderings $\mathfrak{o}_i$ are trivial. In particular in the alternating case the $\tilde{j}_v$ reduces to $\tilde{j}_v=i_{\sigma^{-1}(\tau (v))}$ so that the Pochhammer symbol and also the excess are easier to calculate. 
\end{rem}

\section{Working with cyclic ordering}

In the following we assume that our projection $\mathcal{K}$ is based and we consider the cyclic ordering with respect to the basing. 
This defines a bijection $V\leftrightarrow \underline{\mathfrak{c}-1}$ with $\mathfrak{c}$ denoting the number of crossings. The base vertex on the first component is labelled $0$ by the order because it is deleted in reduced graphs. We will usually identify vertices with their order labeling. The parametrization of the set of states is now given by integer valued vectors
$\mathbf{j}=(j_0,\ldots ,j_{\mathfrak{c}-1})$, possibly with components removed, or just set equal to $0$, in the reduced case. The vector $\mathbf{i}=(i_0,\ldots ,i_{\mathfrak{c}-1})$ is similarly defined. 

With respect to this ordering the undercrossing map has a simple description $\sigma $ is now given by 
$$\sigma (k)=k+1\ \textrm{mod}\ (\mathfrak{c}-1),$$
where we will as usual think of $a \textrm{mod}\ b$ for integers $a,b$ and $b>0$ as an element in $\underline{b}=\{0,1,\ldots ,b-1\}$. 
Theorem \ref{thm:GL2} can be rewritten, now writing $i_k, j_k,\varepsilon (k)$ etc.\ for $i_v, j_v,\varepsilon (v)$ if $v$ is the vertex $k$. The jump map $\tau $ now becomes a map 
$$\tau : \underline{\mathfrak{c}}\rightarrow \underline{\mathfrak{c}}.$$
Recall that this a bijection if and only if $\mathcal{K}$ is an alternating projection. 

In the case of a link of $\mu >1$ components some more notation will be necessary.
Suppose that the number of overcrossing arcs on the $\ell $-th component is $\mathfrak{c}(\ell )$. Then define
$$\mathfrak{d}(\ell ):=\mathfrak{c}(1)+\cdots +\mathfrak{c}(\ell )$$
for $\ell =1,\ldots ,\mu$, and let $\mathfrak{c}(0):=0$. 
Then we enumerate the overcrossing arcs on the $\ell $-th components in the following way
$$\mathfrak{d}(\ell -1),\mathfrak{d}(\ell -1) +1,\ldots ,\mathfrak{d}(\ell )-1$$

In the reduced case we have $i_0=0$ and  $i_{\mathfrak{c}(\ell -1)}=\beta_{\ell }$ is determined by the potential for $\ell =2,\ldots ,\mu $. As explained in section 7 we can determine the colors $i_k$ from the vector 
$$\mathbf{j}=(j_1,\ldots ,j_{\mathfrak{c}-1})$$
Note that he vector is partitioned by components into $\mathbf{j}=(\mathbf{j}(1),\ldots ,\mathbf{j}(\mu ))$ with
$$\mathbf{j}(1)=(j_1,\ldots ,j_{\mathfrak{d}(1)-1})$$
and for $\ell \geq 2$,
\begin{equation}\label{eq=comp}
\mathbf{j}(\ell )=(j_{\mathfrak{d}(\ell -1)},\ldots ,j_{\mathfrak{d}(\ell )-1})
\end{equation}

We always have to take into account the $\mu -1$ cycle relations. In order to determine these we have to first 
find for each $1\leq \ell \leq \mu$ those numbers $k$ for which the corresponding overcrossing arc ends at a crossing of distinct components, say for a fixed $\ell $ these are $k_1,\ldots ,k_{\rho }$. Finally we have to determine for which overcrossing arcs $k$, $\tau^{-1}(k)$ is an overcrossing arc of a component distinct from the $\ell $-th, say those are
$m_1,\ldots ,m_{\rho'}$.
Then the cycle relation for the $\ell $-th component is
$$\sum_{r=1}^{\rho '}j_{\tau^{-1}(m_r)}-\sum_{r=1}^{\rho }j_{k_r}=0$$
Note that $\rho +\rho '$ is always an even number.
The system of equations to find $\mathbf{i}$ from $\mathbf{j}$, given $i_{\mathfrak{c}(k-1)}$ for $k=2,\ldots ,\mu $ is given by the flow equations, equalities to determine the contributing states then are found as explained in section 7. 

We will now briefly show how to find the set of contributing states for a $2$-component alternating braid projection $\mathcal{K}=\hat{\beta }$.
We assume that the first appearance of the braid generator $\sigma _1$ will be corresponding to vertex $0$ and comes as $\sigma_1^{-1}$. This will allow us to work with the reduced graph. 
Let $\tau^{-1}(0)=:\kappa $. 

We have crossings $0,1,\ldots ,\mathfrak{c}(1)-1$ along the first component, and \\
$\mathfrak{c}(1),\ldots ,\mathfrak{c}-1$ with 
$\mathfrak{c}=\mathfrak{c}(1)+\mathfrak{c}(2)$ along the second component, the number of crossings of $\mathcal{K}$. 
The potential is given by $(\mathbf{j},\beta_2)$ and we assume that $\beta_2=i_{\mathfrak{c}(1)}$.
Fiigure 21 shows the arc-graph for a $2$-component alternating link.

\vskip 0.1in

\includegraphics[width=0.9\textwidth]{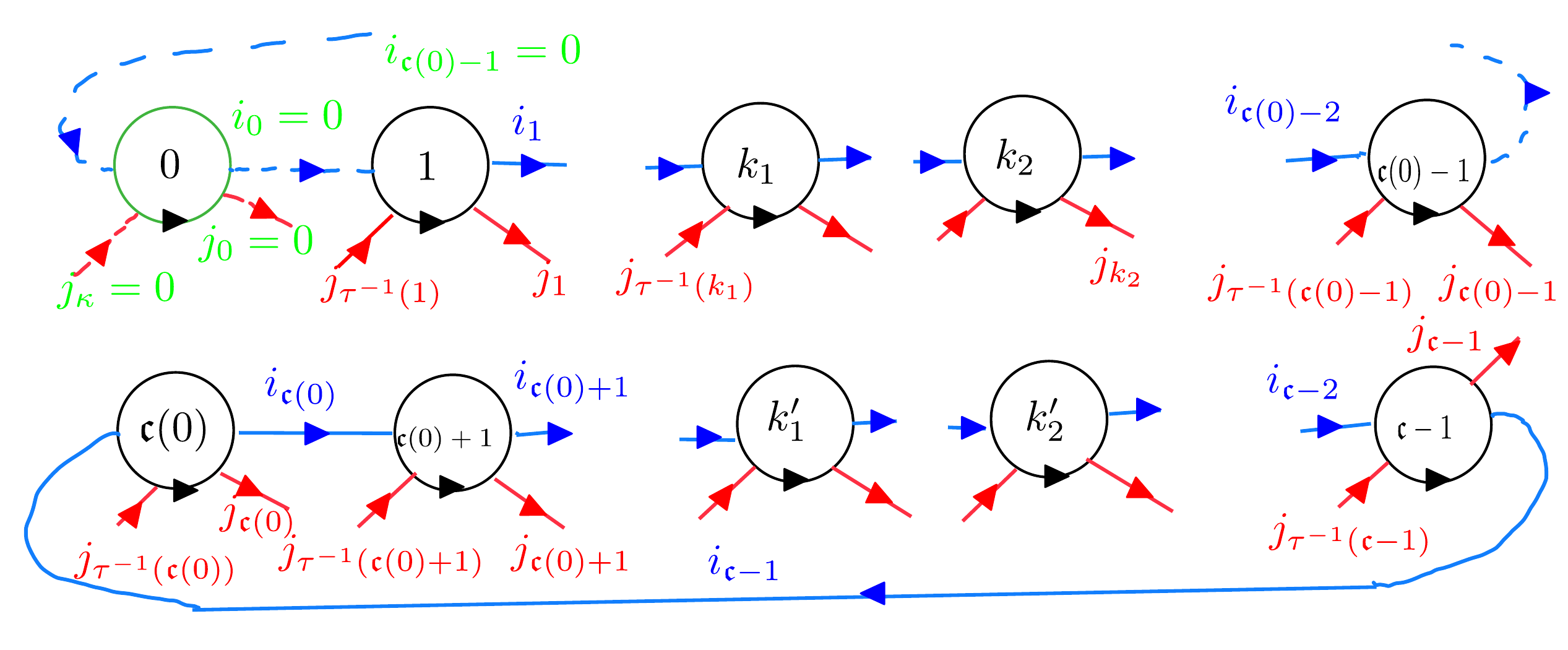}

\centerline{\textit{Figure 21}}

\vskip 0.1in

We have $i_0=i_{\mathfrak{c}(1)-1}=j_0=j_{\kappa }=0$ by the choice of braid. Furthermore we have flow equations \eqref{eq=local 1} for $k=0,\ldots \mathfrak{c}(1)-1$, and as before from the vertices of the first component, see equation (\eqref{eq=cyclic}):
$$\sum_{k=1}^{\mathfrak{c}(1)-1}j_{\tau^{-1}(k-1)}-j_k=0.$$

In the alternating $2$-component case this is not obvious, even though  $\tau $ is a bijection of $\underline{\mathfrak{c}}$. This is because when the first component overcrosses the second component at vertex $k$ then $\tau^{-1}(k)\in \mathfrak{c}(1)+\underline{\mathfrak{c}(2)}=\{\mathfrak{c}(1),\ldots ,\mathfrak{c}-1\}$, while at an undercrossing of the second component $\tau (k) \in \mathfrak{c}(1)+\underline{\mathfrak{c}(2)}$. But the above relation holds anyway because of the cycle condition. Thus $i_0=0$ still implies 
\begin{equation}\label{eq=cycle i}
i_k=\sum_{\ell =1}^kj_{\tau^{-1}(\ell )}-j_{\ell }
\end{equation}
for $k=1,\ldots \mathfrak{c}(1)-2$, and $i_{\mathfrak{c}(1)-1}=0$.
Finally the flow equations and base value at the vertices corresponding to the second component will give $i_{\mathfrak{c}(1)}=\beta_2$ and 
$$i_k=i_{\mathfrak{c}(1)}+\sum_{\ell =\mathfrak{c}(1)+1}^kj_{\tau^{-1}(\ell )}-j_{\ell }$$
for $k=\mathfrak{c}(1)+1,\ldots ,\mathfrak{c}-1$. Again, the cycle equation is necessary to show that the solution of the system of flow equations is consistent. Note that flow equation at the first vertex of the second component holds 
$$\beta_2=i_{\mathfrak{c}(1)}=i_{\mathfrak{c}-1}+j_{\tau^{-1}(\mathfrak{c}(0)}-j_{\mathfrak{c}(0)}$$
because of the cycle relation, compare the arguments following Definition \ref{defn=potential}, where we explain how a state is defined from the potential. Note that with $j_0=j_{\kappa }=0$ we reduce the red colors at arc-graph vertices corresponding to the first component to $\mathfrak{c}(1)-2$. Now we have a choice of $\beta_2$ and a cycle relation for the red colors at arc-graph vertices of the second component. Usually this will allow to eliminate of those red colors from the cycle relation and end up with a total of $\mathfrak{c}-2$ colorings determining a contributing state, of course with restrictions given by
$\mathbf{0}\leq \mathbf{i}+\mathbf{j}\leq \mathbf{n}$ as before. But if the two components split and have no crossings this is not the case, in this case we need $\mathfrak{c}-1$ colors. Also if one of the components is completely crossing over the other component the red colors for those crossings all have to be $0$. In this case we only need $\mathfrak{c}'-1=\mathfrak{c}'(1)
+\mathfrak{c}'(2)-1$ colors to determine a contributing state, where $\mathfrak{c}'(\ell )$ is the number of self-crossings of the $\ell $-th component for $\ell =1,2$. Note that the requirement $\mathbf{0}\leq \mathbf{j}\leq \mathbf{n}$ in combination with cycle conditions will often reduce the \textit{dimension} of the space of red colors giving contributing states. But e.\ g.\  the case of the Hopf link $\sigma_1^{-2}$ shows that the answer in general can be $\mathfrak{c}-1$. 

\begin{exmp} Let $\beta =\sigma_1^{-6}$ and $\mathfrak{s}=2$. Then $\mathcal{K}=\beta ^{\wedge }$ is a $2$-component link projection for a link with linking number $-3$. Note that $\omega (\mathcal{K})=-6$. The set of contributing states then can be described as the set of $(j_1,j_3,j_5,\beta )$ such that $0\leq j_3,\beta \leq n$ and 
$$\textrm{max}(0,j_3-\beta)\leq j_1\leq j_5\leq \textrm{min}(n+j_1-j_3,n-\beta )$$
It is interesting to count the contributing states for fixed $\beta =0,1,\ldots ,n$. For example when $\beta =n$ it follows that 
$j_1=j_5=0$ but $j_3$ can take any values $0,1,\ldots n$, so we have $n+1$ states. On the other hand when $\beta =0$ then the inequalities reduce to
$0\leq j_3\leq j_1\leq j_5\leq n$. So there are $\frac{1}{6}(n+1)(n+2)(n+3)$ states for $\beta =0$.
(Here we use the usual way to enumerate the overcrossing arcs). 

\begin{center}
\includegraphics[width=1.0\textwidth]{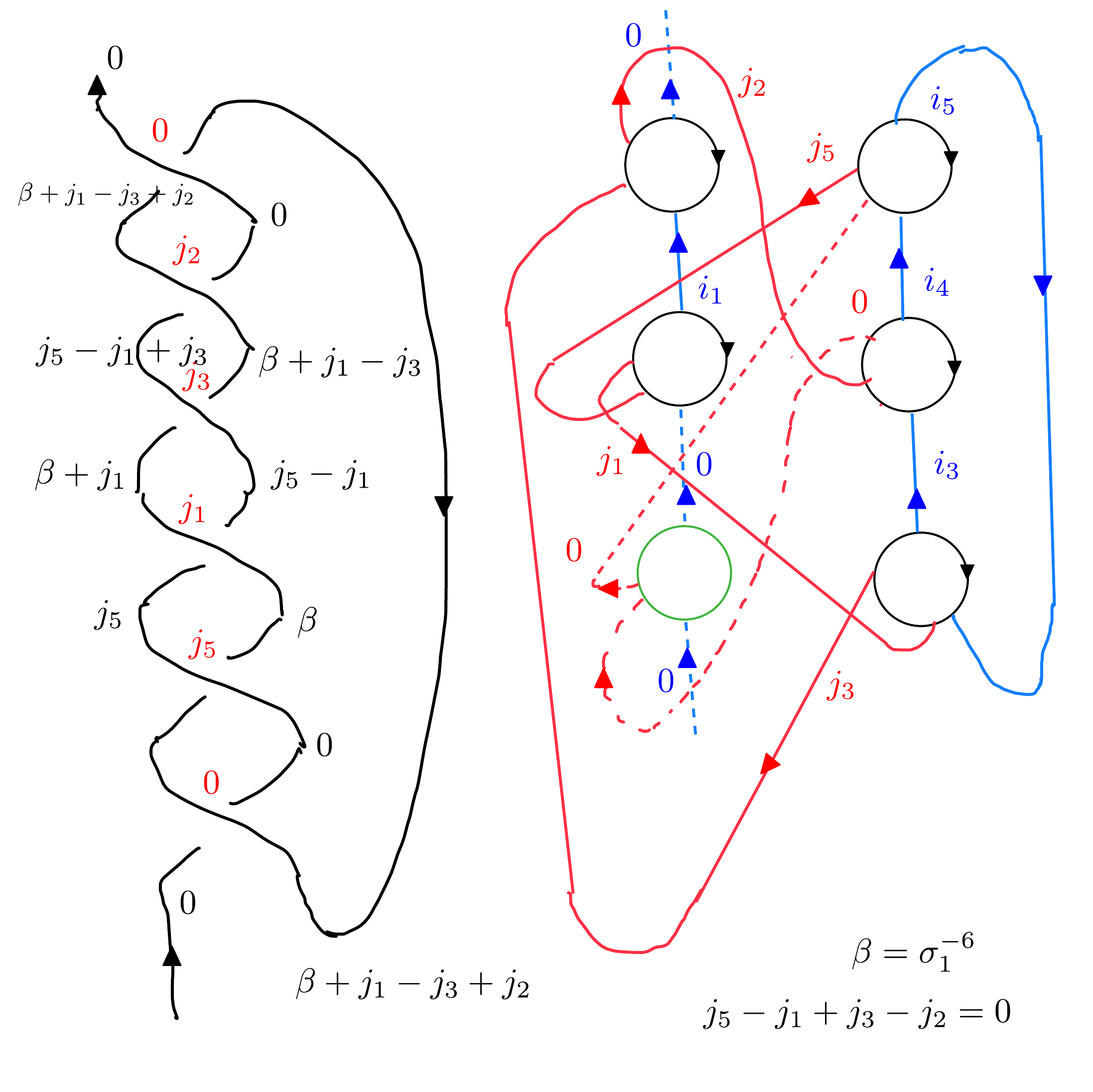}
\end{center}

\vskip -0.1in

\centerline{\textit{Figure 22}}

\end{exmp}

\section{Weaving links}

Recall our notation for $a$ a positive integer, $\underline{a}:=\{0,1,\ldots ,a-1\}$.

In this section we consider only braids with zero writhe and three strands. In this case the framed and unframed colored Jones polynomial coincide. The colored Jones polynomial will take values in $\mathbb{Z}[t^{\pm 1/2}]$, in fact in $\mathbb{Z}[t^{\pm 1}]$ in the case of knots or $n$ even. For $\mathfrak{s}=3$ we have:
$$\delta (\mathcal{K},n)=\delta '(\mathcal{K},n)=n.$$
and vertex contributions for states as in equation \eqref{eq=vertex contribution}.

We will consider those \textit{weaving links}, which are defined by the following alternating $3$-braid closures:
$$\left((\sigma_1^{-1}\sigma_2)^m\right)^{\wedge}=\overline{W(3,m)}.$$
For $m\equiv 1\ \textrm{mod}\ 3$ and $m\equiv 2\ \textrm{mod} \ 3$ the closure of the braid is a knot, for $m\equiv 0 \ \textrm{mod} \ 3$ the closure is a $3$-component link. Note that all links $K=W(3,m)$ are amphichiral and therefore $J_n(K)=J_n(\overline{K})$.
Note that for $m=3\ell +1$ the induced cycle permutation is $(132)$ respectively for $m=3\ell +2$ it is $(123)$.
In the case of weaving links the cyclic order of overcrossing arcs by $k\in \underline{2m-1}$ will be such that
$\varepsilon (k)=(-1)^{k+1}$. 

In the following we will focus on an explicit description of the jump maps $\tau $, rotation and excess,  because these are
the main ingredients to apply \ref{thm:GL2}. The braid projections are alternating and thus $\tau $ is a bijection. 
For small $\ell $ we can explicitly describe the recursion to determine $\mathbf{i}$ from $\mathbf{j}$, and the inequalities for the sets of contributing states. But in general, it seems better to just refer to the general formula \eqref{eq=contributing}.

\noindent \textbf{(a)} Let $\ell \geq 1$ and 
$$K=\left((\sigma_1^{-1}\sigma_2)^{3\ell +1}\right)^{\wedge}=\overline{W(3,3\ell +1)}$$
 with corresponding long component projection $\mathcal{K}$ with $6\ell +2$ crossings. 
 For $\ell =0$ we have $\sigma_1^{-1}\sigma_2$, which is braid representation of the unknot. 
 It is a nice exercise to check that the formula gives the polynomial $1$ in this case. 
 
We can work with the reduced arc graph and thus describe a contributing state (in our case $\tau^{-1}(0)=2\ell $). 
$$\mathbf{j}=(j_1,j_2,\ldots ,j_{6\ell +1}), j_{2\ell }=0$$
and
$$\mathbf{i}=(i_1,i_2,\ldots ,i_{6\ell })$$
We have for $k\in \underline{6\ell +1}$:
$$\tau (k)=k+4\ell +2\ \textrm{mod}\ (6\ell +2),$$
the modulus is defined in $\underline{6\ell +1}$, or $\tau^{-1}(k)=k+2\ell \ \textrm{mod} (6\ell +2)$.
Recall that we set $i_0=j_0=0$. The contributing states are described by a vector of length $6\ell $ with entries in $\underline{n}$ subject to 
$6\ell -1$ \textrm{double inequalities} for contributing states \eqref{eq=contributing}. But there is only one inequality for vertex $2\ell $, 
$i_{2\ell -1}+j_{4\ell }\leq n$.

\vskip 0.1in

\includegraphics[width=0.5\textwidth]{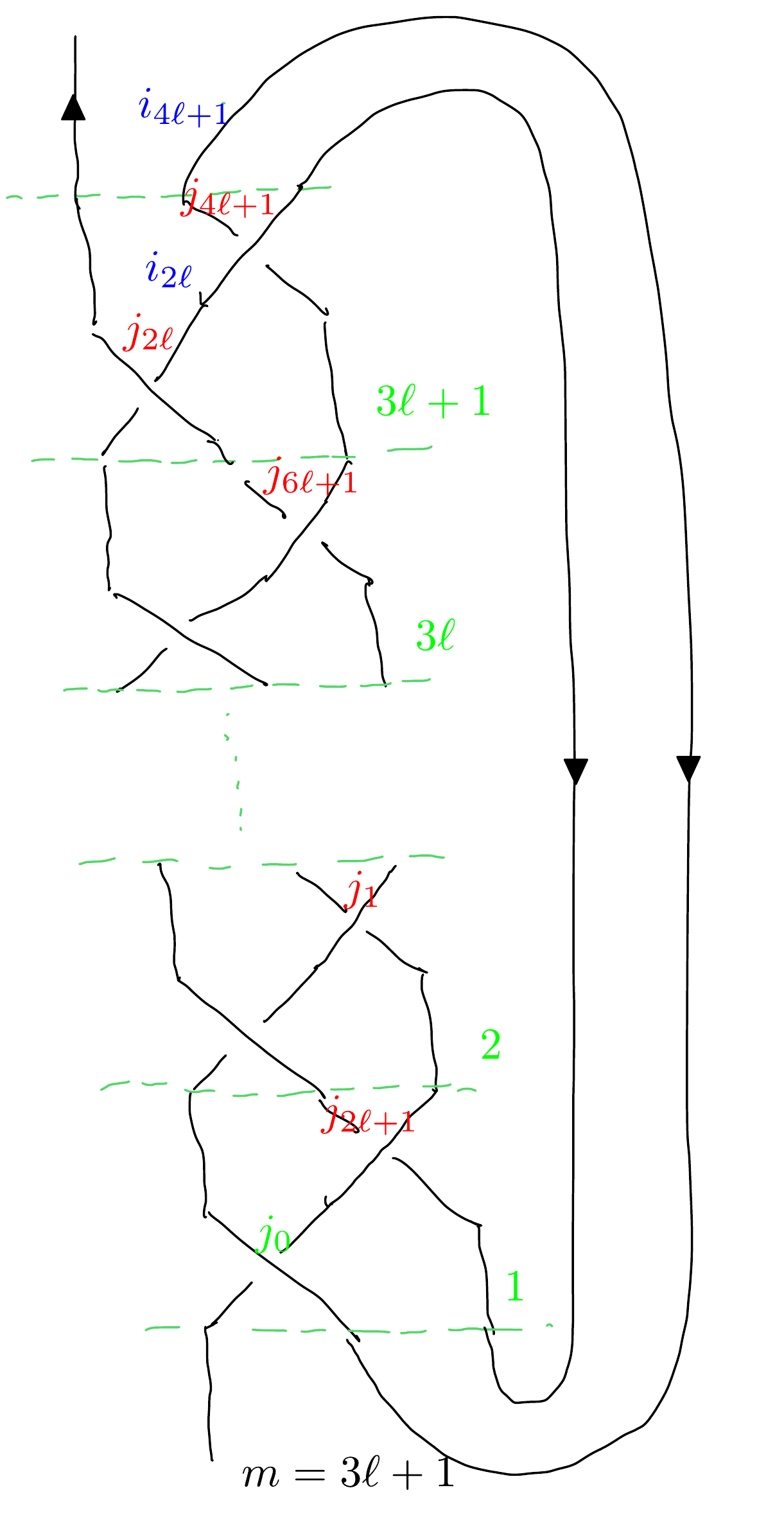}
\includegraphics[width=0.5\textwidth]{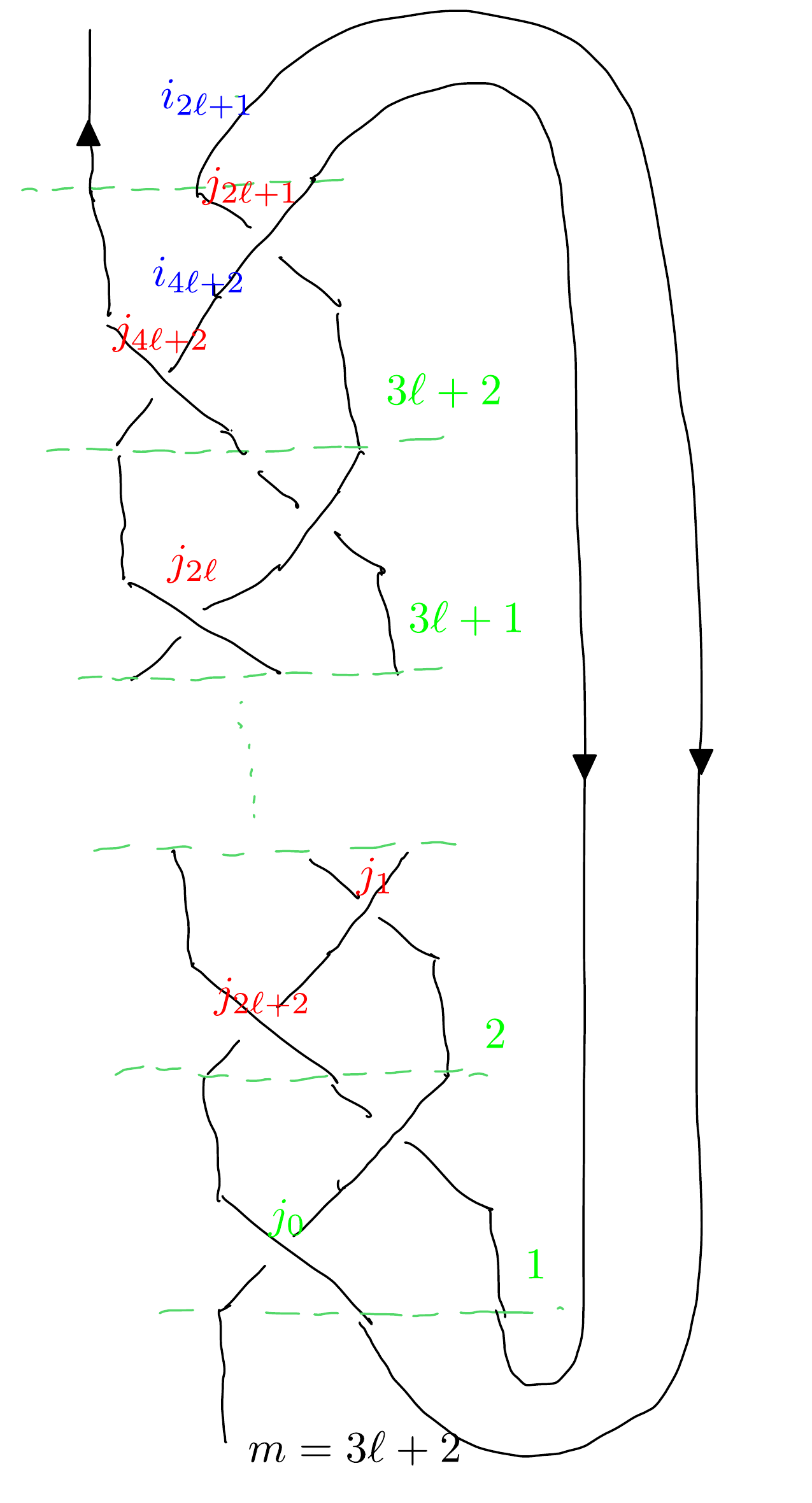}

\vskip -0.1in

\centerline{\textit{Figure 23}}

\vskip 0.1in

We have $\textrm{rot}(f)=(i_{2\ell }+j_{4\ell +1})+i_{4\ell +1}$ and 
$$\textrm{exc}(f)=-\sum_{k=1}^{6\ell +1}(-1)^{k+1}i_ki_{\tau (k)-1}.$$ 
See Figure 21 for the rotation and see the arc graph for the braid projection of $W(3,4)$ below in Figure 22.
Thus $\delta (\mathbf{j})=-(i_{2\ell }+j_{4\ell +1}+i_{4\ell +1})+\sum_{k=1}^{6\ell +1}(-1)^ki_ki_{\tau (k)-1}$.

\newpage

\begin{wrapfigure}{l}{7.5cm}
\includegraphics[width=7.5cm,height=11cm]{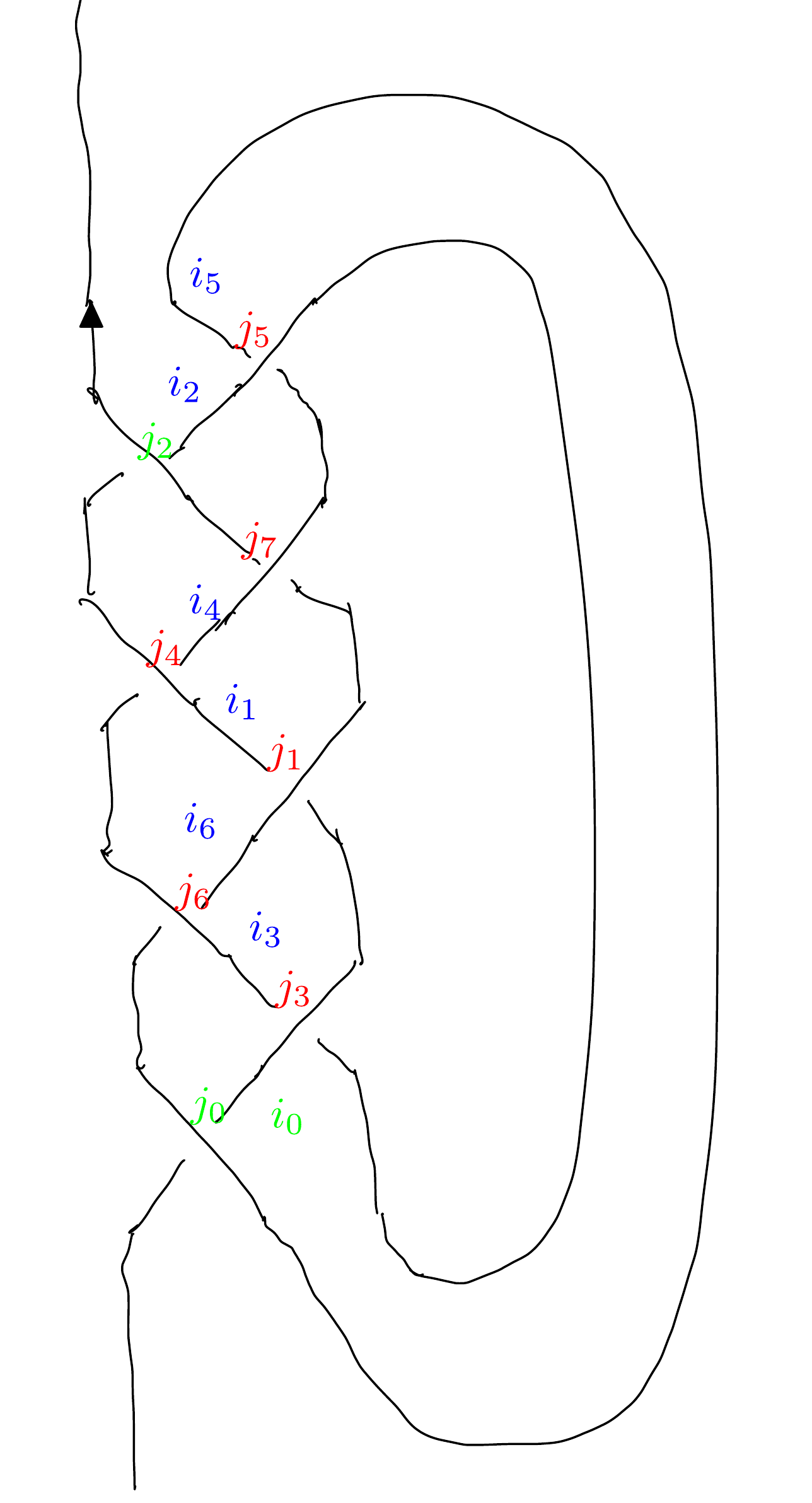}
\vskip -0.1in

\centerline{\textit{Figure 24}}

\vskip 0.1in
\end{wrapfigure}

For $\ell =1$ we are calculating the colored Jones polynomial of the weaving knot $W(3,4)$, which is the knot $8_{18}$. The figure shows the long braid representation of this knot and coloring according to \cite{G}.

Figure 21 shows the corresponding arc graph-graph according to \cite{GL} with colorings. Note that it is our specific choice of braid (starting with a negative braid generator and ending with a positive braid generator), which allows to discard the first overcrossing arc (which we label $0$). The system of equations restricting the values of the vector $\mathbf{j}$ by the non-negativity of the vector $\mathbf{i}$ and $n$ are now obvious from the figure. We would like to point out that working out the restriction to contributing states seems the most challenging piece of the computation. In the case of weaving links this can be reduced to periodicity according to the value of $\ell $. In general we will get a linear system of inequalities restricting the values to an intersection of half-spaces. For $\ell =1$ the system of inequalities can be reduced in the following way for the coefficients of the vector (see \eqref{eq=contributing}):
$$\mathbf{j}'=(j_1,j_3,j_4,j_5,j_6,j_7)\in (\underline{n+1})^6$$
\begin{align*}
j_1&\leq \textrm{min} \{j_3,j_4+j_5.j_5+j_6\}\\
\textrm{max} \{j_3+j_4,j_4+j_5,j_5+j_6\}&\leq n
\end{align*}
The main question seems to be to get a direct geometric understanding of those inequalities in terms of the geometry of the links. 

\begin{center}
\includegraphics[width=1.0\textwidth]{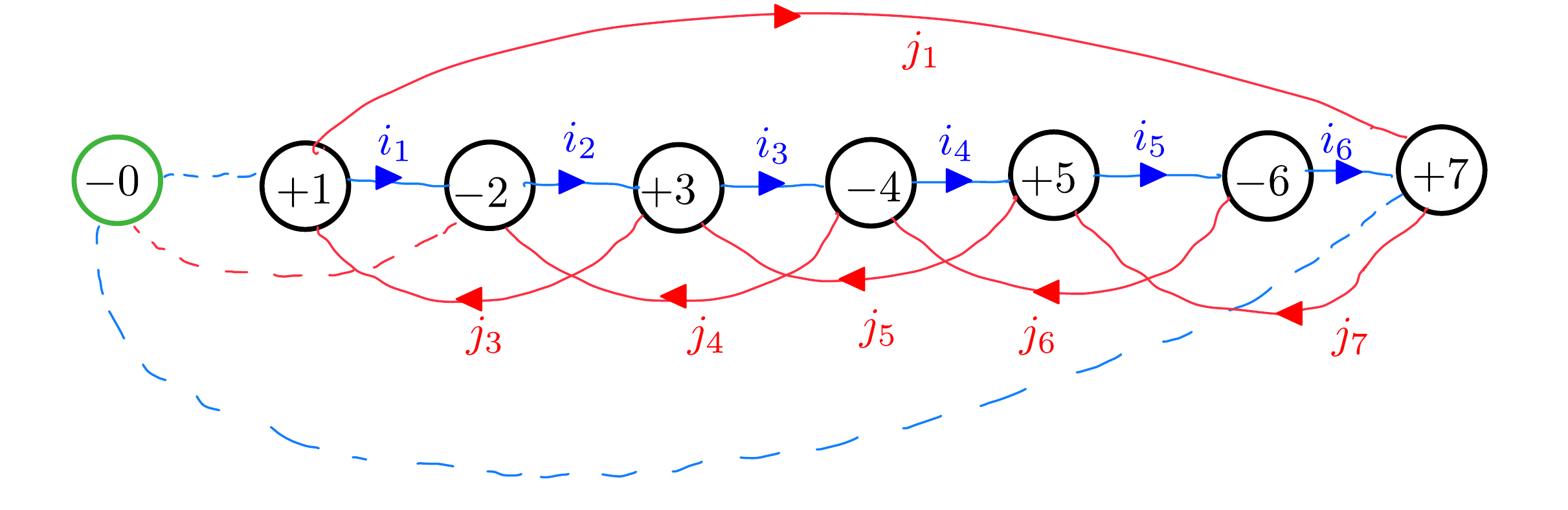}
\end{center}

\vskip -0.1in

\centerline{\textit{Figure 25}}

\vskip 0.1in

\noindent \textbf{(b)} Let $\ell \geq 0$ and consider the knots
$$K=\left((\sigma_1^{-1}\sigma_2)^{3\ell +2}\right)^{\wedge}=\overline{W(3,3\ell +2)}$$
 with corresponding braid projection $\mathcal{K}$. 
For $\ell =0$ we have $(\sigma_1^{-1}\sigma_2)^2$, which is the figure eight knot. The next knot in this series is $W(3,5)$, which is the knot $10_{123}$. The description of the set of contributing states is very similar to (a): 
$$\mathbf{j}=(j_1,\ldots ,j_{6\ell +3}) \ \textrm{with}\  j_{4\ell +2}=0$$
with $\mathbf{i}=(i_1,\ldots ,i_{6\ell +2})\in (\underline{n+1})^{6\ell +2}$ is determined by \eqref{eq=cycle i}
with 
$$\tau (k)=k+2\ell +2\ \textrm{mod}\ (6\ell +4)$$
and the inequalities determining the contributing states are given by \eqref{eq=contributing}.
Note that the braid word contains $6\ell +4$ braid generators. Taking into account $j_{4\ell +2 }=0$ we describe states by vectors 
with values in $\underline{n+1}$ of length $6\ell +2$.
We have 
$$\textrm{rot}(f)=(i_{4\ell +2}+j_{2\ell +1})+i_{2\ell +1}$$
and 
$$\textrm{exc}(f)=\sum_{k=1}^{6\ell +2}(-1)^ki_ki_{\tau (k)-1},$$ and thus 
$\delta (\mathbf{j})=-(i_{4\ell +2}+j_{2\ell +1}+i_{2\ell +1})+\sum_{k=1}^{6\ell +2}(-1)^ki_ki_{\tau (k)-1}$. 

\vskip 0.1in

\noindent \textbf{(c)} Now for $m=3\ell $ and $\ell \geq 1$ we consider the weaving links $W(3,3\ell )$ with $3$ components.
The corresponding braid word includes $6\ell $ braid generators. The resulting links are Brunnian, in fact $W(3,3)$ is the Borromean classical rings link. 

The non-negativity requirement for contributing states and basing of $0$ on the long strand immediately implies $j_0=j_{6\ell -2}=0$. The following two cycle relations have to hold for each coloring (note that the relation for the first component follows from the other two as we know):

\begin{align*} 
j_1+j_2+\cdots +j_{2\ell -1}=j_{2\ell}+j_{2\ell +1}+\cdots +j_{4\ell -2}+j_{4\ell -1}\\
j_{2\ell}+j_{2\ell +1}+\cdots +j_{4\ell -2}+j_{4\ell -1}=j_{4\ell }+j_{4\ell +1}+\cdots +j_{6\ell -1}
\end{align*}

We use the equations to eliminate $j_{2\ell }$ and $j_{4\ell }$ but will use them in subsequent inequalities:
$$j_{2\ell }=j_1+\cdots j_{2\ell -1}-(j_{2\ell +1}+\cdots +j_{4\ell -1})$$
respectively 
\begin{align*}
j_{4\ell }&=j_{2\ell }+\cdots j_{4\ell -1}-(j_{4\ell +1}+\cdots j_{6\ell -1})\\
&=j_1+\cdots j_{2\ell -1}-(j_{4l+1}+\cdots +j_{6\ell -3}, j_{6\ell -1})
\end{align*}

\newpage

\begin{wrapfigure}{l}{7.5cm}
\includegraphics[width=7.5cm,height=13cm]{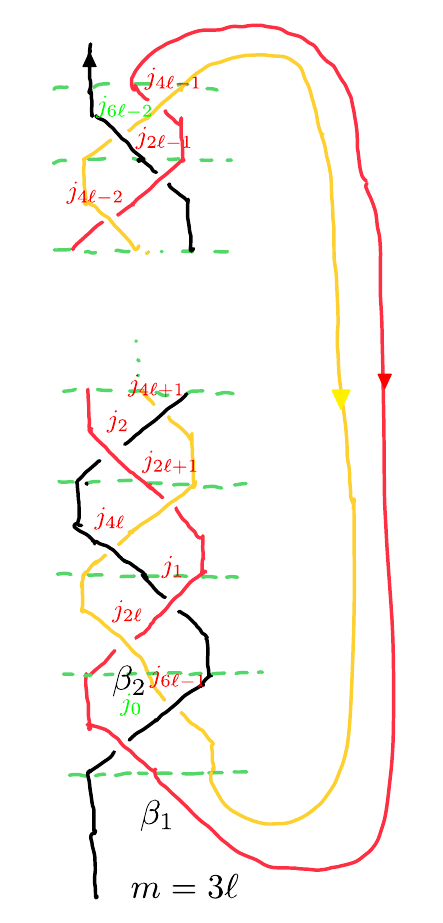}
\vskip -0.1in

\vskip 0.1in 

\centerline{\textit{Figure 26}}

\vskip 0.1in
\end{wrapfigure}

Just like in the knot cases we can determine the set of contributing states $\mathcal{F}_n^*(G\mathcal{K})$. Each contributing state is determined by a vector 
$$\mathbf{j}=(\mathbf{j}(1),\mathbf{j}(2),\mathbf{j}(3),\beta_1,\beta_2)$$
with components in $\underline{n+1}$. 
We $\mathfrak{c}(1)=\mathfrak{c}(2)=\mathfrak{c}(3)=2\ell $.
The vector $\mathbf{i}$ determined from the system of equations \eqref{eq=cycle i}. 

Using the base value on the first component and the two cycle relations above we can omit the first components in the vectors $\mathfrak{j}(\ell )$, 
and so have
\begin{align*}
\mathbf{j}(1)&=(j_1,\ldots ,j_{2\ell -1})\\
\mathbf{j}(2)&=(j_{2\ell +1},\ldots ,j_{4\ell -1})\\
\mathbf{j}(3)&=(j_{4\ell +1},\ldots ,j_{6\ell -3},j_{6\ell -1})
\end{align*}
We have $j_{6\ell -}=j_{\tau^{-1}(0)}=0$.

These vectors satisfy a system of inqualities see \eqref{eq=contributing} determining the contributing states,
in fact one system for each cycle.

Figure 27 shows the corresponding arc-graph.

 \ \newline
\ \newline
\ \newline

\begin{center}
\includegraphics[width=1.0\textwidth]{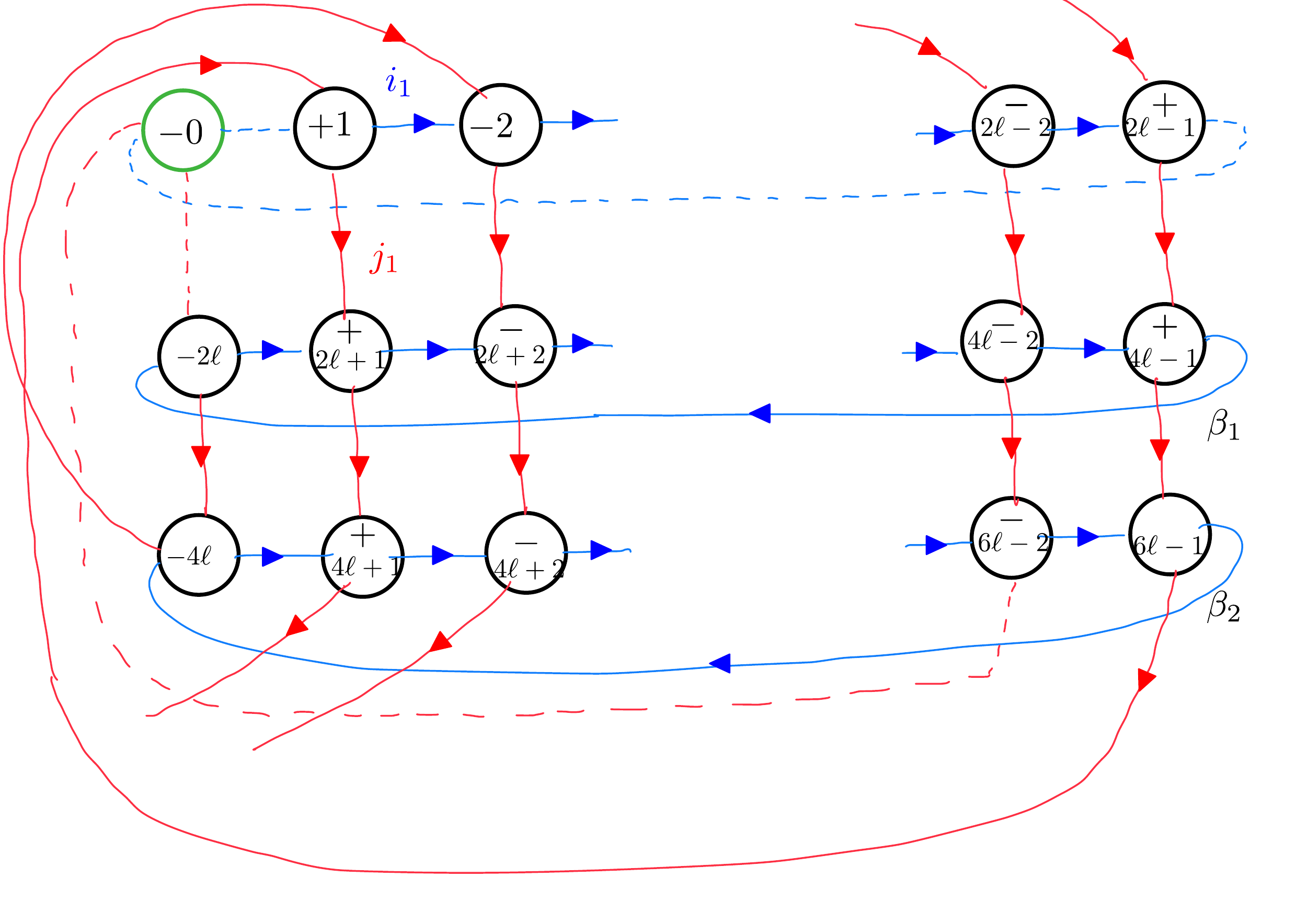}
\end{center}

\vskip -0.1in

\centerline{\textit{Figure 27}}

\vskip 0.1in

The jump map is given by:
$$
\tau (k)=\begin{cases}
k+2\ell \ \textrm{mod}\ 6\ell & \textrm{if} \ k=1,\ldots, 4\ell -1\\
k+2\ell +2\ \textrm{mod}\ 6\ell & \ \textrm{if} \ k=4\ell ,\ldots , 6\ell -1
\end{cases}
$$

\begin{rem}
The classical Jones polynomials $J(W(3,m))$ have also been recently computed by Mishra and Staffeldt in \cite{MS} using recursion in the Hecke algebra.
\end{rem}

\section{Vertex ordering by braids}

In the previous section we used symmetry and did appeal to the geometry of the projections in order to find the formulas for the jump map (and thereby colored Jones polynomial). 

\begin{defn} A \textit{braid word} for an $m$-braid is a \textit{reduced word}:
$$\beta =\sigma_{\eta_1}^{\epsilon_1}\sigma_{\eta_2}^{\epsilon_2}\cdots \sigma_{\eta_{\mathfrak{b} }}^{\epsilon_{\mathfrak{b}}}$$
representing an element in the free group generated generated by the braid generators 
 in the braid generators $\sigma_1,\ldots ,\sigma_{m-1}$.
\end{defn}

A braid word from left to right corresponds to a projection with generators placed from bottom to top.  
The length of the braid word $|\beta |:=|\epsilon_1|+\ldots +|\epsilon_{\mathfrak{b}}|$is well-defined, and
The number of crossings of the corresponding projection is $\mathfrak{c}=|\beta |$.  

In order to work with the reduced arc graph we assume that the first appearance of $\sigma_1$ in the braid word is with negative exponent. We will also assume that the projection has no overcrossing circles. 
Each braid word defines a natural ordering on the set of crossings, which we identify with numbers $0,1,\ldots ,\ldots ,\mathfrak{c}-1$ enumerated from the bottom. In the link case this set decomposes $V=V_1\sqcup \ldots \sqcup V_{\mu}$ according to the $\mu $ components of the link, where we assume $V_{\ell }$ consists of those crossings with the undercrossing arc belonging to the $\ell $-th component for $\ell =1,\ldots ,\mu $. We call the odering on $V$ induced from a braid the \textit{braid ordering}.

Using a cycle decomposition of $\sigma $, just as in section 3, a potential defines (see Definition \ref{defn=potential}) a coloring of $G\mathcal{K}$. In fact, if $\sigma $ defines a $\mathfrak{c}_{\ell }$-cycle on $V_{\ell }$ then we can use 
$$f(e_v^b)=f(e_{\sigma ^{-1}(v)}^b)+f(e_{\tau ^{-1}(v)}^r)-f(e_v^r).$$
and we will reach $f(e_{v_0}^b)$ for some $v_0\in V_{\ell }$ with value prescibed by the potential in finitely many recursion steps. Note that besides rotation and excess the state-sum in Theorem \ref{thm:main} is completely determined by $\sigma ,\tau $, with $\sigma $ a permutation and $\tau $ a mapping of $V\leftrightarrow \underline{\mathfrak{c}}$ 

Note that any two permutations with equivalent length cycle decompositions are conjugate. Thus, by finding $\rho $ such that
$\rho \sigma \rho^{-1}=\sigma_+$ where $\sigma_+$ is the standard shift map on each $V_{\ell }$ module the cycle length, 
we can return to set-up of section 8. If we define $\tau_+:=\rho \tau \rho^{-1}$ then this will be the jump map according to the standard shift map $\sigma_+$. The reason to use $\sigma ,\tau $ defined from the ordering induced from the braid is that 
it is much easier to determine those maps directly from the given braid word than it is to determine $\tau_+$.

We will finish by discussing the case of $3$-braid words:
$$\beta =\sigma_1^{\eta _1}\sigma_2^{\delta_1}\cdots \sigma_1^{\eta_r}\sigma_2^{\delta_r}$$
with $\eta_k, \delta_{\ell}\in \mathbb{Z}\setminus \{0\}$ for $k=1,\ldots r, \ell =1,\ldots r-1$. First note that it is easy to determine the number of components using the map from the free group on $\sigma_1,\sigma_2$ to the group $G=\langle \sigma_1,\sigma_2:\sigma_1^2=\sigma_2^2=1\rangle $. This will result in \textit{contraction elements} $\sigma_1\sigma_2\sigma_1\cdots \in G$ respectively $\sigma_2\sigma_1\sigma_2\cdots \in G$. If the length in $G$ is $\equiv 0(\textrm{mod}\ 6)$ then 
$K=\beta ^{\wedge}$ has three components. If the length is $\equiv 2,4(\textrm{mod}\ 6)$ then $K$ is a knot, and if the length is $\equiv 1,3,5(\textrm{mod}\ 6)$ then $K$ is a $2$-component link. This follows by mapping $\sigma_1$ to $(1,2)$ and $\sigma_2$ to $(2,3)$. In fact, the contraction to $G$ gives more precise information in telling what permutation underlies the braid. Note that using the obvious homomorphism into $\Sigma_3=\langle \sigma_1,\sigma_2:\sigma_1^2=\sigma_2^2=(\sigma_1\sigma_2)^3=(\sigma_2\sigma_1)^3=1\rangle $ it follows that 
$\sigma_1\sigma_2, (\sigma_2\sigma_1)^2$ map to $(132)$, $\sigma_2\sigma_1, (\sigma_1\sigma_2)^2$ map to 
$(123)$, $\sigma_1, (\sigma_2\sigma_1)^2\sigma_2$ maps to $(1,2)$, $\sigma_1\sigma_2\sigma_1, \sigma_2\sigma_1\sigma_2$ map to $(1,3)$ and finally $\sigma_2, (\sigma_1\sigma_2)^2\sigma_1$ map to $(2,3)$.

\vskip 0.2in

In order to find the \textbf{crossing map} $\sigma $ we have to distinguish four cases, depending on whether we calculate $\sigma (i)$ for $i$ corresponding to an overcrossing arc ending at elementary braids $\sigma_i^{\pm 1}$, $i=1,2$. 
In the following four cases we will in each case show all possible segments of the braid word following the element, with the last element in each case giving the next vertex at which the overcrossing arc starting at $i$ ends. Thus $|\sigma (i)-i|$ in each case is just the length of the segment, which is $\leq \mathfrak{c}-1$. Of course $\sigma (i)<i$ is possible if the overcrossing arc starting at crossing $i$ contains a braid closure arc. 

\vskip 0.1in

\noindent \textbf{Case 1:} $\sigma_1$: \textbf{(i)} is followed $\sigma_1$: $\sigma_1^2$, $\sigma_1\sigma_2^{-1}$, $\sigma_1\sigma_2^2$, $\sigma_1\sigma_2\sigma_1^{\alpha }\sigma_2$, $\sigma_1\sigma_2\sigma_1^{\alpha }\sigma_2^{-2}$, 
$\sigma_1\sigma_2\sigma_1^{\alpha }\sigma_2^{-1}\sigma_1$, 
$\sigma_1\sigma_2\sigma_1^{\alpha }\sigma_2^{-1}\sigma_1^{-2}$, $\sigma_1\sigma_2\sigma_1^{\alpha }\sigma_2^{-1}\sigma_1^{-1}\sigma_2^{\beta }\sigma_1^{-1}$, \\ $\sigma_1\sigma_2\sigma_1^{\alpha }\sigma_2^{-1}\sigma_1^{-1}\sigma_2^{\beta }\sigma_1^2$,  $\sigma_1[\sigma_2\sigma_1^{\alpha }\sigma_2^{-1}\sigma_1^{-1}\sigma_2^{\beta }\sigma_1]\ldots $, at this point the possibilities prescribed before are repeated 
by inserting the bracketed segment between $\sigma_1$ starting with the second possibility. Here $\alpha ,\beta \in \mathbb{Z}\setminus \{0\}$ can be any numbers, and in the repetition of the segment $\alpha ,\beta $ can be replaced 
by different numbers $\alpha '\beta '$ etc.\ depending on the braid sequence. The sequence will only stop if there is a 
break in the multiplication by conjugates, or we return to the starting braid $\sigma_1$ in the case that the overcrossing arc we describe starts and ends at the same vertex. The structure emerges from \textit{meandering} of the overcrossing arc between string $1$ and $3$ of the braid. A product $\sigma_1\sigma_2$ not followed by an additional $\sigma_1 $ will run the overcrossing arc from string $1$ to string 3 where it bypasses $\sigma_1^{\alpha }$. Then similarly the $\sigma_2^{-1}\sigma_1^{-1}$, if not followed by an additional $\sigma_1^{-1}$, will run the overcrossing arc from string 1 to string $3$. Note that if $\sigma_1$ is in braid position $i\in \{0,1,\ldots ,\mathfrak{c}-1\}$ then $\sigma (i)\textrm{mod} (\mathfrak{c})\in \{0,1,\ldots ,\mathfrak{c}-1\}$ is given by adding to $i$ the numbers
$2,2,3,3+\alpha,4+\alpha,5+\alpha,5+\alpha +\beta, 6+\alpha+\beta, 5+\alpha+\beta +\textrm{corresponding}$. 

\noindent \textbf{(ii)} if not followed by $\sigma_1$: $\sigma_2^{\beta }$, $\sigma_2^{\beta }\sigma_1^{-1}$, $\sigma_2^{\beta }\sigma_1^2$, 
$\sigma_2^{\beta }\sigma_1\ldots $. The same meandering occurs starting the point $\ldots $. 

\noindent \textbf{Case 2:} $\sigma_1^{-1}$: \textbf{(i)} if followed by $\sigma_1^{-1}$: this will be $\sigma_1^{-1}$ followed by the sequences in \textbf{Case 1} \textbf{(ii)}. 

\noindent \textbf{(ii)} if not followed by $\sigma_1^{-1}$: This will be exactly the sequence in \textbf{Case 1} \textbf{(i)}, starting with the second segment in the sequence and $\sigma_1$ deleted from each element. 

\vskip 0.1in

\noindent \textbf{Case 3:} $\sigma_2$: \textbf{(i)} if followed by $\sigma_2$: this corresponds to \textbf{Case 2} \textbf{(i)}, with interchanging $\sigma_1^{\pm 1}$ and $\sigma_2^{\mp 1}$ at all instances.

\noindent \textbf{(ii)} if not followed by $\sigma_2$: this corresponds to \textbf{Case 2} \textbf{(ii)}, with interchanging $\sigma_1^{\pm 1}$ by $\sigma_2^{\mp 1}$ at all instances.

\vskip 0.1in

\noindent \textbf{Case 4:} $\sigma_2^{-1}$: \textbf{(i)} if followed by $\sigma_2^{-1}$: this corresponds to \textbf{Case 1} \textbf{(i)} with $\sigma_1^{\pm 1}$ replaced by $\sigma_2^{\mp 1}$ at all instances.  

\vskip 0.2in

The second important ingredient in order to apply \ref{thm:main} is the \textbf{jump map} $\tau $. The point of using the vertex ordering defined from the braid is mainly that it is easier to determine $\tau $ from the braid word. The hard work is contained in the discussion of $\sigma $ above. Again we assume that the braids in the following cases are at position $i$. We consider the case of $\sigma_1$ in position $i$, the other cases are similar. If $\sigma_1$ is also in position $i+1$ then $\tau (i)=i+1$. If $\sigma_1$ is in position $i-1$ then $\tau (i)=\sigma (i-1)$ where the $\sigma_1$ in position $i-1$ is in \textbf{Case 1} \textbf{(i)} above. If the $\sigma_1$ in position $i$ is \textit{isolated} so precedded and followed by $\sigma_1^{\pm 1}$ then we replace $\sigma_1$ by $\sigma_1^{-1}$ to get a braid $\beta '$, and find $\sigma (i)$ for $\beta '$ using the above process for $\sigma $. If the answer would be $i$ again then the arc at $i$ would be on an overcrossing circle, which we excluded. 

\begin{rem} (a) Of course the description of the map $\sigma (i)$ can also be used to label the overcrossing arcs in sequence 
as discussed in section 7. We call this cycle order. In fact, if overcrossing arc $k$ ends at the crossing labelled $i$ according to the braid then overcrossing arc $k+1$ in cycle order ends at vertex $\sigma (i)$. 

\noindent (b) The vertex ordering defined from the braid presentation and the corresponding two maps $\sigma ,\tau $ are closely related to the noncommutative state models discussed in \cite{A}, \cite{GL}, \cite{H} and \cite{B}. 
\end{rem}

\begin{exmp} We will conclude with an example: $\beta =\sigma_1^{-3}\sigma_2\sigma_1\sigma_2^2\sigma_1^{-1}$. The contracted sequence is $\sigma_1\sigma_2$ and thus the closure is a knot and the braid permutation is  $(132)$, $\mathfrak{c}=9$ and so the vertex set is identified with $\{0,1,\ldots ,7\}$. We have vectors $\mathbf{j},\mathbf{i}$ as before, the difference now is just that the component number corresponds to the ordering induced from the braid. 

\begin{center}
\includegraphics[width=1.0\textwidth]{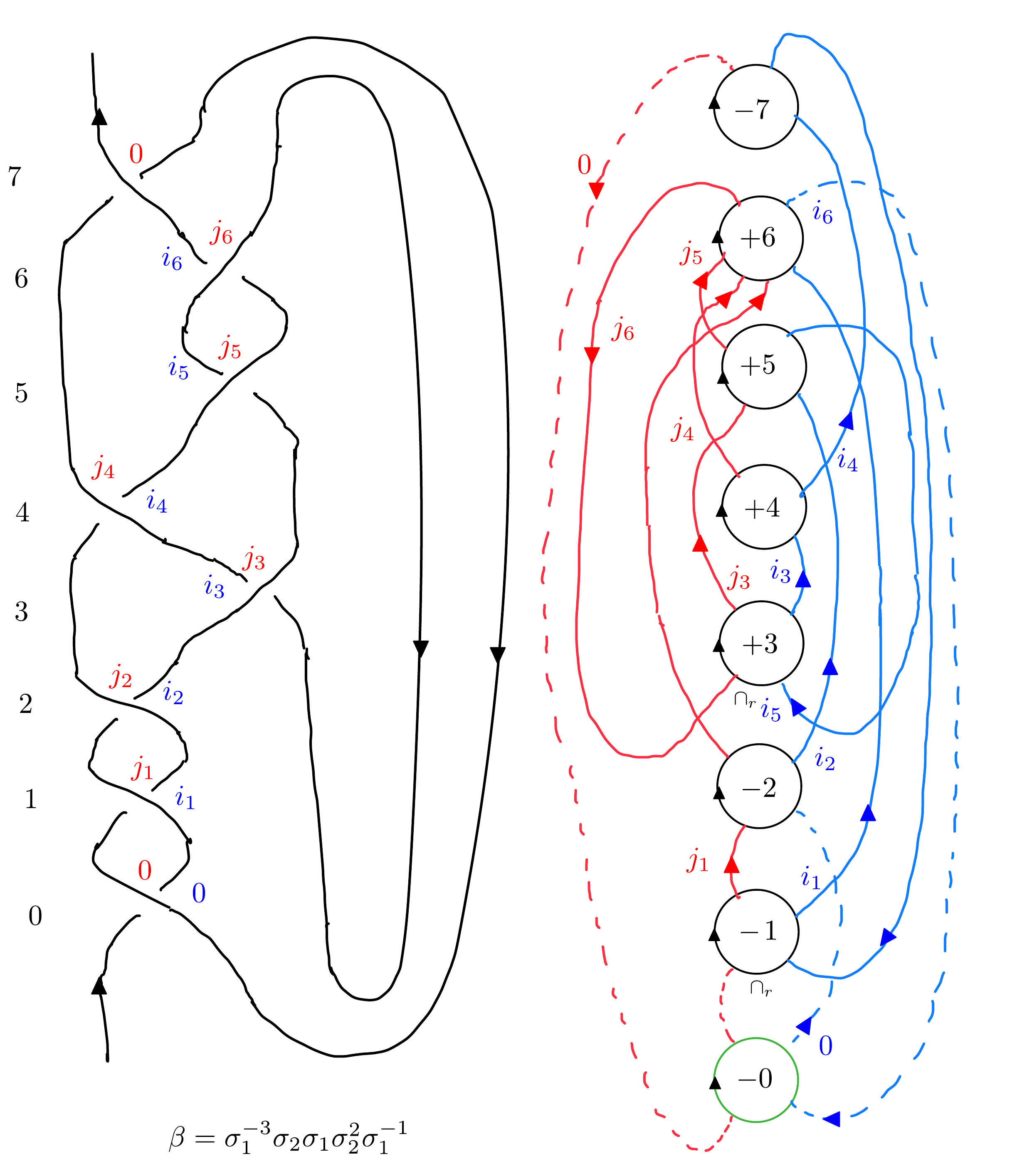}
\end{center}

\vskip -0.1in

\centerline{\textit{Figure 28}}

\vskip 0.1in 

\newpage

The following table is easily read off the part arc-graph. 

\begin{center}
\begin{tabular}{|l|l|l|}
\hline
number & $\sigma $ & $\tau $ \\
\hline
0 & 2 & 1\\
1 & 6 & 2\\
2 & 5 & 6\\
3 & 4 & 5\\
4 & 7 & 6\\
5 & 3 & 6\\
6 & 0 & 3\\
7 & 1 & 0\\
\hline
\end{tabular}
\end{center}

The computation of $\mathbf{i}=(i_1,i_2,i_3,i_4,i_5,i_6,i_7)$ from the vector of jumps $\mathbf{j}=(j_1,j_2,j_3,j_4,j_5,j_6,j_7)$ is recursively in the order 
$(0,2,5,3,4,7,1)$ since $\sigma (0)=2,\sigma(2)=5,\ldots $. Let $\mathbf{j}'=(j_1,j_2,j_3,j_4,j_5,j_6)$ and $\mathbf{i}':=(i_1,i_2,i_3,i_4,i_5,i_7)$ (we know that $i_0=j_0=i_6=j_7=0$ in the reduced description). 
We show explicitly the computation from the flow conditions on the arc graph:

\vskip 0.1in

$\sigma (0)=2, \tau (1)=2$, so $i_2+j_2=i_0+j_1$, or $i_2=j_1-j_2$

$\sigma (2)=5, \tau (3)=5$, so $i_5+j_5=i_2+j_3$, or $i_5=j_1-j_2+j_3-j_5$ 

$\sigma (5)=3, \tau (6)=3$, so $i_3+j_3=i_5+j_6$, or $i_3=j_1-j_2-j_5+j_6$

$\sigma (3)=4$, $\tau^{-1}(4)=\emptyset $, so $i_4+j_4=i_3$, or $i_4=j_1-j_2-j_5-j_4+j_6$

$\sigma (4)=7$, $\tau^{-1}(7)=\emptyset $, so $i_7+j_7=i_4$, or $i_7=j_1-j_2-j_5-j_4+j_6$

$\sigma (7)=1, \tau (0)=1$, so $i_1+j_1=i_7+j_0$, or $i_1=j_6-j_2-j_5-j_4$

\vskip -0.3in

\begin{align*}
\sigma (1)=6, \tau (2)=\tau (4)=\tau (5)=6, & \ \textrm{so} \ i_6+j_6=i_1+j_2+j_4+j_5, \\
& \textrm{or} \ i_6=0
\end{align*}

Note that $i_6=0$ is required by the color of the top left string to be $0$. So we get 
\begin{align*}
\mathbf{i}'=&(-j_2-j_4-j_5+j_6,j_1-j_2,j_1-j_2-j_5+j_6,j_1-j_2-j_4-j_5+j_6,\\
\ & j_1-j_2+j_3-j_5,j_1-j_2-j_4-j_5+j_6)
\end{align*}

The rotation of a state is given by the part-arc colors at the maxima as usual, and thus are:
$$\textrm{rot}(\mathbf{j}')=(i_3+j_4)+(i_5+j_6)=2(j_1-j_2-j_5+j_6)+j_3+j_4$$

\end{exmp}

\vskip 0.2in

We conclude with a list of questions we developed during this work. 

\vskip 0.1in

\section{Questions and projects}

\noindent \textbf{(1)} The colored Jones polynomials of alternating links has been studied in \cite{G1} using so called \textit{centered} states. It should be interesting to translate this work into the state model considered here.  

\noindent \textbf{(2)} Use the flow lemma to derive state-sum formulas for $\textrm{sl}(n,\mathbb{C})$-invariants, and for colored Jones polynomials of links with components colored by distinct representations. 

\noindent \textbf{(3)} Translate the state models for general $\mathfrak{g}$-quantum invariants based on the $R$-matrix into state models on arc graphs.

\noindent \textbf{(4)} Show that the noncommutative state models in \cite{GL} and \cite{HL} are equivalent, and dominate the state models discussed in this article. Extend the noncommutative state models to links. 

\noindent \textbf{(5)} Prove Armond's theorem on colored Jones ploynomials of positive knots using the state models discussed in this article. 

\noindent \textbf{(6)} How can the the arc graph be augmented so that the Khovanov homology of the colored links can be computed from it. This seems to be an open question even for $n=1$.

\noindent \textbf{(7)} Study how the geometry of return loops in the part arc graph is reflected in the state-sum for the colored Jones polynomial. 

\noindent \textbf{(8)} In \cite{KM} we reduce the set of states to compute the colored Jones polynomial of $W(3,4)$ to $10$ variables, see also \cite{A}, while the above calculation only requires $6$ variables. We conjecture that the noncommutative state model of \cite{HL} and \cite{A} dominates the state model of \cite{GL}.

\noindent \textbf{(8)} Work out the rules to find $\sigma ,\tau $ from the braid word for the $m$-string braid words. 

\begin{small}

\end{small}
\end{document}